\documentclass[12pt, reqno]{amsart}
\theoremstyle{plain}
\newtheorem{theorem}{Theorem}[section]
\newtheorem*{theorem*}{Theorem}
\newtheorem{lemma}[theorem]{Lemma}
\newtheorem{corollary}[theorem]{Corollary}
\newtheorem{proposition}[theorem]{Proposition}

\theoremstyle{definition}
\newtheorem{definition}[theorem]{Definition}
\newtheorem{example}[theorem]{Example}
\newtheorem*{example*}{Example}
\newtheorem{remark}[theorem]{Remark}

\numberwithin{equation}{section}

\usepackage{array}%for the extra space in the tables

\usepackage{multirow}
\usepackage{array,booktabs}
\usepackage{todonotes}
\usepackage[margin=1in]{geometry}
\usepackage{verbatim}
\usepackage{amsmath}
\usepackage{amsfonts}
\usepackage{amssymb}
\usepackage{amsthm}
\usepackage{enumerate}
\usepackage{amsrefs}
\usepackage{amscd}
\usepackage{amsrefs}
\usepackage{tensor}
\usepackage{epic}
\usepackage{mathtools}
\usepackage{tikz}
\usetikzlibrary{chains}
\usepackage{graphicx}
\usepackage[font=small,skip=0pt]{caption}
\usepackage{longtable}
\usepackage{makecell}
\usepackage{tensor}
\usepackage[super]{nth}
\usepackage{bm}
\usepackage{amsaddr}
\usepackage{xcolor}

\usepackage{dynkin-diagrams}
\usepackage{graphicx}
\usepackage[font=small,skip=0pt]{caption}

\pgfkeys{/Dynkin diagram,edge length=1.5cm,root radius=.08cm,
indefinite edge/.style={
draw=black,fill=white,thin,densely dashed}}

\newcolumntype{P}[1]{>{\centering\arraybackslash}m{#1}}

\def\beq#1#2\eeq{%
            \begin{equation}%
            \label{#1}%
                #2%
            \end{equation}%
        }

\begin{document}

\author{G. Antoniou, M.V. Feigin, I.A.B. Strachan}
\author{}

\address{ School of Mathematics and Statistics, University of Glasgow, UK}

\email{
 \parbox{\linewidth}{ \vspace{0.444cm} g.antoniou.1@research.gla.ac.uk, misha.feigin@glasgow.ac.uk, \\ ian.strachan@glasgow.ac.uk
 }
 }

\title{\bf{The Saito determinant for Coxeter discriminant strata}}

\begin{abstract}
Let $W$ be a finite Coxeter group and $V$ its reflection representation. The orbit space $\mathcal{M}_W= V/W$ %there exists a pencil of flat metrics. Its key element is the 
has the remarkable Saito flat metric defined as a Lie derivative of the $W$-invariant bilinear form $g$. We find determinant of the Saito metric restricted to an arbitrary Coxeter discriminant stratum in $\mathcal{M}_W$. It is shown that this determinant  is proportional to a product of linear factors in the flat coordinates of the form $g$ on the stratum. We also find multiplicities of these factors   in terms of Coxeter geometry of the stratum. 

This result may be interpreted as a generalisation to discriminant strata of the Coxeter factorisation formula for the Jacobian of the group $W$. As another interpretation, we find determinant of the operator of multiplication by the Euler vector field in the natural Frobenius structure on the strata.
\end{abstract}

\maketitle

\section{Introduction}
Frobenius manifolds lie at the intersection of many areas of mathematics, from singularity theory and algebraic geometry, mirror symmetry, integrable systems, to mathematical physics, via Topological Quantum Field Theories \cite{dub}. A remarkable class of solutions to the underlying Witten--Dijkgraaf--Verlinde--Verlinde (or WDVV) equations is given by polynomial prepotentials. Following a comment by Arnold, Dubrovin showed that polynomial solutions were directly related to finite Coxeter groups, or more precisely, that the orbit space $\mathcal{M}_W$ of a finite Coxeter group $W$ carries the structure of a semi-simple, polynomial, Frobenius manifold \cite{dub}. Furthermore, Dubrovin conjectured and Hertling proved that all semi-simple polynomial solutions arise via this construction (under an extra assumption, see \cite{CH}). 
%The proof that the orbit space carries the structure of a Frobenius manifold was proved by Dubrovin, and the converse was later proved by Hertling. 
A key part of the Frobenius manifold structure on the orbit space $\mathcal{M}_W$ is the Saito metric  and its flat coordinates  \cite{saito1, saito}. 
We begin by recalling the definitions of these objects.

%Frobenius manifolds introduced by Dubrovin in the early $90$s have attracted a lot of attention merely due to their interplay between many areas of mathematics, such as integrable systems, singularity theory, algebraic geometry as well as theoretical physics \cite{dub}. The orbit space $\mathcal{M}_W$ of a finite Coxeter $W$ group acting on its reflection representation constitutes one of the most interesting classes of examples of Frobenius manifolds as the corresponding prepotential is polynomial \cite{dub}. A key element of such Frobenius manifolds is the Saito flat metric $\eta$ \cite{saito1, saito}. Under certain assumptions the orbit spaces $\mathcal{M}_W$ are the only Frobenius manifolds whose prepotential is polynomial when expressed in the flat coordinates of the metric \cite{CH}.

Let $V= \mathbb{C}^n$ with the standard metric $g$ given by $g(e_i, e_j)= (e_i, e_j)= \delta_{ij}$, where $e_i$, $i=1, \dots, n,$ is the standard basis in $V$. Let $x^i$, $i=1, \dots, n$ be the corresponding orthonormal coordinates in $V$. An irreducible finite Coxeter group $W$ of rank $n$ acts in $V$ by orthogonal transformations such that $V$ is the complexified reflection representation of $W$. Then the orbit space $\mathcal{M}_W = V / W$.

Let $\mathcal{R} \subset V$ be the Coxeter root system associated with the group $W$ and let $\mathcal{R}_+ \subset \mathcal{R}$ denote a positive subsystem \cite{cox}. For any $\alpha \in \mathcal{R}$ we define the corresponding \textit{mirror} to be the hyperplane $\Pi_\alpha= \{ x \in V | (\alpha, x)=0\}$. The following notion of discriminant is important. 

\begin{definition}\cite{dub}.
The subset $\Sigma \subset \mathcal{M}_W$ called \textit{discriminant} is defined as the image of the union of the mirrors of $W$ under the quotient map 
\begin{align}\label{quotientmap}
\pi: V \rightarrow \mathcal{M}_W.
\end{align}
Equivalently, $\Sigma$ consists of the irregular orbits of $W$.
\end{definition}

Let us now recall the key definition of the Saito metric. Let $S(V^*)^W$ be the ring of $W$-invariant polynomials on $V$. By Chevalley's theorem, $\mathcal{M}_W$ is an affine variety with coordinate ring 
\begin{align*}
S(V^*)^W = \mathbb{C}[x^1, \dots, x^n]^W =\mathbb{C}[x]^W= \mathbb{C}[p^1, \dots, p^n],
\end{align*}
where $p^i$, $i=1, \dots, n$ are homogeneous polynomials of positive degrees  $\operatorname{deg} p^i = d_i$, ($i=1, \dots, n$). These polynomials are coordinates on the orbit space $\mathcal{M}_W$. We fix the following ordering for the degrees $d_i$, $d_n > d_{n-1} \geq \dots \geq d_{2} > d_1 =2; \quad  d_n=h$, where $h$ is the Coxeter number of the group $W$. 

The quotient map on the complement to the discriminant,
\begin{align*}
\pi_{\Sigma}: V \setminus \cup_{\alpha \in \mathcal{R}_+} \Pi_\alpha \rightarrow \mathcal{M}_W \setminus \Sigma
\end{align*}
is a local diffeomorphism. Then, the linear coordinates $x^{i}$, ($i=1, \dots, n$) on $V$ can be viewed as local coordinates on $\mathcal{M}_W \setminus \Sigma$. These are flat coordinates for the flat metric $g$ which is also defined on $\mathcal{M}_W \setminus \Sigma$ due to its $W$-invariance. Let $g^{ij}$ be the corresponding contravariant metric.  Its matrix entries are given as Arnold's convolution of basic invariants \cite{Ar}.

Note that the vector field $\partial_{p^n}$ is well-defined up to proportionality. The contravariant \textit{Saito metric} is defined to be proportional to $ \partial_{p^n} g^{ij}(p)$. The remarkable property is the flatness of this metric \cite{saito1}. Thus there exist homogeneous basic invariants $$t^\alpha \in \mathbb{C}[x]^W, \quad \operatorname{deg}t^\alpha=d_\alpha, \quad \alpha=1, \dots, n, $$ such that
\begin{align*}
 \frac{\partial g^{\alpha \beta}(t)}{\partial t^n}= \delta^{\alpha + \beta,n+1}, \quad 1\leq \alpha, \beta \leq n, 
\end{align*}
where $\delta^{ij}=\delta_{ij}$ is the Kronecker symbol.
These polynomials are called \textit{Saito polynomials} or \textit{Saito flat coordinates}. They were specified explicitly in \cite{saito} for any $W$ except for $E_7$ and $E_8$ cases which were completed in \cite{DAB}, \cite{flat}. Let us fix the Saito metric in the basis $t^\alpha$ ($\alpha=1, \dots, n$) to be
\begin{align}\label{saitometricflatcoords}
\eta^{\alpha \beta}=\mathcal{L}_e g^{\alpha \beta}(t)= \delta^{\alpha + \beta, n+1},  \quad 1\leq \alpha, \beta \leq n,
\end{align}
where $\mathcal{L}_e=\frac{\partial}{\partial t^n}$ is the Lie derivative along the vector field $e = \frac{\partial}{\partial t^n}$ (see \cite{KS}, \cite{saito}).

Frobenius manifold $\mathcal{M}_W$ gives a family of Frobenius algebras defined on its tangent planes. The corresponding bilinear form is the Saito metric $\eta$. Vector field $e$ is the identity field of the Frobenius manifold \cite{dub}. There is also another almost dual Frobenius structure \cite{DubA}. The corresponding multiplication in the case of $\mathcal{M}_W$ is defined outside of the discriminant $\pi^{-1}_\Sigma(\mathcal{M}_W\setminus\Sigma)$ and the corresponding bilinear form is the standard metric $g$. The prepotential $\mathcal{F}$ of the almost dual structure can be given explicitly \cite{DubA}.

For a general Frobenius manifold the notion of a natural submanifold was introduced by one of the authors in \cite{ian}. Frobenius multiplication can be restricted to natural submanifolds but some of the properties such as flatness of the induced restricted metric may no longer hold. Key examples of natural submanifolds were expected to be discriminant strata in the orbit spaces $\mathcal{M}_W$ as well as caustics. The main aim of this paper is the investigation of Saito metric on the discriminant. Let us define these settings more precisely.  

Let us fix a collection $S$ of linearly independent roots $\beta_1, \dots, \beta_k \in \mathcal{R}$ and let $D=D_S= \cap_{j=1}^k \Pi_{ \beta_j}$. A \textit{discriminant stratum} in the orbit space $\mathcal{M}_W$ is defined to be the image of $D$ under the quotient map $ \pi$ given by \eqref{quotientmap}. Sometimes we will refer to the intersection of hyperplanes $D$ as a discriminant stratum as well, and likewise we will sometimes refer to the union of all hyperplanes $\Pi_\beta$, $\beta \in \mathcal{R}$ as discriminant.

It was shown in \cite{FV} that almost dual Frobenius multiplication has a natural restriction on discriminant strata $D$. However, the properties of Saito metric $\eta$ on the discriminant did not seem to be investigated until the present work. The main object of the present paper is the Saito metric restricted on discriminant strata. 

The covariant metric $\eta$ on $\mathcal{M}_W$ induces a metric on the stratum $\pi(D)$ which is naturally given as the restriction of $\eta$ to $\pi(D)$. We will denote this metric by $\eta_D$. The map $\pi$ is a local diffeomorphism on $D$ near generic point $x_0 \in D$  which allows us to lift metric $\eta_D$ to the linear space $D$. Likewise the metric $\eta$ can be lifted to $V$ near a generic point of $\mathcal{M}_W$.

It is convenient to introduce some notations related to hyperplane arrangements \cite{OT}. We will only be considering central arrangements where hyperplanes pass through the origin. For   such an arrangement  $\mathcal{A}$ in $V$ its defining polynomial is given by
\begin{align}\label{defpol}
{\mathcal I}(\mathcal{A}) = \prod_{H \in \mathcal{A}} \alpha_H,
\end{align}
where $\alpha_H \in V^*$ is such that the hyperplane $H= \{ x \in V : \alpha_H (x)=0\}$. 

Let now $\mathcal{A}$ be the Coxeter arrangement corresponding to $W$, that is the arrangement of mirrors $\Pi_\alpha$, $\alpha \in \mathcal{R}$.  Let ${\mathcal J}(p)$ be the Jacobian ${\mathcal J}(p^1, \dots, p^n)=\operatorname{det}\big(\partial p^i/\partial x^{j}\big)_{i, j =1}^n$, where $p^i$ ($i=1, \dots, n$) are basic invariants for $W$. Recall that ${\mathcal J}(p)$ factors into linear forms defining the mirrors of $W$ \cite{Coxeter}, \cite{cox}, that is there is a proportionality
\begin{equation*}
{\mathcal J}(p) \sim {\mathcal I}(\mathcal{A}).
\end{equation*} 

%Let us take basis $p^i=t^i$ of Saito polynomials and fix $J= J(t^1, \dots, t^n)$. 
Formula \eqref{saitometricflatcoords} implies the following statement. 

\begin{proposition}\label{detunrestricted}
Determinant of the Saito metric $\eta$ in $x$-coordinates $\operatorname{det}\eta(x)$ is proportional to $ {\mathcal I}(\mathcal{A})^2$.
%
%\begin{equation}\label{detsaito}
% \operatorname{det}\eta(x)=-J^2.
%\end{equation}
%
\end{proposition}

Note that determinant of the convolution of invariants metric $g^{ij}$ on the orbit space is also proportional to $ {\mathcal I}(\mathcal{A})^2$ (see e.g. \cite{saito}).

We are interested in the determinant of the restricted Saito metric $\eta_D$ on the discriminant strata $D$. We will show that $\operatorname{det}\eta_D$
is a product of linear forms which can be viewed as a generalization of Proposition  \ref{detunrestricted}.

Let us formulate the main result more precisely. Let $S$ be a collection of linearly independent roots as above and let $D=D_S$ be the corresponding discriminant stratum. Define root system $\mathcal{R}_D = \mathcal{R}\cap \langle S \rangle$ and consider its orthogonal decomposition into irreducible ones,
\begin{equation}\label{RDdecompo}
\mathcal{R}_D=\bigsqcup_{i=1}^l \mathcal{R}_D^{(i)},
\end{equation}
where $l \in \mathbb{N}$. Let us denote by $\mathcal{A}^D$ the Coxeter arrangement in $V$ corresponding to $\mathcal{R}_D$, that is 
$$
\mathcal{A}^D = \{ H\in {\mathcal A}| H \supset D\}.
$$

 Let $\mathcal{A}_D$ be the restriction of the arrangement $\mathcal{A}$ to $D$, that is 
\begin{equation}
\label{resarr}
\mathcal{A}_D=\{ D \cap H | H \in \mathcal{A} \, \, \text{and} \, \, D \not \subset H \}.
\end{equation}
 For $\gamma \in \mathcal{R}_D$ let $\mathcal{A}^D_{\Pi_\gamma}$ be the restriction of the arrangement $\mathcal{A}^D$ to the hyperplane $\Pi_{\gamma} \in \mathcal{A}^D$. Let $r_i$ and $h^{(i)}$ denote the rank and the Coxeter number of the root system $\mathcal{R}_D^{(i)}$ respectively. Let 
\beq{Ii}
{\mathcal I}_i= {\mathcal I}(\mathcal{A}_{\Pi_{\gamma_i}} \setminus \mathcal{A}^D_{\Pi_{\gamma_i}}),
\eeq 
where $\gamma_i \in \mathcal{R}_D^{(i)}$ is arbitrary, be defining polynomial \eqref{defpol} of the arrangement $\mathcal{A}_{\Pi_{\gamma_i}} \setminus \mathcal{A}^D_{\Pi_{\gamma_i}}$ . It is easy to see by considering group action that for the simply laced root system $\mathcal R_D^{(i)}$ the restriction $\left. {\mathcal I}_i\right|_D$ does not depend (up to proportionality) on the choice of $\gamma_i \in \mathcal{R}_D^{(i)}$. It follows from the subsequent considerations in Section \ref{degreessection} that the restriction $\left. {\mathcal I}_i\right|_D$ does not depend on the choice of $\gamma_i \in \mathcal{R}_D^{(i)}$ in general.

Let us consider the determinant of the metric $\eta_D$ in some coordinates on $D$ which are
linear combinations of the coordinates $x^i$, $i=1, \dots, n$. The main result of this work is the following theorem.

\begin{theorem}\label{MMtheorem}
The determinant of the restricted Saito metric $\eta_D$ is proportional to the polynomial
\begin{equation}\label{rfth12}
{\mathcal I}(\mathcal{A}\setminus \mathcal{A}^D)^m  \prod_{i=1}^l {\mathcal I}_i^{r_i}
 \end{equation} 
 on $D$, where $m=2-\sum_{i=1}^lr_i$.
\end{theorem}

Note that if the set $S$ is empty so that the stratum $D=V$ and $\eta=\eta_D$ then $l=0$, $m=2$ and we recover Proposition \ref{detunrestricted}. It will be clear from later considerations that function \eqref{rfth12} is indeed non-singular on $D$ even when $m<0$.

Theorem \ref{MMtheorem} implies that the determinant of $\eta_D$ is a product of linear forms on $D$. Its zeroes correspond to restrictions of roots from $\mathcal{R}$ and the multiplicities of zeroes are special as given by formula \eqref{rfth12}. 

Let us consider a general constant metric of the form $\widehat{\eta}= \sum_{i=1}^n dp^{i} dp^{n+1-i}$. A natural question is whether restriction of such metric to any stratum $D$ satisfies the factorisation property as in Theorem \ref{MMtheorem}. In other words, how special is the property of the metric $\eta_D$ firstly to have a factorised determinant and, secondly, to have prescribed multiplicities of linear factors? Let us consider the following example which explains that the Saito case is very special. 

\begin{example}\label{example1} Let $\mathcal{R}= D_3 = A_3$ and consider the following basic invariants:
\begin{align*}
p^1= \frac{1}{8}( (x^1)^2+(x^2)^2+(x^3)^2) , \quad p^2 = x^1x^2x^3, \quad p^3  = a ( (x^1)^4+(x^2)^4+(x^3)^4) + b (p^1)^2, 
\end{align*}
for some $a, b \in \mathbb{C}$, $a\neq0$. Let $\alpha=e_2-e_3$ and consider the corresponding stratum $D= \Pi_\alpha$. Then the determinant of metric $\widehat{\eta}$ restricted to $D$ is proportional to
\begin{equation*}
(x^3)^2 \big((x^1)^2 -( x^3)^2\big)^2 \big(  ( -64a +32 a^2 -b)(x^1)^2 - 2 (32a +b)(x^3)^2\big). 
\end{equation*}
Furthermore, $\operatorname{det}\widehat{\eta}_D$ is a product of linear factors which all vanish on the restricted arrangement $\mathcal{A}_D$ exactly when $a$ and $b$ satisfy the following relations:
\begin{align}
\operatorname{det}\widehat{\eta}_D & \sim \Big(x^1  x^3 \big((x^1)^2 -( x^3)^2\big)\Big)^2, \quad b= -32a, \tag{A}\\
\operatorname{det}\widehat{\eta}_D &\sim (x^3)^4 \big((x^1)^2 - (x^3)^2\big)^2, \quad b=32 a(a -2 ),\tag{B}\\ 
\operatorname{det}\widehat{\eta}_D& \sim (x^3)^2 \big((x^1)^2 - (x^3)^2\big)^3, \quad b= \frac{32}{3} a\left(a-4 \right) \tag{C}.
\end{align}

Note that $p^i$, ($i=1,2, 3$) are Saito polynomials if $a=-\frac{1}{2}$ and $b= 24$. In this case $\widehat{\eta}=\eta$ and $\operatorname{det}\eta_D$ takes the form (C) as expected from Theorem \ref{MMtheorem}. In cases (A) and (B) the multiplicities of linear factors have different forms. The factor $\frac{1}{8}$ in $p^1$ is necessary in order to ensure that $\widehat \eta$ can specialise into the Saito metric, while a replacement of $p^2$ with its multiple leads to a proportional metric $\widehat \eta$.
\end{example}

Further to Example \ref{example1}, in general metric $\widehat{\eta}$ in higher dimensions can have restriction on a stratum $D$ with nonlinear zero loci of the determinant. 
 
 The structure of the paper is as follows. In Section \ref{degreessection} we show that Theorem \ref{MMtheorem} is equivalent to two statements. The first one is Theorem \ref{thma} which states that $\operatorname{det}\eta_D$ is a product of linear forms whose zero set defines the restricted arrangement $\mathcal{A}_D$. The second statement is Theorem \ref{thma1} which states what the degrees of these linear factors are. They are Coxeter numbers of suitably defined root systems. We also explain in Section \ref{degreessection} that it is sufficient to prove Theorem \ref{MMtheorem} as well as Theorems \ref{thma}, \ref{thma1} for the case when stratum $D$ is in the closure of the fundamental domain. 

In Section \ref{lastsection} we revisit Dubrovin's duality on discriminant strata. Thus we prove in Proposition \ref{multonD} that discriminant strata are natural submanifolds.
% which is used later in Section \ref{classicalseriessection}. 
This allows us to define the operator of multiplication by the Euler vector field on the strata, and we show in Corollary \ref{cordet} that
its determinant is proportional to the determinant \eqref{rfth12} of the restricted Saito metric $\eta_D$.
 
 The proofs of main theorems which we give depend on the type of root system $\mathcal{R}$. Thus we prove Theorems \ref{thma} and \ref{thma1} for classical root systems in Section \ref{classicalseriessection} and \ref{classicalsection2} respectively. Considerations are based on the use of Landau-Ginzburg superpotentials for the corresponding Frobenius manifolds and their discriminant strata. 
 
 In Section \ref{generalformsection} we derive a general formula for the determinant of the restricted Saito metric on discriminant strata. We use this formula in Section \ref{exceptionalsection} where we prove Theorem \ref{MMtheorem} for the strata of the exceptional root systems in dimension $1$ and codimensions $1$, $2$ and $3$. The corresponding analysis in codimension $4$ is similar but technically more involved and we omit it, we refer to \cite{GA} for all the details. This covers all strata in the orbit spaces of the Coxeter groups $I_2(p)$, $H_3$, $H_4$, $F_4$.

 In Section \ref{exceptionalsectionrem} we consider the remaining cases, namely the strata of codimension $5$ in $E_7$ and the strata of codimensions $5$ and $6$ in $E_8$. In these cases we obtain explicit formulae for the determinant of the restricted Saito metric and analyse the corresponding multiplicities with the help of Mathematica. Determinant of the corresponding restricted Saito metric is given explicitly in the tables in that section. This completes Theorem \ref{MMtheorem} for all the cases. 
 
\vspace{3mm}

{\bf Acknowledgements}. We would like to thank Theo Douvropoulos for interesting discussions. The work of Georgios Antoniou was funded by EPSRC doctoral training partnership grants EP/M506539/1, EP/M508056/1, EP/N509668/1.
 
 \section{Degrees of linear factors}\label{degreessection}

In this section we give an equivalent formulation of Theorem \ref{MMtheorem}. 
We start with introducing some notations.

%Theorem \ref{MMtheorem} states, in particular, that the determinant of the restricted Saito metric $\eta_D$ for a discriminant stratum $D$ is a product of linear forms. 
For any stratum $D$ and an  element $H\in \mathcal{A}_D$ of the restricted arrangement $ \mathcal{A}_D$  given by  \eqref{resarr}     we choose a covector $l_H \in D^*$ such that $H=\{ x \in D | l_H(x)=0\}$. We can identify vectors and covectors using bilinear form $(\, , \,)$, so that $\beta \in \mathcal{R}$ will also mean a covector $\beta=\beta(x)=(\beta,x)$, where $x\in V$. Note that for any $H \in \mathcal{A}_D$ there is a root $\beta \in \mathcal{R}$ such that $\left. \beta \right|_D$ is proportional to $l_H$. 

%In these terms Theorem \ref{MMtheorem} implies the following statement. 

Recall that a root system $\mathcal R$ is called irreducible if it is not a disjoint union of its two orthogonal subsets, and otherwise it is called reducible \cite{cox1}.
The following fact follows from this definition.  

\begin{lemma}\label{irredulemma1}
Let $\mathcal{R}$ be a reducible root system so that $\mathcal{R}= \mathcal{R}_1 \sqcup \mathcal{R}_2$, where subsets $\mathcal{R}_1, \mathcal{R}_2 \subset \mathcal{R}$ 
satisfy orthogonality $(\alpha, \beta)=0$ for any $\alpha \in {\mathcal R}_1, \beta\in{\mathcal R}_2$. 
 Consider the corresponding vector space decomposition $\langle \mathcal{R} \rangle = V_1 \oplus V_2$, where $V_i=\langle \mathcal{R}_i \rangle$, $i=1, 2$. Let $\widetilde{\mathcal{R}} \subset \mathcal{R}$ be an irreducible subsystem. Then either $\widetilde{\mathcal{R}} \subset V_1$ or $\widetilde{\mathcal{R}} \subset V_2$.
\end{lemma}

The root system ${\mathcal R}_D$ is defined as the subsystem in the root system $\mathcal R$ orthogonal to the stratum $D$:
\beq{RD}
{\mathcal R}_D = \{\alpha \in {\mathcal R}| (\alpha,x)=0 \,\,  \forall x\in D\}.
\eeq

 Let $u \in V$ and let $B\subset V$ be a subset of vectors. We will denote by $\langle B, u \rangle$ the vector space spanned by elements of $B$ and $u$. For any $\beta \in \mathcal{R}\setminus \mathcal{R}_D$ we define the root system $\mathcal{R}_{D, \beta}= \langle \mathcal{R}_D, \beta\rangle \cap \mathcal{R}$ which can be represented as a disjoint union of irreducible root systems $\mathcal{R}_{D, \beta}^{(i)}$, ($i=0, \dots, p$), as follows:
\begin{equation}\label{decomposition}
\mathcal{R}_{D, \beta}= \bigsqcup_{i=0}^p \mathcal{R}_{D, \beta}^{(i)}.
\end{equation}
We will assume that $\beta \in \mathcal{R}_{D, \beta}^{(0)}$. Recall decomposition \eqref{RDdecompo} of the root system $\mathcal{R}_D$. It follows from Lemma \ref{irredulemma1} that 
\begin{align}\label{irred1}
\mathcal{R}_{D, \beta}^{(0)} \supset \bigsqcup_{i \in I} \mathcal{R}_D^{(i)},
\end{align}
for some subset $I \subset \{1, \dots, l\}$ and
\begin{align}\label{irred2}
\mathcal{R}_{D, \beta}^{(j)} = \mathcal{R}_D^{(i_j)},
\end{align}
where $1 \leq j \leq p$, $p=l-|I|$ and $i_j \in \{1, \dots, l \} \setminus I$.

The next proposition explains that different choices of $\beta$ can lead to the same root systems ${\mathcal R}_{D, \beta}$ and ${\mathcal R}_{D, \beta}^{(0)}$.

\begin{proposition}\label{propodeg} Let $\mathcal{R}_{D, \beta}^{(0)}$ be the root system from the decomposition (\ref{decomposition}). Let $\widetilde{\beta} \in \mathcal{R}$ be such that $\left. \widetilde{\beta} \right|_D$ is a non-zero multiple of $\left. \beta \right|_D$. Then $\widetilde{\beta} \in \mathcal{R}_{D, \beta}^{(0)}$.
\end{proposition}

\begin{proof}
Let $\widehat{V}$ be the vector space $\widehat{V}= \langle \mathcal{R}_D, \beta\rangle= \langle \mathcal{R}_D, \widetilde{\beta} \rangle$ and consider the root system $\widehat{\mathcal{R}}= \widehat{V} \cap \mathcal{R}$. Then $\widehat{\mathcal{R}}$ takes the form 
 $$\widehat{\mathcal{R}}=  \mathcal{R}_{D, \beta} =  \bigsqcup_{i=0}^p \mathcal{R}_{D, \beta}^{(i)}. $$
Clearly $\widetilde \beta \in \widehat {\mathcal R}$. 
Let us suppose that $\widetilde{\beta} \not \in \mathcal{R}_{D, \beta}^{(0)}$. Then $\widetilde{\beta} \in \mathcal{R}_D^{(i)}$ for some $i \in \{1, \dots, l\} \setminus I$, hence $\left. \widetilde{\beta} \right|_D=0$, which is a contradiction. Thus the statement follows. 
\end{proof}

We will use the following useful statement on the cardinality of the restricted Coxeter arrangement.

\begin{proposition}\cite{orlik}\label{arrangement}
Let $\mathcal{A}$ be the Coxeter arrangement for an irreducible Coxeter group $W$, and let $H \in \mathcal{A}$. Then the cardinality of $\mathcal{A}_H$ is
\begin{equation}
|\mathcal{A}_H|= |\mathcal{A}| - h +1,
\end{equation}
where $h$ is the Coxeter number of $W$. In particular, $|\mathcal{A}_H|$ does not depend on the choice of $H$.
\end{proposition}

Let us also recall the following result.

\begin{proposition}\label{propos8}\cite{bourbaki}
Let $\mathcal{R}$ be an irreducible root system of rank $r$, and let $\mathcal{A}$ be the corresponding Coxeter arrangement. Let $h$ be the Coxeter number of $\mathcal{R}$. Then the cardinality $$|\mathcal{A}|=\frac{r h}{2}.$$ 
\end{proposition}

%Now let us state what the degrees $k_H$ in Theorem \ref{thma} are. 
Now we are ready to rewrite the function \eqref{rfth12} on $D$  as a product of linear factors with certain multiplicities, which, in particular, implies that it is a polynomial.  
The multiplicity of the factor $l_H\in D^*$ is given in terms of the root system  $\mathcal{R}_{D, \beta}^{(0)}$, where $\beta \in \mathcal{R}$ is such that its restriction $\left.\beta\right|_D$ is proportional to $l_H$. It follows from Proposition \ref{propodeg} that the subsystem $\mathcal{R}_{D, \beta}^{(0)} \subset \mathcal{R}$ does not depend on the choice of $\beta$ for a given $l_H$. The following theorem holds.

\begin{theorem}
\label{theoremequiv}
For any stratum $D$ let ${\mathcal R}_D$ be the corresponding root system \eqref{RD} with decomposition \eqref{RDdecompo} into irreducible subsystems.
Let $Q$ be  given by
\beq{polQ}
Q=  \mathcal{I}(\mathcal{A}\setminus \mathcal{A}^D)^m  \prod_{i=1}^l \mathcal{I}_i^{r_i},
\eeq
where polynomials $\mathcal{I}_i$ are given by formula \eqref{Ii} with $\gamma_i \in \mathcal{R}_D^{(i)}$, $r_i$ is the rank of $ \mathcal{R}_D^{(i)}$, and  $m=2-\sum_{i=1}^lr_i$.
 For any $H \in \mathcal{A}_D$ let $\beta = \beta_H \in \mathcal{R}$ be such that $\left. \beta \right|_D$ is a non-zero multiple of $l_H$. 
Then restriction of  $Q$ on the stratum $D$ satisfies proportionality 
$$
Q|_D \sim \prod_{H \in {\mathcal A}_D} l_H^{k_H},
$$
where $k_H=h(\mathcal{R}_{D, \beta}^{(0)})$ is the Coxeter number of the root system $\mathcal{R}_{D, \beta}^{(0)}$ from the decomposition (\ref{decomposition}).
\end{theorem}

\begin{proof}
Let $\widehat{\mathcal{A}}$ denote the arrangement corresponding to the root system $\mathcal{R}_{D, \beta}$ and let $\mathcal{A}^{(i)}$ be the arrangement corresponding to the root system $\mathcal{R}_{D, \beta}^{(i)}$, ($0\leq i \leq p$). Let $\alpha \in \mathcal{R}_D^{(i)}$, and let $h^{(i)}$ denote the  Coxeter number of the root system $\mathcal{R}_D^{(i)}$.
 Then the multiplicity of $\left. \beta \right|_D$ in $\left. I_i \right|_D$ is given by
\begin{equation}
|\widehat{\mathcal{A}}_{\Pi_\alpha} \setminus \mathcal{A}^D_{\Pi_\alpha}|=|\widehat{\mathcal{A}}_{\Pi_\alpha}|-|\mathcal{A}^D_{\Pi_\alpha}|=|\widehat{\mathcal{A}}_{\Pi_\alpha}|-\frac{1}{2}\sum_{j=1}^l r_j h^{(j)} +h^{(i)} -1
\end{equation}
by Propositions \ref{arrangement}, \ref{propos8}. Also the multiplicity of $\left. \beta \right|_D$ in $\left. I(\mathcal{A}\setminus \mathcal{A}^D)\right|_D$ is given by
\begin{equation}\label{aaaaa}
|\widehat{\mathcal{A}}|-|\mathcal{A}^D|=\frac{rh}{2} + \sum_{j=1}^p|\mathcal{A}^{(j)}|  -\frac{1}{2}\sum_{j=1}^l r_j h^{(j)},
\end{equation}
where $r$ denotes the rank of the root system $\mathcal{R}_{D, \beta}^{(0)}$ and $h=h(\mathcal{R}_{D, \beta}^{(0)})$. 
Note that for any $\alpha \in \mathcal{R}_D^{(i)}$ we have
\begin{equation}
|\widehat{\mathcal{A}}_{\Pi_\alpha}|=|\mathcal{A}_{\Pi_\alpha}^{(0)}|-h+1 + \sum_{j=1}^p |\mathcal{A}^{(j)}| = \frac{rh}{2}-h+1 +\sum_{j=1}^p |\mathcal{A}^{(j)}|,
\end{equation}
if $i \in I$, and
\begin{equation}\label{rfth14}
|\widehat{\mathcal{A}}_{\Pi_\alpha}|=\frac{rh}{2} + |\mathcal{A}^{(i)}|-h^{(i)} +1 + \sum_{\substack{j=1 \\ j \neq i }}^p |\mathcal{A}^{(j)}| = \frac{rh}{2} +\sum_{j=1}^p |\mathcal{A}^{(j)}|-h^{(i)}+1,
\end{equation}
if $i \notin I$. Therefore by formulae \eqref{aaaaa}--\eqref{rfth14} and Theorem \ref{MMtheorem} we have that the multiplicity $k_H$ of the factor $l_H$ in \eqref{factorisation} is given by
\begin{align*}
k_H&=\sum_{i \in I}r_i \left(\frac{rh}{2} -h +h^{(i)} + \sum_{j=1}^p |\mathcal{A}^{(j)}| - \frac{1}{2}\sum_{j=1}^l r_j h^{(j)}\right) + \sum_{i \notin I} r_i \left( \frac{rh}{2} + \sum_{j=1}^p |\mathcal{A}^{(j)}| - \frac{1}{2}\sum_{j=1}^l r_j  h^{(j)}\right)\\
&+ (2-\sum_{i=1}^lr_i)\left(\frac{rh}{2} + \sum_{j=1}^p |\mathcal{A}^{(j)}| - \frac{1}{2} \sum_{j=1}^l r_j h^{(j)} \right).
\end{align*}
Note that $\sum_{j=1}^p |\mathcal{A}^{(j)}| =\frac{1}{2}\sum_{i \notin I}r_i h^{(i)}$, and $r=\sum_{i=1}^lr_i+1$. Therefore $k_H=h$, as required.
\end{proof}

Theorem \ref{theoremequiv} implies, in particular, the following statement.
\begin{corollary}
The restriction $Q|_D$ of the polynomial \eqref{polQ} does not depend (up to proportionality) on the choices of  roots $\gamma_i \in \mathcal{R}^{(i)}_D$. 
\end{corollary}
Theorem \ref{theoremequiv}  also implies that Theorem \ref{MMtheorem} is equivalent to the next two statements. 
\begin{theorem}\label{thma}
The determinant of the restricted Saito metric $\eta_D$ is proportional to the product of linear forms
\begin{equation}\label{factorisation}
\prod_{H \in \mathcal{A}_D} l_H^{k_H}, 
\end{equation}
where $k_H \in \mathbb{N}$.
\end{theorem}

\begin{theorem}\label{thma1}
%Suppose Theorem \ref{MMtheorem} is true. 
Let $H \in \mathcal{A}_D$. Let $\beta \in \mathcal{R}$ be such that $\left. \beta \right|_D$ is a non-zero multiple of $l_H$. The multiplicity of $l_H$ in the expression (\ref{factorisation}) is $k_H= h$, where $h=h(\mathcal{R}_{D, \beta}^{(0)})$ is the Coxeter number of the root system $\mathcal{R}_{D, \beta}^{(0)}$ from the decomposition (\ref{decomposition}).
\end{theorem}

%Furthermore, Theorem \ref{theoremequiv} also shows that Theorems \ref{thma}, \ref{thma1} imply Theorem \ref{MMtheorem} so the statements are equivalent. Note also that it follows from Theorem \ref{theoremequiv} that restriction on $D$ of the  polynomial $Q$ given by \eqref{polQ} (equivalently, restiction of the polynomial \eqref{rfth12} in Theorem \ref{MMtheorem})  does not depend on the choice of roots $\alpha \in \mathcal{R}^{(i)}_D$. 

Let us show that degree of determinant of the restricted Saito metric is equal to the degree of the polynomial $Q$ given by \eqref{polQ}. Components of the Saito metric $\eta$ are homogeneous polynomials in $x^i$ of degrees equal to the Coxeter number $h$. Therefore degree of $\det \eta_D(x)$ is equal to $h d$, where the dimension of the stratum    $d=\dim D$ is given by
\beq{dimstr}
d = n - \sum_{i=1}^l r_i.
\eeq
 The following statement takes place.
\begin{proposition}
Polynomial \eqref{polQ} has degree  $\deg Q = h d$, where $d$ is given by \eqref{dimstr}.
\end{proposition}
\begin{proof}
By Propositions \ref{arrangement}, \ref{propos8} we have
\beq{techn1}
\deg {\mathcal I}_i = \frac12 nh - h  -\frac{1}{2}\sum_{j=1}^l r_j h^{(j)} +h^{(i)} 
\eeq
for all $i=1, \ldots, l$.
Similarly
\beq{techn2}
\deg {\mathcal I} ({\mathcal A}\setminus {\mathcal A}^D) = \frac{1}{2} nh    -\frac{1}{2}\sum_{j=1}^l r_j h^{(j)}.
\eeq
Since 
$$
\deg Q = \sum_{i=1}^l r_i \deg {\mathcal I}_i  + (2 - n +d) \deg {\mathcal I} ({\mathcal A}\setminus {\mathcal A}^D),
$$
where we used formula \eqref{dimstr},  the statement follows from formulae \eqref{techn1} and \eqref{techn2}. 
\end{proof}

We are going to prove Theorem \ref{MMtheorem} for a subset of simple roots $S=\{\alpha_{i_1}, \dots, \alpha_{i_k}\} \subset \Delta$, $1 \leq k <n$ and the corresponding stratum $D=D_S$. Let us show how the statement of Theorem \ref{MMtheorem} then follows in general.  Let $ \widetilde{D} \subset V$ be a stratum such that there exists $w \in W$ satisfying $\widetilde{D}=wD$.
 
 \begin{lemma}\label{relationdeterminants}
 Let $y^{i}$ and $z^{i}$ ($i=1, \dots, n-k$) be some coordinates on $D$ and $\widetilde{D}$ respectively. Then 
 \begin{equation}
 \operatorname{det}\eta_{\widetilde{D}}(z)=( \operatorname{det}B)^{-2} \operatorname{det}\eta_D(y), 
 \end{equation}
 where $B=(\frac{\partial z^{i}}{\partial y^{j}})_{i, j=1}^{n-k}$ is the Jacobi matrix of the transformation $w \in W$, $w\colon D \rightarrow \widetilde{D}$. 
 \end{lemma}
 
 \begin{proof}
The statement follows from $W$-invariance of the metric $\eta$, that is from the relation 
%
%\begin{equation}\label{relationeta_Deta_widetildeD}
$\eta_{{D}}=w^{*} \eta_{\widetilde D}$, where $w^*$ is the pullback map.
%\end{equation}
%
%It follows from equality (\ref{relationeta_Deta_widetildeD}) that the determinant of $\eta_{\widetilde{D}}$ is obtained from the determinant of $\eta_D$ by replacing $y$ coordinates with $z$ coordinates, and the statement follows.
 \end{proof}
 
 This implies the following statement. %$W$-invariance of Theorem \ref{thma}.
 
 \begin{proposition}
 Suppose that Theorem \ref{MMtheorem} is true for $D$. Then it is also true for $\widetilde{D}$. 
 \end{proposition}
 
 The following result implies that it is sufficient to prove Theorem \ref{MMtheorem} for $S\subset \Delta$. It follows from the transitivity of the action of $W$ on the family of alcoves and their closure.
 
 \begin{proposition}
\label{propwaction}
\cite{cox} Let $\widetilde{S}$ be a collection of linearly independent roots $\widetilde{S}= \{\gamma_{1}, \dots, \gamma_{k}\} \subset  \mathcal{R}$ and let $\widetilde{D}$ be the corresponding stratum $ \widetilde{D} =D_{\widetilde{S}}$. Then there exists $w \in W$ such that $D = w^{-1} \widetilde{D}$ has the form $D=D_S$, where $S\subset \Delta$. 
 \end{proposition}
 
%\textcolor{red}{We are going to prove the main theorems for any root system assoiated to a finite Coxeter group. When the root system $\mathcal{R}$ is a classical root system, $A_n$, $B_n$ or $D_n$ we use supepotential description of the Frobenius manifold $\mathcal{M}_W$. In the case of exceptional root systems we proceed by different analysis which depends crucially on codimension of the stratum.}

%Below in this section we explain how the multiplicities $k_H$ in Theorem \ref{thma} can be described in terms of Coxeter geometry of the stratum $D$. 

\section{Almost duality on discriminant strata}\label{lastsection}
Almost duality for Frobenius manifolds was introduced in \cite{DubA}. Given a Frobenius manifold $\mathcal{M}$ with multiplication $\circ$ on the tangent spaces $T_x \mathcal{M}$ one defines a new multiplication $*$ on $T_x \mathcal{M}$. Each multiplication is a part of Frobenius algebra structure on the tangent space and almost duality is a relation between these two families of Frobenius algebras on the tangent spaces.  

In this section we revisit almost duality relation between Frobenius structures on discriminant strata. Multiplication $*$ is well-defined on the strata and a version of such a duality was established in \cite{FV}. It was suggested earlier in \cite{ian} that discriminant strata are natural submanifolds so that multiplication $\circ$ of tangential vectors is defined and belongs to the stratum. We prove this below. 

Let $\circ$ denote Frobenius multiplication on the orbit space $\mathcal{M}_W$. For any $x \in \mathcal{M}_W \setminus \Sigma$ the almost dual Frobenius multiplication is defined by the following formula \cite{DubA}:
\begin{align}\label{dualformula}
u * v= E^{-1} \circ u \circ v, 
\end{align}
where $u, v \in T_x \mathcal{M}_W$ and $E$ is the Euler vector field $$E=\frac{1}{h}x^i \frac{\partial}{\partial x^i}= \frac{1}{h}\sum_{\alpha} d_\alpha t^\alpha \frac{\partial}{\partial t^\alpha}, $$ where $h$ is the Coxeter number. Recall that 
$e=\partial_{t^n}$ is the identity field for the multiplication $\circ$ while $E$ is the identity field for the almost dual  multiplication $*$. 
  Let vector field $e^{-1}$ be the inverse field for the field $e$ with respect to the almost dual multiplication, namely $e^{-1} * e = E$. It follows from formula \eqref{dualformula} that $E=E^{-1} \circ e^{-1}$, and hence $e^{-1}$ can be represented as 
\begin{align}\label{inversee}
e^{-1}= E \circ E. 
\end{align}
Note that we also have by formulae \eqref{dualformula}, \eqref{inversee} that 
\begin{align}\label{dualformula1}
e^{-1} * u *v= E^{-1} \circ (e^{-1} * u) \circ v= E^{-1} \circ (E^{-1} \circ e^{-1} \circ u) \circ v = u \circ v.
\end{align}
Let us now recall that Saito metric $\eta$ and metric $g$ are related as follows \cite{dub}: 
\begin{align}\label{saitog1}
\eta(u, v)= g( E \circ u, v).
\end{align}
Let us consider the vector field $e^{-1}=e^{-1}(x)$ as a vector field on $V$, $x \in V$. 

\begin{lemma}\label{einverseonx0}
The vector field $e^{-1}(x)$ is well-defined at $x=x_0 \in D$. Moreover, $e^{-1}(x_0) \in T_{x_0}D$.
\end{lemma}

\begin{proof}
We have by formulae \eqref{dualformula1} and \eqref{saitog1} that
\begin{equation}\label{saitog}
\eta(u, v)= g( E \circ u, v)= g( e^{-1} * u, v).
\end{equation} 
For the components $(e^{-1})^j$ ($1 \leq j \leq n$) of the vector field $e^{-1}$ we have
\begin{equation*}
(e^{-1})^{j} = g(e^{-1},  \frac{\partial}{\partial x^{j}})=g(e^{-1} * E,  \frac{\partial}{\partial x^{j}})= \eta(E, \frac{\partial}{\partial x^{j}}),
\end{equation*}
where the last equality follows by \eqref{saitog}. Then 
\begin{align}\label{einverse(x)}
(e^{-1})^{j}= \frac{1}{h} \sum_{\alpha=1}^n d_\alpha t^\alpha \frac{\partial t^\beta}{\partial x^j} \eta(\frac{\partial}{\partial t^\alpha}, \frac{\partial }{\partial t^\beta})= \frac{1}{h} \sum_{\alpha=1}^n d_\alpha t^\alpha \frac{\partial t^\beta}{\partial x^j}\eta_{\alpha \beta}= \frac{1}{h} \sum_{\alpha=1}^n d_\alpha t^{\alpha} \frac{\partial t^{n+1-\alpha}}{\partial x^{j}}, 
\end{align}
since $\eta_{\alpha \beta}= \delta_{\alpha + \beta, n+1}$. Thus the first part of the statement follows. 

Let $\gamma \in \mathcal{R}_D$, and let $\partial_\gamma$ be the corresponding vector field orthogonal to $D$. Then we have by formula \eqref{einverse(x)} that $(e^{-1}(x), \partial_\gamma)=0$ as $x$ tends to $x_0$ since $\partial_\gamma t^{n+1-\alpha}|_D=0$. Therefore $e^{-1}(x_0) \in T_{x_0}D$, as required. 
\end{proof}

It was shown in \cite{FV} that the left-hand-side of equality \eqref{dualformula} can be restricted to any stratum $D$ with $u$ and $v$ being tangential vectors to $D$. More precisely, let $\Sigma_D$ be the union of the hyperplanes $\Pi_\gamma \cap D$ in $D$, where $\gamma \in \mathcal{R} \setminus \mathcal{R}_D$ and consider the point $x_0 \in D \setminus \Sigma_D$. Let $u, v \in T_{x_0}D$ and consider extension of $u$ and $v$ to two local analytic vector fields $u(x), v(x) \in T_x V$ such that $u(x_0)=u$ and $v(x_0)=v$. Recall that the multiplication $u(x) *v(x)$ has a limit when $x$ tends to $x_0$ which is proportional to 
\begin{align*}
 \sum_{\alpha \in \mathcal{R} \setminus \mathcal{R}_D} \frac{(\alpha,u)(\alpha,v)}{(\alpha,x_0)} \alpha.
\end{align*}
Furthermore, the product $u *v$ at $x_0$ is tangential to $D$ \cite{FV}. As a result we get the following statement using Lemma \ref{einverseonx0} and formula \eqref{dualformula1}.

\begin{proposition}\label{multonD} Let $u, v \in T_{x_0}D$, $x_0 \in D \setminus \Sigma_D$. Then $u \circ v \in T_{x_0}D$. Furthermore, $u \circ v = e^{-1} * u *v$.
\end{proposition}

The proposition is the strengthening of the results and observations from \cite{ian, FV}. Namely, in Dubrovin's duality formula \eqref{dualformula1} both sides are well-defined if $u, v \in T_{x_0}D$ and equality remains to be satisfied on $D$.

\begin{remark}
The metric $\eta_D$ is in fact non-degenerate on $D$ at any point $x_0 \in D \setminus \Sigma_D$. In the case of classical root systems this follows from Propositions \ref{criticalv}, \ref{criticalvBDN} below. In the case of strata of codimension $1$, $2$ and $3$ in exceptional root systems this follows from the analysis in Section \ref{exceptionalsection}. For the strata considered in Section \ref{exceptionalsectionrem} and for the strata of dimension $1$ it follows by calculations in Mathematica. For the remaining case of strata of codimension $4$ in exceptional root systems we refer to \cite{GA}.
\end{remark}

Let $E_D$ be the restriction of the Euler vector field on $D$. Let $x\in D\setminus \Sigma_D$. Note that $E_D \in T_x  D$ and by Proposition \ref{multonD} we have the operator of multiplication
$$
E_D\colon T_x D \to T_x D, \quad E_D(u) = E_D \circ u, \quad (u\in T_x D).
$$
Let $g_D$ be the restriction of the metric $g$ on the stratum $D$. Formula \eqref{saitog1}  then implies that we have the following relation
\begin{equation}\label{saitogD}
\eta_D(u, v)= g_D( E_D \circ u, v)
\end{equation} 
for any $u, v \in T_xD$.

Let us consider coordinates  on $D$ which are linear combinations of the coordinates $x^i$. Let $\det \eta_D$ be the determinant of the matrix of the metric $\eta_D$ in these coordinates.
Note that the metric $g_D$ is constant in these coordinates. Therefore relation \eqref{saitogD} implies the following statement.
\begin{corollary}
\label{cordet}
Determinant of the operator $E_D$ is proportional to $\det \eta_D$.
\end{corollary}

\section{Classical series: Theorem \ref{thma}}\label{classicalseriessection}

In this section we prove Theorem \ref{thma} for the classical Coxeter groups $W$. 

We recall that Frobenius manifold $\mathcal{M}_W$ is semisimple for any irreducible finite Coxeter group $W$ \cite{dub}, that is Frobenius algebra at a generic point on the manifold is semisimple. Locally, near a semisimple point $y \in \mathcal{M}_W$ there exists a basis of commuting vector fields $\delta_i$, $i=1, \dots, n$ such that $\delta_i \circ \delta_j = \delta_{ij} \delta_j$,
where $\delta_{ij}$ is the Kronecker symbol.
 Then $\delta_i$ determine the canonical coordinate system $(u_1, \ldots, u_n)$ such that $\delta_i = \frac{\partial}{\partial u_i}$.

Semisimple Frobenius manifolds  admit a description as Landau--Ginzburg (LG) models (see \cite{dub}). 
The theory of topological LG models involves a meromorphic function called superpotential which also depends on additional variables $y=(y_1, \dots, y_n)$. The metric and the structure constants of the Frobenius multiplication can be expressed via residue formulae involving derivatives of the superpotential $\lambda(p)=\lambda(p; y)$. Canonical coordinates $u_i$ are critical values of the superpotential $\lambda(p)$:
\begin{align*}
u_{i}(y) = \lambda(q_i(y); y), \quad \left. \lambda^{'}(p)= \frac{d\lambda(p)}{dp} \right|_{p=q_i(y)}=0,\quad i=1, \dots, n.
\end{align*}
%

%In this section we show that Theorem \ref{thma} holds for the determinant of the Saito metric restricted to a discriminant stratum of a classical Coxeter group. 
We will use LG superpotential description of the Frobenius structures on the discriminant strata. The superpotential depends on the type of the group $W$, and we consider different cases below.

\subsection{$\mathbf{A_n}$ discriminant strata}\label{subsectionA_N}
Let us recall that the LG superpotential is given by \cite{dub}
\begin{equation}
\lambda(p)= \prod_{i=1}^{n+1} ( p - x^i ),
\end{equation}
where $p$ is some auxiliary variable and $x^{i}$, $1 \leq i \leq n+1$, are the standard orthonormal coordinates in $\mathbb{C}^{n+1}$ with the additional assumption $\sum_{i=1}^{n+1} x^i =0$. Then $\lambda(p)$ is a function on the orbit space $\mathbb{C}^{n+1} / S_{n+1}$ for any fixed $p$, where the symmetric group $S_{n+1}$ acts in ${\mathbb C}^{n+1}$ by permutation of coordinates. The corresponding root system has the standard form $A_n=\{\pm (e_i - e_j)| 1\le i<j\le n+1\}$, where $e_i$ is the standard basis in ${\mathbb C}^{n+1}$.

By applying the action of the symmetric group $S_{n+1}$ we can assume that a stratum $D$ has the form 
%Let us consider an arbitrary discriminant stratum $D$ given by the following equations: 
%
\begin{alignat}{2}\label{stratumANsuper}
 x^1 &= \dots &&= x^{m_0}=\xi_{0}, \nonumber\\
 x^{m_0 +1} &= \dots &&= x^{m_0 +m_1}=\xi_{1}, \\
&\hspace{0.75cm} \vdots \nonumber \\
x^{\sum_{i=0}^{d-1} m_i +1} &= \dots &&= x^{\sum_{i=0}^{d} m_i}=\xi_d, \nonumber
\end{alignat}
where $d, m_i \in \mathbb{N}$ and $\sum_{i=0}^d m_i = n+1$. Note that the dimension of this stratum is $d$, and $\xi_1, \dots, \xi_d$ can be considered as coordinates on $D$, $\xi_0 = - \sum_{i=1}^d \frac{m_i}{m_0} \xi_i$. 

Then the superpotential for the stratum $D$ has the form 
\begin{equation}\label{superpotentialAn}
\lambda_D(p)=\left. \lambda(p)\right|_D= \prod_{i=0}^d ( p - \xi_i )^{m_i}.
\end{equation}
 The expressions for the Saito metric $\eta_D= \left. \eta \right|_D$ and algebra multiplication are then given as follows (cf. \cite{dub} for the case $m_i=1 \, \forall i$), 
\begin{align}
\eta_D( \zeta_1, \zeta_2) &= \sum_{p_s: \lambda_D^{'}(p_s)=0}  {\rm res}_{p=p_s} \frac{\zeta_1(\lambda_D)\zeta_2(\lambda_D)}{ \lambda_D^{'}(p)}dp, \label{saitoAn}\\
\eta_D(\zeta_1 \circ \zeta_2, \zeta_3) &= \sum_{p_s: \lambda_D^{'}(p_s)=0}   {\rm res}_{p=p_s} \frac{\zeta_1(\lambda_D)\zeta_2(\lambda_D)\zeta_3(\lambda_D)}{ \lambda_D^{'}(p)}dp\label{multAn},
\end{align}
where $\zeta_1, \zeta_2, \zeta_3 $ are  vector fields on $D$ and $\lambda_D^{'}(p)= \frac{d\lambda_D(p)}{dp}$. 
\begin{proposition}\label{superpropAn}
%On the stratum $D$ we have the following expression for 
The derivative $\lambda_D^{'}$ has the form 
\begin{equation}\label{derilambda}
\lambda_D^{'}(p)= (n+1) \prod_{i=0}^d (p - \xi_i)^{m_i -1} \prod_{i=1}^d (p - q_i),
\end{equation}
for some   $q_1,\dots, q_d \in \mathbb{C}$. The second derivative of $\lambda_D$ satisfies
\begin{equation}\label{lambda''(q_i)AN}
\lambda_D^{''}(q_l) = (n+1) \prod_{\substack{j=1\\ j \neq l}}^d q_{lj} \prod_{a=0}^d (q_l - \xi_a)^{m_a -1},
\end{equation}
where $q_{lj} = q_l -q_j$.
\end{proposition}

\begin{proof}
By formula (\ref{superpotentialAn}), we have
\begin{equation*}
\lambda^{'}_D(p)= \prod_{i=0}^d (p- \xi_i)^{m_i -1} R(p),
\end{equation*}
for some polynomial $R \in \mathbb{C}[p]$, $\operatorname{deg}R=d$. The statement follows.
\end{proof}

%The following formula which follows from Proposition \ref{superpropAn} will be useful below
%
%\begin{equation}\label{lambda''(q_i)AN}
%\lambda_D^{''}(q_l) = (n+1) \prod_{\substack{j=1\\ j \neq l}}^d q_{lj} \prod_{a=0}^d (q_l - \xi_a)^{m_a -1},
%\end{equation}
%
%where $q_{lj} = q_l -q_j$.

Let $u_i = \lambda_D(q_i)$, $i=1, \dots, d$. Similarly to the case $d=n$  \cite{dub} we have the following statement. 

\begin{lemma}\label{superpotentialdelta} For any $i,j=1, \ldots, d$ we have
\begin{equation}\label{partiallambda}
\left. \partial_{u_i} \lambda_D(p)\right|_{p=q_j} = \delta_{ij}.
\end{equation}
\end{lemma}

\begin{proof}By definition we have 
\begin{equation*}
\delta_{ij} = \frac{\partial u_j}{\partial u_i} = \partial_{u_i} \lambda_D(q_j). 
\end{equation*}
By considering the Taylor expansion of $\lambda(p)$   at $p=q_j$ we have
\begin{equation*}
\lambda_D(p) = \lambda_D(q_j) + \mathcal{O}, 
\end{equation*}
where the term $\mathcal{O}$  has zero of order at least two at $p=q_j$. Then 
\begin{equation}\label{partialu_i lambda1}
\left. \partial_{u_i} \lambda_D(p) \right|_{p=q_j} = \partial_{u_i}\lambda_D(q_j),
\end{equation}
and the statement follows.
\end{proof}

Similarly to the case $d=n$ we obtain the following result.

\begin{lemma}\label{jacobianlemma}We have 
\begin{equation}\label{jacobimatrixxiuAn}
\partial_{u_l} \lambda_D(p) = \frac{\lambda_D^{'}(p)}{(p-q_l)\lambda_D^{''}(q_l)},
\end{equation}
and 
\begin{equation}\label{jacobian}
\partial_{u_l} \xi_a = \frac{1}{(q_l - \xi_a)\lambda_D^{''}(q_l)}, 
\end{equation}
where $a, l=1, \dots, d$.
\end{lemma}

\begin{proof}
It follows from formula (\ref{superpotentialAn}) that
\begin{equation}\label{lambdauiF}
\partial_{u_l} \lambda_D(p) = \prod_{a=0}^d (p -\xi_a)^{m_a -1} F(p;l), 
\end{equation}
where $F \in \mathbb{C}[p]$ and $\operatorname{deg}F= d-1$. From Lemma \ref{superpotentialdelta} we have that
\begin{equation}
\label{extrafor}
F(q_j;l) = \frac{\delta_{lj}}{\prod_{a=0}^d (q_j - \xi_a)^{m_a -1}}.
\end{equation}
Since $\operatorname{deg} F = d-1$, the values $F(q_j; l)$, $j=1, \dots, d$ completely determine the polynomial $F(p; l)$ and therefore by the Lagrange interpolation formula we have $F(p;l) = \sum_{k=1}^d F_k(p;l)$,  where 
\begin{equation}\label{Fkpoly}
F_k(p;l) = F(q_k;l) \prod_{\substack{ i=1\\ i\neq k}}^d \frac{p-q_i}{q_k - q_i}.
\end{equation}
%
%Hence 
%
%\begin{equation}\label{Fkpoly}
%F_k (p;l) = \frac{\delta_{lk}}{\prod_{a=0} (q_k - \xi_a)^{m_a -1}} \prod_{\substack{  i \neq k}} \frac{p-q_i}{q_k - q_i}.
%\end{equation}
%
Note  that $F(p;l)= F_l(p;l)$. Therefore  
%where $\lambda_D^{''}(q_l)$ is given by \eqref{lambda''(q_i)AN}, 
formula \eqref{jacobimatrixxiuAn} follows by considering $\frac{\lambda_D^{'}(p)}{\lambda_D^{''}(q_l)}$, and applying  Proposition \ref{superpropAn} and formulae \eqref{lambdauiF} -- \eqref{Fkpoly}. 

%Let us express $\lambda_D(p)$ as the product $\lambda_D(p) = (p-\xi_0)^{m_0} \prod_{a=1}^n (p-\xi_a)^{m_a}$. Then
%
%\begin{align*}
%\partial_{u_l} \lambda_D(p)& = \partial_{u_l}\big((p-\xi_0)^{m_0}\big)  \prod_{a=1}^n (p-\xi_a)^{m_a} - (p- \xi_0)^{m_0} \prod_{a=1}^n (p-\xi_a)^{m_a} \big( \sum_{b=1}^n m_b \frac{\partial_{u_l} \xi_b}{p-\xi_b} \big) \\
Note that 
$$
\partial_{u_l} \lambda_D(p)=  \partial_{u_l}\big((p-\xi_0)^{m_0}\big) \prod_{a=1}^d (p-\xi_a)^{m_a}- \lambda_D(p) \sum_{b=1}^d m_b \frac{\partial_{u_l} \xi_b}{p-\xi_b},
$$
which is equal to 
 $\frac{\lambda_D^{'}(p)}{(p-q_l)\lambda_D^{''}(q_l)}$ by formula \eqref{jacobimatrixxiuAn}. Dividing both sides by $(p- \xi_k)^{m_k -1}$ for some $k$, $1 \leq k \leq d$ we arrive at the following relation
\begin{equation}\label{lambdaDoverp-xik1}
\partial_{u_l}\big[(p-\xi_0)^{m_0}\big] (p-\xi_k) \prod_{\substack{a=1 \\ a \neq k}}^d (p-\xi_a)^{m_a}-\frac{\lambda_D(p)}{(p-\xi_k)^{m_k-1}} \sum_{b=1}^d m_b \frac{\partial_{u_l} \xi_b}{(p-\xi_b)} =\frac{\lambda_D^{'}(p)}{(p-q_l)(p-\xi_k)^{m_k-1}\lambda_D^{''}(q_l)}.
\end{equation}
Note that 
\begin{align}\label{lambdaDoverp-xik}
\left. \frac{\lambda_D^{'}(p)}{(p-\xi_k)^{m_k-1}}\right|_{p=\xi_k}= m_k \prod_{\substack{b=0 \\ b \neq k}}^d (\xi_k-\xi_b)^{m_b}.
\end{align}
 By substituting $p=\xi_k$ in the relation \eqref{lambdaDoverp-xik1} the statement follows  with the help of \eqref{lambdaDoverp-xik}. 
%
%\begin{equation*}
%-   \prod_{\substack{a=0 \\ a \neq k}}^d (\xi_k - \xi_a)^{m_a} \partial_{u_l} \xi_k =   \frac{\prod_{\substack{b=0 \\ b \neq k}}^d (\xi_k-\xi_b)^{m_b}}{(\xi_k-q_l) \lambda_D^{''}(q_i)}.
%\end{equation*}
%
%The statement follows. 
\end{proof}

The critical values $u_i=\lambda_D(q_i)$, ($i=1, \dots, d$) play a role of the canonical coordinates for the metric and multiplication  (\ref{saitoAn}), (\ref{multAn}) on the stratum $D$. More precisely, the following statement holds.

\begin{proposition}\label{criticalv}
The following formulae hold for any $i,j = 1, \ldots, d$:
\begin{align*}
\eta_D(\partial_{u_i}, \partial_{u_j}) &= \frac{\delta_{ij}}{\lambda_D^{''}(q_i)},
\end{align*}
\begin{align*}
\partial_{u_i} \circ \partial_{u_j} &= \delta_{ij} \partial_{u_j}.
\end{align*}
In particular, the metric $\eta_D$ is non-degenerate.
\end{proposition}

\begin{proof}
Let us consider the formula  \eqref{saitoAn} with $\zeta_1 = \partial_{u_i}$, $\zeta_2 = \partial_{u_j}$. Note that the right-hand side is non-singular at $p=\xi_i$ for all $i=0, \ldots, d$.
Furthermore, the residue at $q_l$ ($1 \leq l \leq d$) is zero if $i \neq j$ by Lemma \ref{jacobianlemma}, and hence
%We use formulae \eqref{saitoAn}, \eqref{multAn} together with \eqref{jacobimatrixxiuAn}. We consider consider formulae \eqref{saitoAn}, \eqref{multAn} with the vector fields 
 %Note that the residues are trivial in $\xi_a$ ($0 \leq a \leq n$) by Lemma \ref{jacobianlemma}.
%Let us consider first formula \eqref{saitoAn}. In the case when $i \neq j$ the residues at $q_l$ ($1 \leq l \leq n$) are trivial by Lemma \ref{jacobianlemma}, 
$\eta_D(\partial_{u_i}, \partial_{u_j})=0$. Further on, by Lemma \ref{jacobianlemma}  we  have
\begin{align*}
\eta_D (\partial_{u_i}, \partial_{u_i}) = \sum_{s=1}^d {\rm res}_{p=q_s} \frac{(\partial_{u_i}\lambda_D(p))^2}{\lambda^{'}_D(p)}dp = \frac{1}{\lambda^{''}_D(q_i)^2} \sum_{s=1}^d {\rm res}_{p=q_s} \frac{\lambda^{'}_D(p)}{(p-q_i)^2}dp.
\end{align*}
It follows by Proposition \ref{superpropAn}  that
\begin{align*}
\eta_D (\partial_{u_i}, \partial_{u_i})= \frac{n+1}{\lambda^{''}_D(q_i)^2}\sum_{s=1}^d {\rm res}_{p=q_s} \frac{\prod_{j=0}^d (p-\xi_j)^{m_j -1 } \prod_{j \neq i}^d(p-q_j)}{p-q_i}dp=\frac{1}{\lambda^{''}_D(q_i)},
\end{align*}
as required. 

Let us now consider formula \eqref{multAn} with $\zeta_1 = \partial_{u_i}$, $\zeta_2 = \partial_{u_j}$ and $\zeta_3 = \partial_{u_k}$. In the case when at least two of the indices $i, j, k$ are different the residues at $q_l$ ($1 \leq l \leq d$) are zero by Lemma \ref{jacobianlemma}, and hence $\eta_{D}(\partial_{u_i} \circ \partial_{u_j}, \partial_{u_k})=0$. Further on, by Lemma \ref{jacobianlemma} we have
\begin{align*}
\eta_D ( \partial_{u_i} \circ \partial_{u_i}, \partial_{u_i}) =  \sum_{s=1}^d {\rm res}_{p=q_s}\frac{(\partial_{u_i}\lambda_D(p))^3}{\lambda^{'}_D(p)}dp = \frac{1}{\lambda^{''}_D(q_i)^3} \sum_{s=1}^d {\rm res}_{p=q_s} \frac{(\lambda^{'}_D(p))^2}{(p-q_i)^3}dp.
\end{align*}
It   follows by Proposition \ref{superpropAn}  that
\begin{align*}
\eta_D (\partial_{u_i} \circ \partial_{u_i} , \partial_{u_i})= \frac{(n+1)^2}{\lambda^{''}_D(q_i)^3}\sum_{s=1}^d {\rm res}_{p=q_s} \frac{\prod_{j=0}^d (p-\xi_j)^{2(m_j -1) } \prod_{j \neq i}^d(p-q_j)^2}{p-q_i}dp=\frac{1}{\lambda^{''}_D(q_i)}.
\end{align*}
Therefore
\begin{align*}
\eta_{D}( \partial_{u_i} \circ \partial_{u_j} , \partial_{u_k} ) = \frac{\delta_{ij} \delta_{jk}}{\lambda^{''}_D(q_i)},
\end{align*}
which implies the required relation $\partial_{u_i} \circ \partial_{u_j} = \delta_{ij} \partial_{u_j}$  since $\partial_{u_i} \circ \partial_{u_j} \in T_*D$ by Proposition \ref{multonD} and since $\eta_D$ is non-degenerate by above. 
\end{proof}

Now we are able to find determinant of the metric $\eta_D$ in the coordinates $\xi_i$, $1 \leq i \leq d$.

\begin{theorem}\label{etaDAn}
The determinant of the restricted Saito metric $\eta_D$ in the coordinates $\xi_i$ has the form 
\begin{equation}\label{restricteddethA_N}
\operatorname{det}\eta_D(\xi) = \kappa \prod_{0 \leqslant i < j \leqslant d}\xi_{ij}^{m_i + m_j},
\end{equation}
where $\xi_{ij}= \xi_i -\xi_j$ and $\kappa= (-1)^{\sum_{i=1}^d i m_i + n d} (n+1)^{-n} \prod_{a=1}^d m_a^2 \prod_{a=0}^d m_a^{m_a -1}$.
\end{theorem}

\begin{proof}
By Proposition \ref{criticalv}  the determinant of the Saito metric $\eta_D$ in the coordinates $u_i$, $1 \leq i \leq d$,  has the form 
\begin{equation}
\operatorname{det}\eta_D(u)= \prod_{i=1}^d \frac{1}{\lambda_D^{''}(q_i)}.
\end{equation}
Therefore by Proposition \ref{criticalv} we have that
\begin{equation}\label{saitodeterminantAn}
\operatorname{det}\eta_D(\xi)= (\operatorname{det}A)^{-2} \prod_{i=1}^d \lambda_D^{''}(q_i),
\end{equation}
where $A$ is the $d\times d$ matrix with the matrix elements $  \frac{1}{\xi_a -q_i }$, $1 \leqslant i, a \leqslant d$. Determinant of $A$ is the Cauchy determinant which can be expressed as
\begin{equation}\label{detofxiaqi}
\operatorname{det} A= (-1)^{\frac{d(d-1)}{2}} \frac{ \prod_{\substack{i=1 \\ i < j }}^d \xi_{ij} q_{ij}}{\prod_{i,a=1}^d (\xi_a-q_i)},
\end{equation}
where $q_{ij}=q_i -q_j$.  

By Proposition \ref{superpropAn}  we get 
\begin{equation}\label{detofxiaqisimpler}
\prod_{i=1}^d (\xi_a - q_i) = (n+1)^{-1} \left. \frac{\lambda_D^{'}(p)}{(p-\xi_a)^{m_a -1}}\right|_{p=\xi_a} \prod_{\substack{b=0 \\ b \neq a }}^d \xi_{ab}^{-m_b +1} = \frac{m_a}{n+1} \prod_{\substack{b=0 \\ b \neq a }}^d \xi_{ab}.
\end{equation}
%
%Therefore by considering the Taylor expansion of $\lambda_D(p)$ around $p=\xi_a$ we have
%
%\begin{equation*}
%\lambda_D^{'}(p)= \sum_{i=1}^{\infty} i  \frac{\lambda_D^{(i)}(\xi_a)}{i!} (p-\xi_i)^{i-1}.
%\end{equation*}
%
%Further on, we note that $\lambda^{(i)}(\xi_a)=0$ for $ 1 \leqslant i \leqslant m_a -1$ by formula  (\ref{derilambda}). Returning to relation (\ref{difference}), we have by \eqref{difference} 
%
%\begin{equation*}
%\prod_{i=1}^n (\xi_a - q_i) = \frac{m_a}{(N+1) m_a)!} \lambda_D^{(m_a)} (\xi_a) \prod_{\substack{b=0 \\ b \neq a }}^n \xi_{ab}^{-m_b +1},
%\end{equation*}
%
%where $\lambda_D^{(m_a)}(\xi_a) = m_a!  \prod_{\substack{b=0 \\ b \neq a }} \xi_{ab}^{m_b}$ using formula (\ref{superpotentialAn}). Therefore, we have
%%
%\begin{equation}
%\prod_{i=1}^n (\xi_a - q_i)= .
%\end{equation}
%
%We have the following relation
%
%\begin{equation}\label{xiab}
%\prod_{ 0 \leq a < b \leq b} \xi_{ab} = (-1)^{\frac{n(n+1)}{2}} \prod_{\substack{a=0\\ a <b}} \xi_{ab}^2.
%\end{equation}
%
%Then by formulae \eqref{detofxiaqi}, \eqref{detofxiaqisimpler} and \eqref{xiab} we have
%
%\begin{equation}\label{det(xa-qi)}
%\operatorname{det} \big( \frac{1}{\xi_a -q_i }\big)= \frac{(-1)^{n^2}(N+1)^n}{\prod_{a=1}^n m_a} %\frac{\prod_{1 \leqslant i <j \leqslant n}q_{ij}}{\prod_{0 \leqslant a < b \leqslant n}\xi_{ab}}.
%\end{equation}
%
%Moreover from Proposition \ref{superpropAn} we get $\lambda^{''}_D ( q_i) = (N+1) \prod_{a=0}^n (q_i - \xi_a)^{m_a -1} \prod_{ j \neq i } q_{ij}$. 
By formula \eqref{lambda''(q_i)AN} we get
\beq{anotherform}
\frac{\prod_{i=1}^d \lambda_D^{''}(q_i)}{\prod_{i <j} q_{ij}^2} = (-1)^{\frac{d(d-1)}{2}}(n+1)^d \prod_{\substack{0 \leqslant a \leqslant d \\ 1 \leqslant i \leqslant d }} (q_i -\xi_a)^{m_a -1}.
\eeq 
By substituting  formula \eqref{detofxiaqisimpler}  into \eqref{anotherform} we get 
 \begin{align}
\frac{\prod_{i=1}^d \lambda_D^{''}(q_i)}{\prod_{1\le i <j} q_{ij}^2}&= c \prod_{a=0}^d  \prod_{\substack{b=0 \\ b \neq a }}^d \xi_{ab}^{m_a -1}\label{fraction},
\end{align}
where $c=  (-1)^{\frac{d(d-1)}{2}+d\sum_{a=0}^d( m_a -1)} (n+1)^d \prod_{a=0}^d \big( (n+1)^{-1} m_ a \big)^{m_a -1} $. 

Combining formulae (\ref{saitodeterminantAn}), \eqref{detofxiaqi}, (\ref{fraction}), we obtain the following expression for $\operatorname{det}\eta_D$:
\begin{align*}
\operatorname{det}\eta_D(\xi)&= \frac{\prod_{i=1}^d \lambda_D^{''}(q_i) \prod_{i,a=1}^d (\xi_a - q_i)^2}{ \prod_{i<j} \xi_{ij}^2 q_{ij}^2}= c \prod_{a=1}^d \big( (n+1)^{-1} m_a \big)^2 \prod_{0 \leqslant a < b \leqslant d } \xi_{ab}^2 \prod_{a=0}^d \prod_{\substack{b=0 \\ b \neq a }}^d \xi_{ab}^{m_a -1}.
\end{align*}
Finally, we note that 
\begin{equation*}
\prod_{a=0}^d \prod_{\substack{b=0 \\ b \neq a }}^d \xi_{ab}^{m_a -1}= (-1)^{\sum_{i=1}^d i m_i -\frac{d(d+1)}{2}} \prod_{0 \leqslant a < b \leqslant d} \xi_{ab}^{m_a +m_b -2},
\end{equation*}
which gives the required statement as $c \prod_{a=1}^d \big( (n+1)^{-1} m_a \big)^2 (-1)^{\sum_{i=1}^d i m_i -\frac{d(d+1)}{2}}  = \kappa$. 
\end{proof}

\subsection{$\mathbf{B_n}$, $\mathbf{D_n}$ discriminant strata}\label{BDzeroallowedsubsection}We consider the LG superpotential 
\begin{equation}\label{superpotentialBDN}
\lambda(p) = p^{2k} \prod_{i=1}^n (p^2-(x^i)^2),
\end{equation}
where $p$ is some auxiliary variable and $x^i$, $1 \leq i \leq n$ are the standard orthonormal coordinates in $\mathbb{C}^n$ and $k=0, -1$. In the case $k=0$ polynomial \eqref{superpotentialBDN}  is the superpotential for the $B_n$ orbit space, and in the case $k=-1$ it is the superpotential for the $D_n$ orbit space \cite{DZ} (see also \cite{MB}).
The corresponding root systems have the standard form $D_n=\{\pm e_i \pm e_j| 1\le i<j\le n\}$, $B_n=D_n\cup\{\pm e_i| 1\le i \le n\}$, where $e_i$ is the standard basis in ${\mathbb C}^n$.

Let us consider a $B_n$/$D_n$ stratum $D$ in $\mathbb{C}^n$ which can be assumed to have the following form after applying the group action:

\begin{alignat}{2}\label{stratumBDNsuper}
x^1 &= \dots &&= x^l =0, \nonumber\\
\varepsilon_1  x^{l+1} &= \dots &&=\varepsilon_{m_1} x^{l+m_1}=\xi_{1}, \nonumber\\
 \varepsilon_{m_1 +1} x^{l+ m_1 +1} &= \dots &&=\varepsilon_{m_1+m_2}  x^{l+m_1 +m_2}=\xi_{2}, \\
&\hspace{0.75cm} \vdots\nonumber \\
\varepsilon_{\sum_{i=1}^{d-1} m_i +1} x^{l+ \sum_{i=1}^{d-1} m_i +1} &= \dots &&= \varepsilon_{\sum_{i=1}^{d} m_i} x^{l + \sum_{i=1}^{d} m_i}=\xi_d,\nonumber
\end{alignat}
where $l \in \mathbb{N} \cup \{0\}$, $\varepsilon_j= \pm 1$ ($j=1, \dots, n-l$), $d, m_i \in \mathbb{N}$ ($i=1, \dots, d$), $\sum_{i=1}^d m_i = n- l$, and $\xi_1, \dots, \xi_d$ are coordinates on $D$.  Equations \eqref{stratumBDNsuper} define discriminant stratum for $D_n$ provided $l \neq 1$, and they always define a discriminant stratum for $B_n$. 

We will consider in this subsection the following more general superpotential on $D$:
\begin{equation}\label{superpotentialBNDN}
\lambda_D(p)= p^{2m} \prod_{i=1}^d (p^2 - \xi_i^2)^{m_i}, 
\end{equation}
where $m \in \mathbb{Z}$ is such that $N :=m + \sum_{i=1}^d m_i\neq 0$. In the cases when $m \geq -1$ the superpotential $\lambda_D(p)$ corresponds to the restriction of the superpotential \eqref{superpotentialBDN} to the discriminant stratum \eqref{stratumBDNsuper}. Indeed, we get $\lambda_D(p)$ with $m=l-1$, where $l=0$ or $ l \geq 2$ by restricting $\lambda(p)$ with $k=-1$ (type $D_n$) to the stratum \eqref{stratumBDNsuper}. And we get $\lambda_D(p)$ with $m=l \geq 0$ as the restriction of $\lambda(p)$ with $k=0$ (type $B_n$) to the stratum \eqref{stratumBDNsuper}. Superpotentials \eqref{superpotentialBNDN} with $m_i=1$ for all $i$ and $-d+1 \leq m \leq 0$ were considered in \cite{MB,DZ}. 

Derivatives of the superpotential $\lambda_D$  are given by the next statement.

\begin{proposition}\label{derilambdawrtp}
The derivative $\lambda^{'}_D$ of the superpotential \eqref{superpotentialBNDN}  has the form 
\begin{equation}\label{derivativelambdaBNDN}
\lambda_D^{'}(p) = 2N p^{2m-1} \prod_{i=1}^d (p^2 - \xi_i^2)^{m_i -1} \prod_{i=1}^d ( p^2 - q_i^2),
\end{equation}
for some $q_i \in \mathbb{C}$. The second derivative $\lambda^{''}_D$ satisfies
\begin{equation}\label{lambda"(qi)}
\lambda_D^{''}(q_i)= 4 \epsilon_i N q_i^{2m} \prod_{a=1}^d (q_i^2 - \xi_a^2)^{m_a -1} \prod_{\substack{b=1 \\ b \neq i}}^d (q_i^2 - q_b^2),
\end{equation}
where $\epsilon_i =1$ if $q_i \neq 0$ and $\epsilon_{i}= \frac{1}{2}$ if $q_i=0$.  
\end{proposition}

Note that $q_i=0$ for some $i$  if and only if $m=0$. 

%From formula (\ref{superpotentialBNDN}) we have
%
%\begin{align*}
%\lambda_D^{'}(p)&= 2m p^{2m-1} \prod_{i=1}^n ( p^2 - \xi_i^2)^{m_i} + 2 p^{2m+1} %\sum_{j=1}^n m_j (p^2 - \xi_j^2)^{m_j -1} \prod_{\substack{i=1 \\ i \neq j}}^n (p^2 - \xi_i^2)^{m_i} \\
%& = 2N p^{2m-1}  \prod_{i=1}^n ( p^2 - \xi_i^2)^{m_i-1} Q(p),
%\end{align*}
%
%where $Q \in \mathbb{C}[p]$. We have that $\operatorname{deg}\lambda_D^{'}= 2N-1$ and thus $\operatorname{deg}Q= 2n$. Let us define canonical coordinates $u_i= \lambda_D(q_i)= \lambda_D(-q_i)$, $i=1, \dots, n$ such that $\lambda_D^{'}( \pm q_i)=0$. Then we must have that $Q(p)= \prod_{i=1}^n (p^2 - 	q_i^2)$ and the statement follows.
%\end{proof}

Let us define coordinates $u_i= \lambda_D(q_i)$, $i=1, \dots, d$ on $D$, where
$\lambda_D$ is given by formula \eqref{superpotentialBNDN},
 which we later show to be the canonical coordinates on the stratum \eqref{stratumBDNsuper}.

\begin{lemma}\label{relationlambdaatp=pmqj}
We have the following relation for the superpotential \eqref{superpotentialBNDN}:
\begin{equation}
\left. \partial_{u_i} \lambda_D(p) \right|_{p=q_j} = \delta_{ij}
\end{equation}
for any $i,j=1, \ldots, d$.
\end{lemma}

%\begin{proof}
The proof is similar to the proof of Lemma \ref{superpotentialdelta}.
%By definition we have
%
%\begin{equation*}
%\delta_{ij}= \frac{\partial u_j}{\partial u_i}= \partial_{u_i} \lambda_D(\pm %q_j). 
%\end{equation*}
%
%Then considering the Taylor expansion of $\lambda_D$ centered at $p = \pm q_j$ we have
%
%\begin{equation*}
%\lambda_D(p)= \lambda_D(\pm q_j) + \left. \lambda_D^{'}(p) \right|_{p=\pm q_j} %(p \pm q_j) + \mathcal{O}, 
%\end{equation*}
%
%where $\mathcal{O}$ denotes the rest of the terms in the series. Then
%
%\begin{equation}\label{partialu_i lambda}
%\partial_{u_i}\lambda_D(p) = \partial_{u_i} \lambda_D(\pm q_j) + \partial_{u_i} \big(\left.  \lambda_D^{'}(p) \right|_{p=\pm q_j} \big) (p \pm q_j) \mp \left. \lambda_D^{'}(p) \right|_{p=\pm q_j} \frac{\partial q_j}{\partial u_i} + \widetilde{\mathcal{O}},
%\end{equation}
%%
% $\widetilde{\mathcal{O}}$ denotes the $\partial_{u_i}$ derivative of $\mathcal{O}$. Since $ \left. \lambda_D^{'}(p) \right|_{p = \pm q_j} = \left. \widetilde{\mathcal{O}}\right|_{p= \pm q_j} =0$, restricting equality (\ref{partialu_i lambda}) to $p = \pm q_j$ the statement follows. 
%\end{proof}

%The following formula which follows from Proposition \ref{derilambdawrtp} will be useful below:

\begin{lemma}\label{partiallambaBDN} We have the relation
\begin{equation}\label{comparisonofpartialuilambda}
\partial_{u_i} \lambda_D(p) = \frac{2 \epsilon_i p}{p^2 -q_i^2} \frac{\lambda_D^{'}(p)}{\lambda_D^{''}(q_i)}
\end{equation}
for any $i=1, \ldots, d$.
\end{lemma}

\begin{proof}

Let $U_i(p)= \frac{2 \epsilon_i p}{p^2 -q_i^2} \frac{\lambda_D^{'}(p)}{\lambda_D^{''}(q_i)}$. By Proposition \ref{derilambdawrtp}  we get
\begin{align}\label{Ui}
U_i(p)=\frac{p^{2m} \prod_{a=1}^d (p^2 - \xi_a^2)^{m_a-1} \prod_{\substack{b=1 \\ b \neq i}}^d  (p^2 -q_b^2)}{q_i^{2m} \prod_{a=1}^d (q_i^2 - \xi_a^2)^{m_a-1} \prod_{\substack{b=1 \\ b \neq i}}^d  (q_i^2 - q_b^2)} , 
\end{align}
with $\operatorname{deg} U_i(p)= 2 N-2$. It  follows from the form \eqref{Ui} that $\left. U_i(p) \right|_{p=q_j}= \delta_{ij}$. 

Note that each of the functions $U_i(p)$ and $\partial_{u_i} \lambda_D(p)$ has the form of a product of an even polynomial of degree $2d-2$ and the function $p^{2m} \prod_{a=1}^d (p^2 - \xi_a^2)^{m_a-1} $. 
We also have that $\operatorname{deg} \partial_{u_i} \lambda_D(p)= 2 N-2$. It follows   by Lemma \ref{relationlambdaatp=pmqj} that $U_i = \partial_{u_i} \lambda_D(p)$.
\end{proof}

\begin{comment}
Starting from formula (\ref{superpotentialBNDN}) we obtain
%
\begin{equation}\label{mult p divide p}
\partial_{u_i} \lambda_D(p) = 2 p^{2m-1} \prod_{a=1}^n (p^2 - \xi_a^2)^{m_a -1} U(p), 
\end{equation}
%
where $ U \in \mathbb{C}[p]$ and $\operatorname{deg}U= 2n-1$.
% $\operatorname{deg} \partial_{u_i} \lambda_D = 2N-2= 2 ( m + \sum_{a=1}^n m_a -n ) + \operatorname{deg}U$

By Lemma \ref{relationlambdaatp=pmqj} we then have
%
\begin{equation}\label{widetildeU}
U(\pm q_j) = \pm \frac{\delta_{ij}}{2 q_j^{2m-1}\prod_{a=1}^n (q_j^2 - \xi_a^2)^{m_a -1}} .
\end{equation}
%
By Lagrange interpolation the points $(\pm q_i, U(\pm q_i))$, $i=1, \dots, n$ determine $U$ uniquely. Let $\pi(p)= \prod_{a=1}^n (p^2 -q_a^2)$. Then 
%
\begin{equation}\label{widetildeU}
U(p) = \frac{\pi(p)}{2 q_i^{2m-1}\prod_{a=1}^n (q_i^2 - \xi_a^2)^{m_a -1}} \Big[ \frac{1}{(p-q_i) \pi^{'}(q_i)} - \frac{1}{(p+q_i) \pi^{'}(-q_i)}\Big].
\end{equation}
%
 We note that $$\pi(p) = \lambda_D^{'}(p) (2N)^{-1} p^{1-2m} \prod_{a=1}^n (p^2 - \xi_a^2)^{1-m_a}$$ and $\lambda_D^{''}$ is an even polynomial. It follows that
%
\begin{equation*}
\pi^{'}(\pm q_i) = \pm \frac{\lambda_D^{''}(q_i)}{2N q_i^{2m-1} \prod_{a=1}^n (p^2 - \xi_a^2)^{m_a-1}}.
\end{equation*}
%
We substitute $\pi^{'}(\pm q_i)$ in formula (\ref{widetildeU}) and the statement follows.
\end{proof}
\end{comment}

Next we determine the Jacobi matrix between the coordinates $\xi_i$ and $u_i$. 

\begin{lemma}\label{JacobimatrixBDN}
We have
\begin{equation}
\partial_{u_i} \xi_a =  \frac{2 \epsilon_i \xi_a}{q_i^2 - \xi_a^2} \frac{1}{\lambda_D^{''}(q_i)},
\end{equation}
for any  $i, a=1, \dots, d$. 
\end{lemma}

\begin{proof}
From the defining formula (\ref{superpotentialBNDN}) we get 
%
%\begin{equation*}
%\partial_{u_i} \lambda_D(p)= -2 p^{2m} \sum_{j=1}^n \prod_{\substack{a=1 \\ a \neq j}}^n (p^2 - %\xi_a^2)^{m_a}  (p^2 - \xi_j^2)^{m_j -1}m_j \xi_j \partial_{u_i} \xi_j, 
%\end{equation*}
%
%which we express as
%
\begin{equation}\label{comparison1ofpartialuilambda}
\partial_{u_i} \lambda_D(p) = -2 \sum_{j=1}^d \frac{\lambda_D(p)}{p^2 - \xi_j^2} m_j \xi_j \partial_{u_i} \xi_j. 
\end{equation}
By Lemma \ref{partiallambaBDN} we obtain the following identity from \eqref{comparison1ofpartialuilambda} for any $a$:
\begin{equation}\label{identityBNDN}
 \frac{2\epsilon_i p}{(p^2 - q_i^2)(p^2 - \xi_a^2)^{m_a-1}} \frac{\lambda_D^{'}(p)}{\lambda_D^{''}(q_i)}=-2 \sum_{j=1}^d \frac{\lambda_D(p)}{(p^2 - \xi_j^2)(p^2 - \xi_a^2)^{m_a-1}} m_j \xi_j \partial_{u_i} \xi_j.
\end{equation}
Note that by L'Hospital's rule
\beq{LHop}
\left.\frac {\lambda_D(p)}{(p^2-\xi_a^2)^{m_a}}\right|_{p=\xi_a}  = \left.\frac{\lambda_D'(p)}{(p^2-\xi_a^2)^{m_a-1}}\right|_{p=\xi_a} \frac{1}{2 m_a \xi_a}.
\eeq
The statement follows by substituting $p= \xi_a$ in the relation (\ref{identityBNDN}) and applying \eqref{LHop}.

%We then consider the Taylor expansion of $\lambda_D$ centered at $p= \xi_k$ observing that $\lambda_D^{(r)}(\xi_k) =0$, $r=1, \dots, m_k-1$. Finally we substitute $p= \xi_k$ in the identity (\ref{identityBNDN}) and we obtain
%
%\begin{equation*}
%\frac{ 2\epsilon_i \xi_k }{(\xi_k^2 - q_i^2) \lambda_D^{''}(q_i)} \left. \frac{\lambda_D^{'}(p)}{(p^2 - \xi_k^2)^{m_k-1}}\right|_{p=\xi_k}=- \frac{\lambda_D^{(m_k)}(\xi_k)}{(m_k -1)! (2 \xi_k)^{m_k-1}} \partial_{u_i} \xi_k ,
%\end{equation*}
%
%which implies the statement. 

%Using that 
%
%\begin{equation}\label{ratiolmabdaprime}
%\left. \frac{\lambda_D^{'}(p)}{(p^2 - \xi_k^2)^{m_k-1}}\right|_{p=\xi_k} =%\frac{1}{(m_k -1)!} \frac{\lambda_D^{(m_k)}(\xi_k)}{(2 \xi_k)^{m_k-1}},
%end{equation}
\end{proof}

Let us now show that the critical values $u_i=\lambda_D(q_i)$, ($i=1, \dots, d$)  with the superpotential \eqref{superpotentialBNDN} play a role of the canonical coordinates for the metric and multiplication  (\ref{saitoAn}), (\ref{multAn}) on the stratum   \eqref{stratumBDNsuper}. More precisely, the following statement holds.

\begin{proposition}\label{criticalvBDN}
For any $i,j=1, \ldots, d$ the following formulae hold:
\begin{align}
\eta_D(\partial_{u_i}, \partial_{u_j}) &= \frac{2\epsilon_i \delta_{ij}}{\lambda_D^{''}(q_i)},
\nonumber
\\
\partial_{u_i} \circ \partial_{u_j} &= \delta_{ij} \partial_{u_j} \label{onDst}.
\end{align}
In particular, the metric $\eta_D$ is non-degenerate.
\end{proposition}

\begin{proof}
Let us consider the formula  \eqref{saitoAn} with $\zeta_1 = \partial_{u_i}$, $\zeta_2 = \partial_{u_j}$. Note that the right-hand side is non-singular at $p=\xi_i$ for all $i=1, \ldots, d$.
Furthermore, the residue at $q_l$ ($1 \leq l \leq d$) is zero if $i \neq j$ by Lemma \ref{partiallambaBDN}, and hence
$\eta_D(\partial_{u_i}, \partial_{u_j})=0$. Further on, by Lemma \ref{partiallambaBDN} we  have
\begin{align*}
\eta_D (\partial_{u_i}, \partial_{u_i}) = \sum_{p_s: \lambda^{'}_D(p_s)=0} {\rm res}_{p=p_s} \frac{(\partial_{u_i}\lambda_D(p))^2}{\lambda^{'}_D(p)}dp = \frac{4 \epsilon_i^2}{\lambda^{''}_D(q_i)^2} \sum_{p_s: \lambda^{'}_D(p_s)=0} {\rm res}_{p=p_s} \frac{p^2 \lambda^{'}_D(p)}{(p^2-q_i^2)^2}dp.
\end{align*}
It  follows from Proposition \ref{derilambdawrtp} that
\begin{align*}
\eta_D (\partial_{u_i}, \partial_{u_i})= \frac{8 N \epsilon_i^2 }{\lambda^{''}_D(q_i)^2}\sum_{p_s: \lambda^{'}_D(p_s)=0} {\rm res}_{p=p_s} \frac{p^{2m+1} \prod_{j=1}^d (p^2-\xi_j^2)^{m_j -1 } \prod_{\substack{j=1 \\ j \neq i}}^d (p^2-q_j^2)}{p^2-q_i^2}dp=\frac{2 \epsilon_i}{\lambda^{''}_D(q_i)},
\end{align*}
as required. 

Let us now consider formula \eqref{multAn} with $\zeta_1 = \partial_{u_i}$, $\zeta_2 = \partial_{u_j}$ and $\zeta_3 = \partial_{u_k}$. In the case when at least two of the indices $i, j, k$ are different the residues at $q_l$ ($1 \leq l \leq d$) are zero by Lemma \ref{partiallambaBDN}, and hence $\eta_{D}(\partial_{u_i} \circ \partial_{u_j}, \partial_{u_k})=0$. Further on, by Lemma \ref{partiallambaBDN} we have

\begin{align*}
\eta_D ( \partial_{u_i} \circ \partial_{u_i}, \partial_{u_i}) =  \sum_{p_s: \lambda^{'}_D(p_s)=0} {\rm res}_{p=p_s}\frac{(\partial_{u_i}\lambda_D(p))^3}{\lambda^{'}_D(p)}dp = \frac{8 \epsilon_i^3}{\lambda^{''}_D(q_i)^3} \sum_{p_s: \lambda^{'}_D(p_s)=0} {\rm res}_{p=p_s} \frac{p^3(\lambda^{'}_D(p))^2}{(p^2-q_i^2)^3}dp.
\end{align*}
It  follows by Proposition \ref{derilambdawrtp}   that
\begin{align*}
\eta_D (\partial_{u_i} \circ \partial_{u_i} , \partial_{u_i})= \frac{32 N^2 \epsilon_i^3}{\lambda^{''}_D(q_i)^3}\sum_{p_s: \lambda^{'}_D(p_s)=0} {\rm res}_{p=p_s} \frac{p^{4m+1}\prod_{j=1}^d (p^2-\xi_j^2)^{2(m_j -1) } \prod_{\substack{j=1 \\ j \neq i}}^d (p^2-q_j^2)^2}{p^2-q_i^2}dp.
\end{align*}
Therefore 
$$
\eta_D (\partial_{u_i} \circ \partial_{u_i} , \partial_{u_i})=\frac{2 \epsilon_i}{\lambda^{''}_D(q_i)},
$$ 
and hence
\begin{align*}
\eta_{D}( \partial_{u_i} \circ \partial_{u_j} , \partial_{u_k} ) = \frac{2 \epsilon_i \delta_{ij} \delta_{jk}}{\lambda^{''}_D(q_i)}.
\end{align*}
This implies formula \eqref{onDst} since $\partial_{u_i} \circ \partial_{u_j}  \in T_*D$ by Proposition \ref{multonD} and as $\eta_D$ is non-degenerate by above. 

\end{proof}

The following lemmas will be useful below. 

\begin{lemma}\label{prodxa^2-qi^2} We have 
\begin{equation}
\prod_{i=1}^d (\xi_a^2 -q_i^2) = \frac{m_a}{N} \xi_a^2 \prod_{\substack{b=1 \\ b \neq a}}^d ( \xi_a^2 - \xi_b^2)
\end{equation}
for any $a=1, \ldots, d$.
\end{lemma} 

\begin{proof}
Proposition \ref{derilambdawrtp} implies that
\begin{equation}
\label{form1a}
\prod_{i=1}^d (\xi_a^2 -q_i^2) = \left. \frac{\lambda_D^{'}(p)}{2N p^{2m-1} (p^2 - \xi_a^2)^{m_a-1}}\right|_{p=\xi_a} \prod_{\substack{b=1 \\ b \neq a}}^d (\xi_a^2 -\xi_b^2)^{-m_b +1}. 
\end{equation}
It follows from the defining formula \eqref{superpotentialBNDN} for $\lambda_D$  that
\begin{equation}
\label{form1b}
\left. \frac{\lambda_D^{'}(p)}{(p^2 - \xi_a^2)^{m_a-1}}\right|_{p=\xi_a} = 2 m_a \xi_a^{2m+1} \prod_{\substack{b=1 \\ b \neq a}}^d (\xi_a^2 -\xi_b^2)^{m_b},
\end{equation} 
The statement follows from formulae \eqref{form1a}, \eqref{form1b}. 
\end{proof}

\begin{lemma}\label{prodqi^2}We have
\begin{equation*}
\prod_{a=1}^d q_a^2 = \frac{m}{N} \prod_{a=1}^d \xi_a^2. 
\end{equation*}
\end{lemma}

\begin{proof}
 The derivative $\lambda_D^{'}$ can be expressed as
\begin{equation*}
\lambda_D^{'}(p) = 2m p^{2m-1} \prod_{a=1}^d (p^2 - \xi_a^2)^{m_a} + p^{2m} \frac{d}{dp}  \prod_{a=1}^d (p^2 - \xi_a^2)^{m_a} .
\end{equation*}
Let us equate this expression to the right-hand side of the formula  (\ref{derivativelambdaBNDN}) and divide both sides by $p^{2m-1}$. Then we substitute $p=0$ and obtain 
\begin{equation*}
2N \prod_{a=1}^d (-\xi_a^2)^{m_a -1} \prod_{a=1}^d (-q_a^2) = 2m \prod_{a=1}^d (-\xi_a^2)^{m_a}, 
\end{equation*}
which implies the statement.
\end{proof}

\begin{lemma}\label{zratio}Let 
\begin{equation}
\label{defz}
z= \frac{\prod_{i=1}^d \lambda_D^{''}(q_i)}{\prod_{\substack{i=1 \\ i <j}}^d (q_i^2 - q_j^2)^2} . 
\end{equation}
Then 
\begin{equation*}
z= c \prod_{a=1}^d \xi_a^{2 (m + m_a -1)} \prod_{\substack{a, b=1 \\ a \neq b}}^d (\xi_a^2 - \xi_b^2)^{m_a -1},
\end{equation*}
where 
\begin{align}\label{cconstant}
c =(-1)^{d ( \sum_{a=1}^d m_a - \frac{(d+1)}{2})} \epsilon 4^dN^{2d - m - \sum_{a=1}^d m_a} \prod_{a=1}^d m_a^{m_a -1} m^m, 
\end{align}
and $\epsilon = \prod_{i=1}^d \epsilon_i$.
\end{lemma}

\begin{proof}
It follows from formula \eqref{lambda"(qi)}
that 
%next line keep in PhD/ not in paper
\begin{align*}
z= (-1)^{\frac{d(d-1)}{2}}  \epsilon (4N)^d \prod_{i=1}^d q_i^{2m} \prod_{a, i=1}^d (q_i^2 - \xi_a^2)^{m_a -1}. 
\end{align*}
The statement follows by Lemmas \ref{prodxa^2-qi^2}, \ref{prodqi^2}.
\end{proof}

\begin{theorem}\label{etaDBDN}
The determinant of the metric $\eta_D$ given by \eqref{saitoAn}  with  the superpotential \eqref{superpotentialBNDN} in the coordinates $\xi_i$, $1 \leq i \leq d$, has the form
\begin{equation}\label{restrctedsaitodetBNDNcasemneq0}
\operatorname{det} \eta_D(\xi) = \kappa \prod_{i=1}^d \xi_i^{2 (m_i +m)} \prod_{ 1 \leq i < j \leq d} (\xi_i^2 - \xi_j^2)^{m_i + m_j},
\end{equation}
where $\kappa= (-1)^{d^2 + d (N-m) + \sum_{i=1}^{d-1} i m_{i+1}} 2^d m^m  N^{-N} \prod_{a=1}^d m_a^{m_{a}+1}$.
\end{theorem}

\begin{proof}
In the coordinates $\xi_i$ the determinant of the metric $\eta_D$ by Proposition \ref{criticalvBDN} has the form 
\begin{equation}\label{detetaDBDN}
\operatorname{det} \eta_D(\xi) =\epsilon (\operatorname{det}A)^{-2} \prod_{i=1}^d \frac{2}{\lambda_D^{''}(q_i)},
\end{equation}
where $A$ is the Jacobi matrix $(\partial_{u_i} \xi_a)_{i,a=1}^d$. By Lemma \ref{JacobimatrixBDN} we have
\begin{equation}\label{detetaDBDN1}
\operatorname{det}A = (-2) ^{d}\epsilon \prod_{a=1}^d \xi_a  \prod_{i=1}^d (\lambda_D^{''}(q_i))^{-1} \operatorname{det}B,
\end{equation}
where matrix $B=( \frac{1}{\xi_a^2 - q_i^2})_{i, a=1}^d$. The determinant of the matrix $B$ is a Cauchy's determinant which can be expressed as
\begin{equation*}
\operatorname{det}B= (-1)^{\frac{d(d-1)}{2}} \frac{\prod_{\substack{i=1 \\ i <j}}^d (\xi_i^2 - \xi_j^2)(q_i^2 - q_j^2)}{\prod_{i, a=1}^d (\xi_a^2 -q_i^2)}.
\end{equation*}
Hence by Lemma \ref{prodxa^2-qi^2} $\operatorname{det}B$ can be expressed  as 
\begin{equation}\label{detetaDBDN2}
\operatorname{det}B= \frac{N^d}{\prod_{a=1}^d m_a \xi_a^2 }\frac{\prod_{\substack{i=1 \\ i <j}}^d(q_i^2 - q_j^2)}{\prod_{1 \leq a < b \leq d} (\xi_a^2 - \xi_b^2)}.
\end{equation}
Combining formulae \eqref{detetaDBDN} -- \eqref{detetaDBDN2}  we get that
\begin{equation}
\operatorname{det} \eta_D = \epsilon^{-1}(2N^2)^{-d} \prod_{a=1}^d m_a^2 \xi_a^2  \prod_{1 \leq a < b \leq d} (\xi_a^2 - \xi_b^2)^2  z
\end{equation}
with $z$ defined by \eqref{defz}.  By Lemma  \ref{zratio}
the statement follows since 
$$
(-1)^{\sum_{i=1}^{d-1} i m_{i+1} - \frac{d(d-1)}{2}} \epsilon^{-1} (2N^2)^{-d} c \prod_{a=1}^d m_a^2  = \kappa.
$$ 
%
%
%\begin{equation}
%\operatorname{det}\eta_D = K \prod_{a=1}^n \xi_a^{2(m+m_a)} \prod_{1 \leq a < b \leq b}^n (\xi_a^2 - \xi_b^2)^{m_a + m_b}, 
%\end{equation}
%
%where $K= (-1)^{n^2 + n \sum_{a=1}^n m_a + \sum_{i=1}^{n-1} i m_{i+1}} m^m N^{-N} \prod_{a=1}^n m_a^{m_{a}+1}$.
\end{proof}

\section{Classical series: Theorem \ref{thma1}}\label{classicalsection2} 

In this section we show that Theorem \ref{thma1} holds for the root systems $A_n$, $B_n$ and $D_n$. Thus we complete the proof of Theorem \ref{MMtheorem} in these cases.

\begin{theorem}
Suppose $\mathcal{R}=A_n$. Then the statement of Theorem \ref{thma1} is true.
\end{theorem}

\begin{proof}
Let $S \subset A_n$ be a collection of roots such that the discriminant stratum $D= \cap_{\gamma \in S }\Pi_\gamma$ is given by equations \eqref{stratumANsuper}. Let $\xi_0, \dots, \xi_d$ be the corresponding functions on $D$ (see \eqref{stratumANsuper}).

Let $\mathcal{R}_D$ be the root system $$\mathcal{R}_D=\langle S \rangle \cap A_n=\{\alpha \in A_n |\left. \alpha \right|_D=0\}.$$ 
Then decomposition  \eqref{RDdecompo} of the root system $\mathcal{R}_D$ has the form $$\mathcal{R}_D= \bigsqcup_{i: m_i > 1} A_{m_i-1}.$$

Let $H$
 be a hyperplane in $D$ which is an element of the restricted arrangement ${\mathcal A}_D$. Then there exists a root $\beta \in A_n$ such that $H=\{v \in D| l(v)=0\}$, 
where   $l =l_H = \left. \beta\right|_D \in D^*$.  
 In the coordinates $\xi$ covector $l$ has the $l(\xi) =  \xi_a -\xi_b$ for some $0 \leq a <b \leq d$, and the corresponding root $\beta$ can be taken as $\beta= e_{m_0 + \dots + m_a} - e_{m_0 + \dots + m_b}$. 

The multiplicity $k_H$ of the linear form $ l(\xi)$ in the formula (\ref{restricteddethA_N}) is equal to $m_a+m_b$.
%Let us first consider the case when $m_a = m_b=1$. Then $\beta$ is defined uniquely by $l$. We have  $ \langle S ,\beta \rangle \cap A_N= \mathcal{R}_D \sqcup A_1$ where the last term $A_1 = \{ \pm \beta\}$. Therefore $m_a + m_b =2= h(A_1)=k_H$ as required. 
%Let us now suppose that $m_a > 1$, $m_b=1$. Then $$ \langle S ,\beta \rangle \cap A_N= 
%\bigsqcup_{\substack{i: m_i >1\\ i \neq a }} A_{m_i-1}\sqcup A_{m_a}, $$ where the last root %system $A_{m_a}$ contains $\beta$. Therefore $m_a + m_b=m_a+1= h(A_{m_a}) = k_H$ as required. 
%Finally, we consider the case when both $m_a, m_b > 1$. In this case
On the other hand decomposition   \eqref{decomposition} of the root system ${\mathcal R}_{D, \beta} =  \langle S ,\beta \rangle \cap A_n$ has the form 
 $${\mathcal R}_{D, \beta} = 
\bigsqcup_{\substack{i: m_i >1\\ i \neq a,b }} A_{m_i-1}\sqcup A_{m_a+m_b -1}, $$ where the last root system $A_{m_a + m_b -1}$ contains $\beta$. 
Therefore ${\mathcal R}_{D, \beta}^{(0)} = A_{m_a +m_b -1}$ and 
 $k_H= h(A_{m_a +m_b -1})$ as required. This completes the proof for the root system $A_n$. 
\end{proof}

\begin{theorem}
Suppose $\mathcal{R}=B_n$. Then the statement of Theorem \ref{thma1} is true.
\end{theorem}

\begin{proof}
 Let $S \subset B_n$ be a collection of roots such that the discriminant stratum $D=\cap_{\gamma \in S} \Pi_{\gamma}$ is given by equations \eqref{stratumBDNsuper}. Let $\xi_1, \dots, \xi_d$ be the corresponding coordinates on $D$ (see \eqref{stratumBDNsuper}).

Let $\mathcal{R}_D$ be the root system $$\mathcal{R}_D= \langle S \rangle \cap B_n= \{ \alpha \in B_n | \left. \alpha \right|_D=0\}.
$$ 
%and consider root system 
%$A_{m_i-1}$ with corresponding simple system
%
%\begin{align*}
%\varepsilon_j e_{j+l} - \varepsilon_{j+1} e_{j+1+l},  \quad \quad  \sum_{k=1}^{i-1} m_k + 1 \leq j \leq \sum_{k=1}^i m_k -1.
%\end{align*}
%
Note that if $l=0$, then $\mathcal{R}_D$ takes the form
\begin{align*}
\mathcal{R}_D = \bigsqcup_{i: m_i >1} A_{m_i-1},
\end{align*}
and
\begin{align}\label{mathcalR1}
\mathcal{R}_D = \bigsqcup_{i: m_i >1} A_{m_i-1} \sqcup \mathcal{R}^{(1)},
\end{align}
where $\mathcal{R}^{(1)}=B_l$ if $l \geq 2$,  $\mathcal{R}^{(1)}=A_1$, if $l=1$. 

 Let $\widehat{H}$ and $\widetilde{H}$ be hyperplanes in $D$ given by the kernels of the forms $\widehat{l}(\xi)= \widehat{\beta}|_D(\xi)=\xi_a$  ($1\le a \le d$) and $\widetilde{l}(\xi)= \widetilde{\beta}|_D(\xi)= \xi_a \pm \xi_b$ ($1\le a<b \le d$)  respectively, 
where one can choose the corresponding roots $\widehat{\beta}, \widetilde{\beta} \in B_n$
as follows: $$\widehat{\beta} = e_{l +m_1 + \dots + m_a}, \quad \quad \widetilde{\beta}= \varepsilon_{m_1 + \dots + m_a} e_{l+m_1 +\dots + m_a} \pm \varepsilon_{m_1 + \dots + m_b} e_{l+m_1 + \dots + m_b}.$$ 

Let $k_{\widehat H}=2(m_a+m)$ and $k_{\widetilde H}= m_a+m_b $  be the multiplicities of the linear forms $\widehat{l}(\xi)$ and $\widetilde{l}(\xi)$  respectively in the product \eqref{restrctedsaitodetBNDNcasemneq0}. Recall that $m=l$ in the superpotential \eqref{superpotentialBNDN} and Theorem \ref{restrctedsaitodetBNDNcasemneq0} for the $B_n$ strata.
%We choose corresponding roots $\widehat{\beta}, \widetilde{\beta} \in B_n$ such that 
%$\widehat{l}= \left. \widehat{\beta} \right|_D$ and $\widetilde{l}= \left. \widetilde{\beta}\right|_D$, as follows: $$\widehat{\beta} = e_{l +m_1 + \dots + m_a}, \quad \quad \widetilde{\beta}= \varepsilon_{m_1 + \dots + m_a} e_{l+m_1 +\dots + m_a} \pm \varepsilon_{m_1 + \dots + m_b} e_{l+m_1 + \dots + m_b}.$$ 
 
 Let us consider firstly the form $\widehat{l}(\xi)$. If $m_a=1$ then the root system ${\mathcal R}_{D, \widehat\beta} =  \langle S, \widehat{\beta} \rangle \cap B_n$ has decomposition into irreducible subsystems 
  $$
{\mathcal R}_{D, \widehat\beta}  =  \bigsqcup_{i: m_i >1} A_{m_i-1} \sqcup \mathcal{R}^{(2)},
$$ 
where $\mathcal{R}^{(2)}=A_1$ when $l=0$ and $\mathcal{R}^{(2)}= B_{l+1}$ when $l  \geq 1$. The root system $\mathcal{R}^{(2)}$ contains $\widehat{\beta}$
so $\mathcal{R}^{(2)} = {\mathcal R}_{D, \widehat\beta}^{(0)} $ in the decomposition \eqref{decomposition}, and $k_{\widehat{H}} = 2(m+1)= h(\mathcal{R}^{(2)})$ as required.
 
  If $m_a > 1$ then the root system ${\mathcal R}_{D, \widehat\beta} =  \langle S, \widehat{\beta} \rangle \cap B_n$ has decomposition into irreducible subsystems 
 \begin{align*}
 {\mathcal R}_{D, \widehat\beta}  =  \bigsqcup_{\substack{i: m_i >1\\ i \neq a }} A_{m_i-1}\sqcup  B_{l+m_a},
 \end{align*}
 where the root system $B_{l+m_a}$ contains $\widehat{\beta}$. Therefore 
$ {\mathcal R}_{D, \widehat\beta}^{(0)}=B_{l+m_a}$ and $k_{\widehat{H}} = h(B_{l+m_a})$ as required.

Let us now consider the form $\widetilde{l}(\xi)$. 
%
%Let us firstly suppose that $m_a=m_b=1$. Then $\widetilde{\beta}$ is defined uniquely by $\widetilde{l}$. We have that $\langle S, \widetilde{\beta}\rangle \cap B_N= \mathcal{R}_D \sqcup  A_1$ where the last term $A_1= \{ \pm\widetilde{\beta}\}$. Therefore $m_a + m_b =2= h(A_1)= k_{\widetilde{H}}$ as required. 
%
%
%Now let us suppose that $m_a > 1$ and $m_b=1$. Then $$\langle S, \tilde{\beta}\rangle \cap B_N=\bigsqcup_{\substack{ i: m_i >1 \\ i \neq a}} A_{m_i-1} \sqcup A_{m_a} \sqcup \mathcal{R}^{(1)}, $$ where the root system $A_{m_a}$ contains $\widetilde{\beta}$. Therefore $m_a+m_b=m_a+1= h(A_{m_a})=k_{\widetilde{H}}$ as required. 
%
Then $ {\mathcal R}_{D, \widetilde\beta}  = \langle S, \widetilde{\beta}\rangle \cap B_n$  has decomposition into irreducible subsystems
 $$
{\mathcal R}_{D, \widetilde\beta}  =\bigsqcup_{\substack{i: m_i >1\\ i \neq a,b }} A_{m_i-1} \sqcup  A_{m_a + m_b -1} \sqcup \mathcal{R}^{(1)},
$$ 
where the root system $A_{m_a + m_b -1}$ contains $\widetilde{\beta}$ and $\mathcal{R}^{(1)}$ is the same as in \eqref{mathcalR1}.   Therefore
${\mathcal R}_{D, \widetilde\beta}^{(0)} =   A_{m_a + m_b -1}$ and 
 $k_{\widetilde{H}} = h(A_{m_a + m_b -1})$ as required. 
\end{proof}

\begin{comment}
\textbf{\small Case} \pmb{$m = 0$}: Let $L \subset B_N$ be a collection of roots such that the stratum $D=\cap_{\gamma \in L} \Pi_{\gamma}$, $\operatorname{dim}D=n$ is given as in subsection \ref{BDzeroallowedsubsection}. Here we have that $\widetilde{L}= \langle L \rangle \cap B_N =\bigsqcup_{i \in I} A_{m_i-1}$, where $I=\{1, \dots, k\}$. 

Let $\beta= e_a$ and suppose that $a \notin I$. Then $U_\beta= A_1 \sqcup \widetilde{L}$ and therefore $m_\beta= h(A_1)$ which agrees with formula (\ref{BNcasem=0}) since $m_a=1$. Suppose that $ a \in I$. Then $U_\beta= B_{m_a} \sqcup (\bigsqcup_{i \in I \setminus \{a\}} A_{m_i-1})$. Thus $m_\beta= h(B_{m_a})=2m_a$ which agrees with (\ref{BNcasem=0}).  

Let $\beta= e_a \pm e_b$. Let us suppose that $a, b \notin I$. Then $U_\beta = A_1 \sqcup \widetilde{L}$ and therefore $m_\beta= h(A_1)=2=m_a +m_b$. Suppose that $a \in I$ and $b \notin I$. Then $U_\beta = A_{m_a} \sqcup (\bigsqcup_{i \in I \setminus \{a\}} A_{m_i-1})$. and hence $m_\beta= h(A_{m_a})=m_a +1$ which agrees with (\ref{BNcasem=0}) since $m_b=1$. Finally suppose that $a, b \in I$. Then $U_\beta= A_{m_a+m_b -1} \sqcup (\bigsqcup_{i \in I \setminus \{a,b\}} A_{m_i-1})$. Therefore $m_\beta= h(A_{m_a+m_b -1})=m_a +m_b$ which agrees with (\ref{BNcasem=0}). This concludes the proof for $B_N$. 
\end{comment}

\begin{theorem}
Suppose $\mathcal{R}=D_n$. Then the statement of Theorem \ref{thma1} is true. 
\end{theorem}

\begin{proof}
 Let $S \subset D_n$ be a collection of roots such that the discriminant stratum $D=\cap_{\gamma \in S} \Pi_{\gamma}$ is given by equations \eqref{stratumBDNsuper}, where $l \neq 1$. Let $\xi_1, \dots, \xi_d$ be the corresponding coordinates on $D$ (see \eqref{stratumBDNsuper}).

Let $\mathcal{R}_D$ be the root system $$\mathcal{R}_D= \langle S \rangle \cap D_n= \{ \alpha \in D_n | \left. \alpha \right|_D=0\}.$$ Note that 

\begin{align}\label{RDDN}
\mathcal{R}_D = \bigsqcup_{i: m_i >1} A_{m_i-1} \sqcup \mathcal{R}^{(1)},
\end{align}
where the root system $\mathcal{R}^{(1)}=D_l$ if $l \geq 3$,   $\mathcal{R}^{(1)}=A_1\times A_1$ if $l=2$, and $\mathcal{R}^{(1)}$ is empty if $l=0$.

Let $\widehat{H}$ and $\widetilde{H}$ be hyperplanes in $D$ given by the kernels of the forms $\widehat{l}(\xi)= \widehat{\beta}|_D(\xi)=\xi_a$  ($1\le a \le d$) and $\widetilde{l}(\xi)= \widetilde{\beta}|_D(\xi)= \xi_a \pm \xi_b$ ($1\le a<b \le d$)  respectively, for the corresponding root vectors $\widehat{\beta}, \widetilde{\beta} \in D_n$.

%as follows: $$\widehat{\beta} = e_{l +m_1 + \dots + m_a}, \quad \quad \widetilde{\beta}= \varepsilon_{m_1 + \dots + m_a} e_{l+m_1 +\dots + m_a} \pm \varepsilon_{m_1 + \dots + m_b} e_{l+m_1 + \dots + m_b}.$$ 

%We are interested in the multiplicities of the linear forms $\widehat{l}(\xi)= \xi_a$ ($1\leq a \leq n$) and $\widetilde{l}(\xi)= \xi_a \pm \xi_b$ $(1 \leq a < b \leq n)$ in \eqref{restrctedsaitodetBNDNcasemneq0}. We choose corresponding roots $\widehat{\beta}, \widetilde{\beta} \in D_N$ such that $\widehat{l}= \left. \widehat{\beta} \right|_D$ and $\widetilde{l}= \left. \widetilde{\beta}\right|_D$. Let $\widehat{H}$ and $\widetilde{H}$ be hyperplanes in $D$ given by the kernels of $\widehat{l}$ and $\widetilde{l}$ respectively.
 
Let $k_{\widehat H}=2(m_a+m)$ and $k_{\widetilde H}= m_a+m_b $  be the multiplicities of the linear forms $\widehat{l}(\xi)$ and $\widetilde{l}(\xi)$  respectively in the product \eqref{restrctedsaitodetBNDNcasemneq0}, where we assume that $k_{\widehat H}>0$ . Recall that $m=l-1$ in the superpotential \eqref{superpotentialBNDN} and Theorem \ref{restrctedsaitodetBNDNcasemneq0} for the $D_n$ strata.

 Let us consider firstly the form $\widehat{l}(\xi)$. This form has non-zero power in the formula \eqref{restrctedsaitodetBNDNcasemneq0} which implies that $l \geq 2$ or $m_a \geq 2$. In the former case we choose $\widehat{\beta}= e_l + e_{l+m_1 + \dots + m_a}$ and in the latter case we choose $$ \widehat{\beta}= \varepsilon_{m_1 + \dots + m_{a}-1} e_{l+ m_1 + \dots + m_{a}-1} + \varepsilon_{m_1 + \dots + m_a} e_{l+ m_1 + \dots + m_a}.
$$ 
If $m_a=1$ then $l \geq 2$ and we have that
 the root system ${\mathcal R}_{D, \widehat\beta} =  \langle S, \widehat{\beta} \rangle \cap D_n$ has decomposition into irreducible subsystems 
 $${\mathcal R}_{D, \widehat\beta} =  \bigsqcup_{i: m_i >1} A_{m_i-1} \sqcup D_{l+1},$$ where the root system $D_{l+1}$ contains $\widehat{\beta}$. Therefore $ {\mathcal R}_{D, \widehat\beta}^{(0)}=D_{l+1}$ and 
$$ k_{\widehat{H}} =2(m+1)=2l= h(D_{l+1})$$ as required. 
 
 If $m_a \geq 2$ then  the root system ${\mathcal R}_{D, \widehat\beta}$ has decomposition into irreducible subsystems  
 \begin{align*}
{\mathcal R}_{D, \widehat\beta} =  \bigsqcup_{\substack{i: m_i >1\\ i \neq a }} A_{m_i-1}\sqcup  \mathcal{R}^{(2)},
 \end{align*}
 where $\mathcal{R}^{(2)}= A_1 \times A_1$ if $m_a=2$ and $l=0$, and $\mathcal{R}^{(2)}= D_{l+m_a}$ if $l+m_a \geq 3$. The root system $\mathcal{R}^{(2)}$ contains $\widehat{\beta}$. If $\mathcal{R}^{(2)}= A_1 \times A_1$ then $ {\mathcal R}_{D, \widehat\beta}^{(0)}=A_1$ and 
$
 k_{\widehat{H}}  =
   2(m_a + l-1)= 2 = h ( A_1 )  
$
as required.
If $\mathcal{R}^{(2)}= D_{l+m_a} $ then $ {\mathcal R}_{D, \widehat\beta}^{(0)}= \mathcal{R}^{(2)}$ and 
$ k_{\widehat{H}}  =
   2(m_a + l-1)= h ( \mathcal{R}^{(2)})  
$
 as required. 

Let us now consider the form $\widetilde{l}(\xi)$. The root $\widetilde{\beta}$ can be chosen as 
\begin{align*}
\widetilde{\beta}= \varepsilon_{m_1 + \dots + m_a} e_{l+m_1 +\dots + m_a} \pm \varepsilon_{m_1 + \dots + m_b} e_{l+m_1 + \dots + m_b}.
\end{align*}
%
% Let us first suppose that $m_a=m_b=1$. Then $\widetilde{\beta}$ is defined uniquely by $\widetilde{l}$. We have that $\langle S, \widetilde{\beta}\rangle \cap D_N= \mathcal{R}_D \sqcup  A_1$ where the last term $A_1= \{ \pm\widetilde{\beta}\}$. Therefore $m_a + m_b =2= h(A_1)= k_{\widetilde{H}}$ as required. 
%
%
%Now let us suppose that $m_a > 1$ and $m_b=1$. Then $$\langle S, \tilde{\beta}\rangle \cap D_N=\bigsqcup_{\substack{ i: m_i >1 \\ i \neq a}} A_{m_i-1} \sqcup A_{m_a} \sqcup \mathcal{R}^{(1)}, $$ where the root system $A_{m_a}$ contains $\widetilde{\beta}$. Therefore $m_a+m_b=m_a+1= h(A_{m_a})=k_{\widetilde{H}}$ as required. 
%
Then the root system ${\mathcal R}_{D, \widetilde\beta} =  \langle S, \widetilde{\beta} \rangle \cap D_n$ has decomposition into irreducible subsystems 

$${\mathcal R}_{D, \widetilde\beta} = \bigsqcup_{\substack{i: m_i >1\\ i \neq a,b }} A_{m_i-1} \sqcup  A_{m_a + m_b -1} \sqcup \mathcal{R}^{(1)},$$ where the root system $A_{m_a + m_b -1}$ contains $\widetilde{\beta}$ and $\mathcal{R}^{(1)}$ is the same as in \eqref{RDDN}.  Therefore $ {\mathcal R}_{D, \widetilde\beta}^{(0)} = A_{m_a + m_b -1}$ and $k_{\widetilde{H}} = h(A_{m_a + m_b -1})$ as required. 
\end{proof}

\section{A general formula for the restricted Saito determinant}\label{generalformsection} Let us fix a basis $\Delta$ of simple roots $\alpha_k$ ($k=1, \dots, n$) for the root system $\mathcal R$. %We will find a general formula for the determinant of $\eta_D$.
These vectors have the form  $\alpha_k=(\alpha^{(1)}_k, \dots, \alpha^{(n)}_k)$ in the coordinate system $x^i$.
Let us define the corresponding directional partial derivatives operators 
\begin{equation}\label{defnpartial}
 \partial_{\alpha_k} = \sum_{i=1}^n \alpha_{k}^{(i)} \frac{\partial}{\partial x^{i}}.
\end{equation}
The basis of fundamental coweights  $\omega^{i} \in V$ ($i=1, \dots, n$) is defined by
\begin{equation}\label{fundamcow}
(\omega^{i}, \alpha_j)=\delta^{i}_j.
\end{equation}
Let us  also define a new coordinate system on $V$ given by 
\beq{xtilde}
\widetilde{x}^i=(\omega^{i}, x), \quad  i=1, \dots, n, \,\,  x \in V. 
\eeq

\begin{lemma}\label{lemmaforg}
In the coordinates $\widetilde{x}^{i}$, $1 \leq i \leq n$, we have $\frac{\partial}{\partial \widetilde{x}^{i}}= \partial_{\alpha_i}$.  
\end{lemma}

\begin{proof}
Let $x=(x^{1}, \dots, x^n)^{\intercal}$ and $\widetilde{x}=(\widetilde{x}^1, \dots, \widetilde{x}^n)^{\intercal}$. Then $\widetilde{x}^\intercal = \Omega x^{\intercal}$, where $\Omega$ is the $n \times n$ matrix $\Omega=(\Omega_{ij})_{i,j=1}^n$ with $\Omega_{ij} = \omega_{(j)}^i$ if $\omega^i = (\omega_{(1)}^i, \dots, \omega_{(n)}^i )$. Then $x^\intercal = \Omega^{-1} \widetilde{x}^\intercal$, and it is easy to see that the $(i, j)$-th entry of $\Omega^{-1}$ equals $\alpha^{(i)}_j$. Therefore $\frac{\partial}{\partial \widetilde{x}^{i}}=  \frac{\partial x^{k}}{\partial \widetilde{x}^{i}} \frac{\partial}{\partial x^{k}}=  \alpha^{(k)}_i \frac{\partial}{\partial x^k}= \partial_{\alpha_i}$. 
\end{proof}

For any set of (homogeneous) basic invariants $p^i\in {\mathbb C}[x]^W$, $i=1, \dots, n$, let $B_k(p)$ be the $(n-1) \times (n-1)$ matrix obtained from the Jacobi matrix $ (\partial_{\alpha_j} p^{i})_{i,j=1}^n$ by eliminating the $k$-th column and $n$-th row ($1\le k \le n$). Let 
\beq{Jk}
J_{k}(p)=J_k(p^1, \dots, p^{n-1})=\operatorname{det}B_k(p).
\eeq
 Note that $J_{k}(p)$ is a homogeneous polynomial in $x$ of degree $|\mathcal{R}_+|- h +1$, where $h$ is the Coxeter number.  
%, since the entries of the $n$-th row consist of homogeneous polynomials of degrees $h-1$, and $\operatorname{deg}J=|\mathcal{R}_+|$. 
Let also $J(p)$ be the Jacobian $J(p^1, \dots, p^n)=\operatorname{det}\big(\partial_{\alpha_j} p^i \big)_{i, j =1}^n$. 

The next statement will be useful below.
\begin{proposition}\label{yoshi}\cite{yoshinaga}
The vector field $\frac{\partial}{\partial p^n}$ can be represented as
\begin{align*}\label{identityyoshi}
\frac{\partial}{\partial p^n}=
J^{-1} (p)
\begin{vmatrix}
{\partial_{\alpha_1}  p^1} & \dots &{\partial_{\alpha_n}  p^1} \\
\vdots & \ddots & \vdots \\
{\partial_{\alpha_{1}}  p^{n-1}} & \dots & {\partial_{\alpha_{n}}  p^{n-1}}\\
{\partial_{\alpha_{1}}  }& \dots &{\partial_{\alpha_{n}} }
\end{vmatrix}.
\end{align*}
\end{proposition}

%The above proposition can be checked easily by applying left-hand side and right-hand side of equality \eqref{identityyoshi} to the polynomials $p^i$. Similarly, one can replace coordinates $x^i$ in the right-hand-side of \eqref{identityyoshi} with another coordinate system on $V$. This gives the following statement. 

We will write $J_k=J_k(t^1, \dots, t^{n-1})$ and $J=J(t^1, \ldots, t^n)$ for the basis $p^i=t^i$  of Saito polynomials. Similarly to Proposition \ref{yoshi} we have its following version.

\begin{proposition}\label{identity}
The  vector field  $e=\frac{\partial}{\partial t^n}$  can be represented as
\begin{equation}\label{identityfield}
e=    \sum_{i=1}^n e^i \partial_{\alpha_i}, \quad e^i = (-1)^{n+i} J^{-1} J_{i}.
\end{equation}
\end{proposition}

Formula \eqref{identityfield} can be checked by applying it to the polynomials $t^i$ (cf. \cite{yoshinaga}).
 
In the next statement we specify components of the contravariant Saito metric in terms of the identity field $e$.

\begin{proposition}\label{Saitom} In the coordinates $\widetilde{x}^{i}$ ($i=1, \dots, n$) the contravariant Saito metric $ \eta^{ij} \frac{\partial}{\partial \widetilde{x}^i}  \frac{\partial}{\partial \widetilde{x}^j}$ is given by
\begin{equation}
\eta^{i j} = (-1)^{n+1+j} \partial_{\omega^{i}} \frac{J_j}{J} + (-1)^{n+1+i}\partial_{\omega^{j}} \frac{J_i}{J}.
\end{equation}
\end{proposition}

\begin{proof}
For the Saito metric we have
\begin{align*}
\eta^{i j} =\mathcal{L}_e g^{ij} = - g^{k j} \partial_{\alpha_k} e^{i}  - g^{i k} \partial_{\alpha_k} e^{j} 
= - \partial_{ u^j} e^{i} - \partial_{ u^i} e^{j},
\end{align*}
where vector $u^{i} = g^{ik}\alpha_k$,  and $e^i$ is the $i$-th component of the vector field $e$. 

For the Euclidean metric $g$ in the coordinates $\widetilde{x}^{i}$ we have by Lemma \ref{lemmaforg} $$g_{ij}=(\partial_{\alpha_i}, \partial_{\alpha_j})= (\alpha^{(k)}_{i} \frac{\partial}{\partial x^{k}}, \alpha^{(l)}_{j} \frac{\partial}{\partial x^{l}})= \alpha^{(k)}_{i} \alpha^{(l)}_{j} \delta_{kl}= \sum_{k=1}^n  \alpha^{(k)}_{i} \alpha^{(k)}_{j} = (\alpha_{i}, \alpha_j).$$ 
  Therefore we have 
\begin{equation*}
(u^{i}, \alpha_j)= \sum_{k=1}^n g^{ik} (\alpha_k, \alpha_j)= g^{ik} g_{kj} = \delta^{i}_j.
\end{equation*}
Hence   $u^{i}=\omega^i$, and the result follows by Proposition \ref{identity}. 
\end{proof}
%
%\begin{lemma}\label{deri}
%We have $$\partial_{\omega^i} \prod_{\alpha \in \Delta} \alpha = \prod_{\alpha \in \Delta \setminus \{\alpha_i \}}\alpha, \quad i=1, \dots, n.$$
%\end{lemma}
%\begin{proof} We have
%
%\begin{align*}
%\partial_{\omega^i} \prod_{\alpha \in \Delta} \alpha=\sum_{\alpha \in \Delta} \frac{(\omega^i,\alpha)}{\alpha} \prod_{\alpha \in \Delta} \alpha= %\prod_{\alpha \in \Delta \setminus \{ \alpha_i\}} \alpha,
%\end{align*}
%by the definition of fundamental coweights.
%\end{proof}

To get the determinant of the restricted Saito metric $\eta_D$ we will need the following general result from linear algebra. 
We denote the set $\underline{n}=\{1, \dots, n\}$.

\begin{proposition}\cite{prasolov}\label{Prasolov}
Let A= $(a_{ij})_{i,j=1}^n$ be an $n \times n$ matrix and let $A_{ij}=(-1)^{i+j} M_{ij}$, where $M_{ij}$ is the $(i,j)$-th minor of $A$. Let  $\{i_1, \ldots, i_p\}$ be a subset of the set $\underline{n}$ of cardinality $p$, $1 \leqslant p   \leqslant n$.
%and let
%$
%\sigma=
 %\begin{pmatrix}
  %   i_1 &\dots & i_n \\
  % j_1 &\dots & j_n
 % \end{pmatrix}
%\in S_n$. 
Then
\begin{align*}
\begin{vmatrix}
A_{i_1 i_1} & \dots & A_{i_1 i_p} \\
\vdots & \ddots & \vdots\\
A_{i_p i_1} & \dots & A_{i_p i_p}
\end{vmatrix}
= (\operatorname{det}A)^{p-1}
\begin{vmatrix}
a_{i_{p+1},i_{p+1}} & \dots & a_{i_{p+1},i_{n}} \\
\vdots & \ddots & \vdots\\
a_{i_n, i_{p+1}} & \dots & a_{i_{n} j_n}
\end{vmatrix},
\end{align*}
%where ${\rm sign}(\sigma)$ is the sign of $\sigma$.
\end{proposition}
In the next statement we define a convenient coordinate system on the discriminant strata.

\begin{lemma} Let $I \subset \underline{n}$ be a subset of cardinality $|I|=k$, $1 \leq k <n$. Let $D= \cap_{q \in I} \Pi_{\alpha_q}$ be the corresponding discriminant stratum. Then $\widetilde{x}^{i}$, ${i \notin I}$ is a coordinate system on $D$. 
\end{lemma}

\begin{proof} Note that $D=\langle \omega^i : i \notin I \rangle$.  Let us consider a linear dependence of $\widetilde{x}^i$ ($i \notin I$) on $D$:
\beq{lindep}
\sum_{i \notin I} a_i \widetilde{x}^i =\sum_{i \notin I} a_i (\omega^i,x)=0,
\eeq
 where $a_i \in \mathbb{C}$. 
By considering relation \eqref{lindep} for $x = \omega^j$, $j\notin I$, we get from positive definiteness of the form $g$  that
%Then $\sum_{i \notin I} a_i (\omega^i, \omega^j)=0$ for all $j \notin I$. This is a system of $n-k$ linear equations and the matrix of this system is $\Omega=(\Omega_{ij})_{i, j \notin I}$, $\Omega_{ij}=(\omega^{i}, \omega^{j})$. Since $\omega^{i}$, ${i \notin I}$ are linearly independent the Gram matrix $\Omega$ is non-degenerate. Therefore the only solution to this system is the trivial one, 
$a_i=0$ for all $i$. 
\end{proof}

%Let us now fix basic invariants to be Saito polynomials. 

Now we can state the main theorem of this section which gives a formula for the determinant  of the restricted Saito metric.

\begin{theorem}\label{Saitovo}
Let $I= \{i_1, \dots, i_k\}\subset \underline{n}$, where $k=|I|<n$, 
%$1 \leq i_1 < \dots < i_k \leq n$ 
and let $D= \cap_{q \in I} \Pi_{\alpha_q}$. 
Let $P$ be the product 
\begin{align}\label{determinantgeneraltheorem00}
P = 
J^2
\begin{vmatrix}
\eta^{i_1i_1} & \dots & \eta^{i_1i_k} \\
\vdots & \ddots & \vdots\\
\eta^{i_k i_1 }& \dots & \eta^{i_ki_k}
\end{vmatrix},
\end{align}
where $\eta^{ij}$ are components of the contravariant metric $\eta$ in the coordinate system $\widetilde x^i$ ($i\in \underline{n}$). 
Then function $P$ has a well-defined limit $P_D = P|_D$, and 
 determinant $\operatorname{det}\eta_D$ of the restricted Saito metric $\eta_D$ in the coordinates $\widetilde{x}^{i}$ ($i \notin I$) is given as
\begin{align}\label{determinantgeneraltheorem}
\operatorname {det}\eta_D=- P_D.
\end{align}
%In particular, the right-hand-side of the expression \eqref{determinantgeneraltheorem} has a well-defined limit as one tends to $D$. 
\end{theorem}

\begin{proof}
Let us consider the covariant Saito metric in the flat coordinates $t^{i}$, ($1 \leq i \leq n$), 
\begin{align}
\eta&= \eta_{i j} dt^{i} dt^{j}=\sum_{i=1}^n  dt^{i} dt^{n+1 -i} 
=\sum_{i=1}^n  \sum_{r=1}^n \partial_{\widetilde{x}^r} t^i d\widetilde{x}^r \sum_{l=1}^n \partial_{\widetilde{x}^l} t^{n+1 -i} d\widetilde{x}^{l}. \label{covariantsaitonocoo}
\end{align}
 Note that for any $p \in \mathbb{C}[x]^W$ one has that
\begin{equation*}\label{inv}
\left.\partial_{\widetilde{x}^i}p(x) \right|_{\alpha_i=0}=\left. \partial_{\alpha_i} p(x) \right|_{\alpha_i=0} = 0.
\end{equation*}
Therefore, $\left. \partial_{\widetilde{x}^l}t^{i}\right|_D=0$ if $l\in I$. Hence, by restricting formula (\ref{covariantsaitonocoo}) on $D$, we get that the restricted Saito metric   is given by
\begin{align}
\eta_D&=
%\sum_{i=1}^n  \sum_{r \in \widehat{I}}\partial_{\widetilde{x}^r} t^i d\widetilde{x}^r \sum_{l \in \widehat{I}} \partial_{\widetilde{x}^l} t^{n+1 -i} d\widetilde{x}^{l}= 
 \sum_{r,l \in \widehat{I}} \eta_{r l} d\widetilde{x}^r d\widetilde{x}^{l} \label{restr},
\end{align}
where 
\beq{etapol}
\eta_{rl}= \sum_{i=1}^n (\partial_{\widetilde{x}^r} t^i)(\partial_{\widetilde{x}^l} t^{n+1-i})
\eeq
  and  $\widehat{I}=\underline{n}\setminus I$.

Let matrix $Q=(\eta^{ij})_{i,j=1}^n$. Then 
\begin{equation*}
\eta_{rl}= \frac{(-1)^{r+l}Q_{rl}}{\operatorname{det}Q}, \quad r, l=1, \dots, n,
\end{equation*}
where $Q_{rl}$ is the $(r,l)$-th minor of $Q$. Consider the matrix $C=(\eta_{rl})_{r,l \in \widehat{I}}$. It follows from formula (\ref{restr}) that  restriction $\left. C \right|_D$ is the matrix of $\eta_D$ and hence 
\beq{detmain}
\operatorname{det}\eta_D= \operatorname{det}\left. C \right|_D.
\eeq
 By Proposition \ref{Prasolov} applied to $A=Q$ and  $p= |\widehat{I}|=n-k$ %and $\sigma= \mathrm{Id}$ 
we have
\begin{equation}
\label{Pcalc}
\operatorname{det}C=  (\operatorname{det}Q)^{-(n-k)} (\operatorname{det}Q)^{n-k-1} \operatorname{det}Q_I=  \operatorname{det}Q^{-1} \operatorname{det}Q_I, 
\end{equation}
where $Q_I$ is the matrix $(\eta^{ij})_{i,j \in I}$. Since $\operatorname{det}Q^{-1}= \operatorname{det}\eta$, which is equal to $-J^2$,
it follows from \eqref{Pcalc} that $P=-\det C$, which is well-defined on $D$ since entries \eqref{etapol} of $C$ are polynomials. 
The statement follows from \eqref{detmain}. 
\end{proof}

Polynomials $J_k(p(x))$ given by \eqref{Jk} have certain vanishing property which is stated in the next proposition.
\begin{proposition}\label{divisible}
The polynomial $J_k(p)$ ($1\le k \le n$) is divisible by $\alpha(x) $ for all $\alpha \in \mathcal{R} \cap U $, where vector space $U = \langle \alpha_i : 1 \leq i \leq n , i \neq k \rangle$.
\end{proposition}

\begin{proof}
Any $\alpha \in  \mathcal{R} \cap U$ can be represented as $\alpha= \sum_{\substack{ i=1 \\ i \neq k}}^n c_i \alpha_i$ for some $c_i \in \mathbb{C}$. Consider the linear combination of columns of the matrix $B_k(p)$ where the $i$-th column is taken with the coefficient $c_i$. The $j$-th entry in the resulting column vector is equal to $\partial_{\alpha} p^{j}$, which vanishes if $\alpha(x)=0$. Hence $J_k(p)=\det B_k(p)$ also vanishes if $\alpha(x)=0$.
\end{proof}

We are going to show that the polynomial $J_k(p)$ is not identically zero. We will need the following lemma, which is a particular case of a more general statement established in \cite{KM}.

\begin{lemma}\label{identityBn} Let $\mathcal{R}=B_n$. Then the identity field $e$ takes the following form:
\begin{align*}
e= c \sum_{i=1}^n (x^i)^{-1} \prod_{\substack{ j=1 \\ j \neq i}}^n ( (x^i)^2- (x^j)^2)^{-1} \frac{\partial}{\partial x^i},
\end{align*}
where $c \in \mathbb{C}^\times$. 
\end{lemma}

\begin{proof}
By Proposition \ref{yoshi} the identity field $e$ is proportional to the field
$$
\frac{\partial}{\partial p^n}= J^{-1}(p) \sum_{i=1}^n (-1)^{n+i} J_{i}(p) \frac{\partial}{\partial x^i}.
$$ 
Basic invariants can be chosen as $$p^i = \sum_{j=1}^n (x^j)^{2i},
$$
where  $1 \leq i \leq n$. Therefore $J_i(p)$ is proportional to  $$\prod_{ \substack{ j=1 \\ j \neq i}}^nx^{j} \prod_{\substack{1 \leq l <k \leq n \\ l,k \neq i}} ((x^l)^2- (x^k)^2),$$
and the polynomial $J(p)$ has the form
$$
J(p) \sim \prod_{j=1}^n x^j \prod_{1 \leq l < k \leq n} ( (x^l)^2 - (x^k)^2).
$$ 
 The statement follows. 
\end{proof}

Lemma  implies that identity field $e$ for the Coxeter group $B_n$ is singular on every mirror of the group. This appears to be the general situation which we show in the next proposition.

Let us numerate simple roots $\Delta$ by a bijection $\sigma\colon  \{1, \dots, n\} \rightarrow \Delta$ given by $\sigma(i) = \alpha_i$.
% Proposition \ref{identity}  can also be restated as follows.
We will denote 
\beq{Jalpha}
J_\alpha:= J_{\sigma^{-1}(\alpha)},\quad  \alpha \in \Delta. 
\eeq
By 
 Proposition \ref{identity}  we can represent the identity field as 
\begin{equation}\label{identityfieldrestated}
e= \sum_{\alpha \in \Delta} e^\alpha \partial_\alpha,
\end{equation}
 where 
 \begin{equation}\label{notationJ_k}
 e^\alpha= \epsilon_\alpha \frac{J_\alpha}{J}, \quad \epsilon_\alpha = (-1)^{n+\sigma^{-1}(\alpha)}.
 \end{equation}
% with  the sign  $\epsilon_\alpha := (-1)^{n+\sigma^{-1}(\alpha)}$.

\begin{proposition}\label{singulare} The identity field $e$ is singular on every hyperplane of the discriminant of $W$.
\end{proposition}

\begin{proof}
Recall that $\operatorname{deg}J=|\mathcal{R}_+|$ and $\operatorname{deg}J_\alpha= |\mathcal{R}_+|-h+1$ for any $\alpha \in \Delta$. By formula \eqref{notationJ_k} $e^{\alpha}$ is a rational function of degree $1-h$. 
Since vector field $e$ is non-zero it has to be singular on $\Pi_\beta$ for some $\beta \in \mathcal{R}$. By $W$-invariance it is also singular on $\Pi_{w(\beta)}$ for any $w \in W$.
This proves the statement in the simply-laced case.

%In the it follows that $e$ is singular on $\Pi_\beta$ for all $\beta \in \mathcal{R}$.

Let us now consider the case where $W$ has two orbits on $\mathcal{R}$. If ${\mathcal R}=B_n$ then the statement follows by Lemma \ref{identityBn}.  Let us now assume that the root system $\mathcal{R}=F_4$.
% and let $\Delta \subset \mathcal{R}$ be a simple system \cite{cox}. 
Let $\alpha \in \Delta$ be such that component $e^\alpha \ne  0$.  Let us show that $e$ is singular on the mirror $\Pi_\alpha$.

%Recall that the corresponding Coxeter number is $h=12$.  
Suppose that $e$ is non-singular on the hyperplane $\Pi_\alpha$. By Proposition \ref{divisible} $J_\beta$ is divisible by $\alpha(x)$ for any $\beta \in \Delta \setminus \{ \alpha\}$. Hence $e^\beta$ is non-singular on $\Pi_\alpha$, and therefore $J_\alpha$ must  be divisible by $\alpha(x)$. It follows that 
\begin{align}\label{proofsingulare}
\prod_{\beta \in \Delta} \beta \mid J_\alpha.
\end{align}
Since $e^\alpha$ is singular on at least $h-1=11$ different hyperplanes inside the discriminant and there are $12$ short and $12$ long positive roots, it follows from \eqref{proofsingulare} that $e^\alpha$ has singularities on hyperplanes from both orbits. It follows by $W$-invariance of $e$ that $e$ is singular on $\Pi_\alpha$, which is a contradiction.

Thus we  have that $e$ is singular on $\Pi_\alpha$, and hence $e$ is singular on $\Pi_\beta$ for all $\beta\in W\alpha\cap \Delta$. It follows by Proposition \ref{divisible}  that $e^\beta \neq 0$.
Let us now assume that simple roots $\Delta=\{\alpha_1, \alpha_2, \alpha_3, \alpha_4\}$ correspond to the Dynkin diagram below
$$
\dynkin[labels={1,2,3, 4},label macro/.code={\alpha_{#1}}]{F}{4}
$$
% Let us now assume that $\Delta=\{\alpha_1, \alpha_2, \alpha_3, \alpha_4\}$, where $$\alpha_1= e_2-e_3, \quad \alpha_2= e_3-e_4, \quad \alpha_3=e_4, \quad \alpha_4= \frac{1}{2} (e_1 -e_2 -e_3 -e_4).$$ 
If $\alpha$ is a long root then it follows from the previous that $J_{\alpha_1} \neq 0$. By Proposition \ref{divisible} $J_{\alpha_1}$ is divisible by $\alpha_2 \alpha_3 \alpha_4 (\alpha_2 + 2\alpha_3)$. Since this product has two long roots and two short roots and $\operatorname{deg}e^{\alpha_1}=-11$ it follows that $e^{\alpha_1}$ has singularities on both orbits. If $\alpha$ is a short root then it follows by the previous that $J_{\alpha_4} \neq 0$. By Proposition \ref{divisible} $J_{\alpha_4}$ is divisible by $\alpha_1 \alpha_2 \alpha_3 (\alpha_2 +\alpha_3)$. Since this product has two long and two short roots it follows  that $e^{\alpha_4}$ is singular on both orbits. The statement for ${\mathcal R} = F_4$ follows by $W$-invariance of $e$.

%We are going to show that  components $e^\beta \ne 0$ for any $\beta\in W\alpha\cap \Delta$.

Let us now consider  the dihedral  case $\mathcal{R}=I_2(2m)$, $m \geq 3$. By Proposition \ref{yoshi} the identity field $e$ has the form
$$
e= J^{-1}  ( x^1 \frac{\partial}{\partial x^2} - x^2 \frac{\partial}{\partial x^1}),
$$ and the statement follows. 
\end{proof}

Proposition \ref{singulare}  has the following implication which was obtained in \cite{yoshinaga} by another argument.

\begin{corollary} The polynomial $J_k$ is not identically zero on the hyperplane $\Pi_{\alpha_k}$ for any $k=1, \ldots, n$. In particular, $J_k$ is not a zero polynomial. 
\end{corollary}

\begin{proof}
By Proposition \ref{singulare}  the identity field $e$ is singular on the hyperplane  $\Pi_{\alpha_k}$. By Proposition \ref{divisible} the polynomial $J_i$ is divisible by $\alpha_k(x)$ for all $i\ne k$ ($1\le i \le n$). It follows that $J_k$ is not divisible by $\alpha_k(x)$. 
\end{proof}

In the rest of this section we study properties of polynomials $J_k$ ($1\le k\le n$). 
%The  vanish on the intersections of two mirrors as we see in  the following statement.

\begin{proposition}\label{structureJ}
Let $\beta, \gamma \in \mathcal{R}$, $\beta \neq \pm \gamma$. Then $J_k$ vanishes  on the stratum $D= \Pi_\beta \cap \Pi_\gamma$ . 
\end{proposition}

\begin{proof}
There exists a non-trivial linear combination of $\beta$ and $\gamma$ such that
\begin{align*}
a_1 \beta + a_2 \gamma = \sum_{\substack{i=1 \\ i \neq k}}^n b_i \alpha_i, 
\end{align*}
where $a_1, a_2, b_i \in \mathbb{C}$. Note that $\left. \partial_\beta p \right|_D= \left. \partial_\gamma p \right|_D = 0$ for any    $p\in {\mathbb C}[x]^W$. Hence the linear combination of columns of the matrix $B_k(t)$ with coefficients $b_i$ vanishes on $D$. 
\end{proof}

For a subset 
$\{i_1, \ldots, i_m\}\subset \underline{n}$ of size $m$ 
let us denote the corresponding discriminant stratum as $ D_{i_1, \dots, i_m}=D_{\alpha_{i_1}, \dots, \alpha_{i_m}}$, that is 
\begin{align}\label{generalDsimple}
D_{i_1, \dots, i_m}=  \cap_{j=1}^m \Pi_{\alpha_{i_j}}.
\end{align}

%Using Propositions \ref{structureJ} and \ref{arrangement} 
Proposition \ref{structureJ} has the following implication  on the structure of the polynomial  $J_{k}$ in terms of defining polynomial of the arrangement  ${\mathcal A}_{D_k}$. Note that it is also stated in \cite[Remark 5]{yoshinaga}.

\begin{corollary}\label{factorJ}
There is a proportionality $\left. J_k \right|_{D_k}\sim {\mathcal I}( \mathcal{A}_{D_k})$.
\end{corollary}

\begin{proof}
By Proposition \ref{arrangement}  $\operatorname{deg}J_k=|{\mathcal A}_{D_k}|$.  It follows from Proposition \ref{structureJ} that $\left. J_{k}\right|_{D_k}$ is divisible by $\left. \beta \right|_{D_k}$ for any $\beta \in \mathcal{R}\setminus \{ \pm \alpha_k\}$, which implies the statement.
\end{proof}

Let us define the following polynomials:
\begin{equation}\label{defn}
I:= J \prod_{\alpha \in \Delta} \alpha^{-1} \quad \text{and} \quad I_k := J_k \prod_{ \alpha \in \Delta \setminus \alpha_k} \alpha^{-1},
\end{equation}
where $1 \leq k \leq n$. 
We have the following useful result on the relation between polynomials $I_m$ and $I_l$ on the stratum $D_{l, m}$. 

\begin{proposition}\label{relationJ}
Let $\alpha_l, \alpha_m \in \Delta$ be such that $|\mathcal{R}_+ \cap S| > 2$, where $S=\langle \alpha_l, \alpha_m \rangle$. Let $D= D_{l, m}$ be the corresponding stratum. Then
\begin{equation}
\left. I_m \right|_D = (-1)^{l-m-1} \left. I_l \right|_D. 
\end{equation}
\end{proposition}

\begin{proof}
 Let $v_k$ denote the column vector
\begin{equation}
v_k = \big(\partial_{\alpha_k} t^1, \dots, \partial_{\alpha_k}  t^{n-1}   \big)^\intercal,
\end{equation}
where $k=1, \dots, n$. We have
\begin{equation}\label{partial}
\partial_{\alpha_k} t^i = \alpha_{k}(x) Q_{ki}(x) 
\end{equation}
for some $Q_{ki}(x) \in \mathbb{C}[x]$. Denote the  column vector $Q_k=(Q_{k1}, \dots, Q_{k,n-1})^\intercal$. 
Note that  $\alpha_l, \alpha_m$ are simple roots for the irreducible two-dimensional root system $\mathcal{R} \cap S$, hence  $(\alpha_l, \alpha_m) \neq 0$. 
%Let us firstly consider the case when $(\alpha_l, \alpha_m) \neq 0$. 
It follows from equality \eqref{partial} that 
\begin{equation}\label{7.83}
\partial_{\alpha_m} \partial_{\alpha_l} t^i= (\alpha_l, \alpha_m) Q_{li}(x) + \alpha_l (x)\partial_{\alpha_m} Q_{li}(x) = (\alpha_l, \alpha_m) Q_{mi}(x) + \alpha_m (x)\partial_{\alpha_l} Q_{mi}(x). 
\end{equation}
%
%Similarly,
%
%\begin{align}\label{7.84}
%\partial_{\alpha_l}  \partial_{\alpha_m}  p^i = (\alpha_l, \alpha_m) Q_{mi}(x) + \alpha_m (x)\partial_{\alpha_l} Q_{mi}(x)
%\end{align}
%
Restriction of equalities \eqref{7.83} to $D$ gives
\begin{equation}\label{equals}
\left. Q_{li}(x)\right|_D=\left.Q_{mi}(x) \right|_D.
\end{equation}
Let $\widetilde{\Delta}= \Delta \setminus \{\alpha_m, \alpha_l\}$. Note that $I_m = a(x) \alpha_l(x)^{-1} J_m(x)$ and $I_l(x)= a(x) \alpha_m(x)^{-1} J_l(x)$, where $a(x)= \prod_{\alpha \in \widetilde{\Delta}} \alpha(x)^{-1}$. If $l < m$ then $\alpha_l(x)^{-1} J_m(x) = \operatorname{det}A_{lm}$, where the matrix $A_{lm}$ has columns 
$$
v_1, \dots, v_{l-1}, Q_l, v_{l+1}, \dots, \widehat{v}_m, \dots, v_n,
$$
and $\widehat{v}_m$ denotes the omitted column $\widehat{v}_m$. Similarly, $\alpha_{m}(x)^{-1} J_l(x)= \operatorname{det}A_{ml}$, where the matrix $A_{ml}$ has columns 
$$
v_1, \dots, \widehat{v}_l, \dots, v_{m-1}, Q_m, v_{m+1}, \dots, v_n
$$ 
and  $\widehat{v}_l$ denotes the omitted  column $\widehat{v}_l$.  By the property (\ref{equals}) the matrices $\left. A_{lm} \right|_D$, $\left. A_{ml}\right|_D$ have the same columns up to a permutation, which implies the statement. The case $m <l$ is similar.

\end{proof}

\section{ Exceptional groups: dimension $1$ and codimensions $1$, $2$, $3$.}\label{exceptionalsection}

In this section we find the determinant of the restricted Saito metric for  the discriminant strata of  codimensions $1, 2 ,3$ and $n-1$. Analysis of strata of codimensions $1, 2$ and $n-1$ applies to any (irreducible) Coxeter group so these results cover exceptional groups and give another derivation of the determinant for the classical groups for such strata. Analysis of codimension $3$ strata applies to any simply laced root system so it covers all codimension 3 strata in the exceptional groups since codimension 3 strata in $F_4 $ and $H_4$ cases are covered by general  dimension $1$ considerations (this also gives another derivation of the determinant for codimension 3  strata in simply laced classical cases).  

Thus we show that the statement of Theorem \ref{MMtheorem} is true for any root system and a stratum of  codimensions $1, 2 ,3$ and $n-1$. Our analysis depends on the type of the subgraph $\Gamma$ of the Coxeter graph of $\mathcal{R}$ such that the stratum corresponds to simple roots which are vertices of $\Gamma$. In some cases different subgraphs of the same type define non-equivalent strata under the group action \cite{OT} but this does not affect our considerations.

\subsection{\bf{Dimension} $\pmb{1}$} Let us choose $n-1$ different elements $i_1, \dots, i_{n-1} \in \underline{n}$ and consider the stratum $D= D_{i_1, \dots, i_{n-1}}$. Let $i_n\in \underline{n}\setminus \{i_1, \ldots, i_{n-1}\}$.

\begin{theorem}\label{theoremDimension1} The determinant of the restricted Saito metric $\eta_D$ is proportional to $(\widetilde{x}^{i_n})^h$, where $h$ is the Coxeter number of $\mathcal{R}$, and coordinates $\widetilde x$ are defined in \eqref{xtilde}.
\end{theorem}

\begin{proof}The covariant Saito metric $\eta$ can be expressed as
\begin{equation}\label{dim1}
\eta=\sum_{k=1}^n \sum_{i=1}^n \partial_{\widetilde{x}^i} t^k d\widetilde{x}^i \sum_{j=1}^n \partial_{\widetilde{x}^j} t^{n+1 -k} d\widetilde{x}^{j}.
\end{equation}
Since $\partial_{\widetilde{x}^{i_j}}t^k|_D=  \partial_{\alpha_{i_j}} t^k|_D=0$ for $j=1, \dots, n-1$ we get 
\begin{equation}
\eta_D= \sum_{k=1}^n \partial_{\widetilde{x}^{i_n}} t^{k} \partial_{\widetilde{x}^{i_n}} t^{n+1-k} (d\widetilde{x}^{i_n})^2.
\end{equation}
Note that $\partial_{\widetilde{x}^{i_n}} t^k \partial_{\widetilde{x}^{i_n}} t^{n+1-k}$ is proportional to $(\widetilde{x}^{i_n})^h$ since degrees $d_k =\operatorname{deg}t^k$ satisfy $d_k + d_{n+1 -k} = h+2$. This implies the statement. 
\end{proof}

\begin{corollary}
The statement of Theorems \ref{MMtheorem} for $D$ is true. 
\end{corollary}

\subsection{\bf{Codimension} $\pmb{1}$}
Let us fix $m\in \mathbb N$, $1 \leq m \leq n$ and consider the corresponding $(n-1)$-dimensional stratum $D=D_m$. 
%It is sufficient to prove the following statement in order to establish Theorem \ref{MMtheorem} in the codimension 1 case. 

\begin{theorem}\label{codim1th}
The determinant of the restricted Saito metric $\eta_D$ is proportional to 
\begin{equation}\label{cod11}
 I(\mathcal{A}_D) \prod_{\beta \in \mathcal{R}_+ \setminus \{\alpha_m\}} \left. \beta\right|_D.
\end{equation}
\end{theorem}

\begin{proof}
 By Theorem \ref{Saitovo} we have that $\operatorname{det}\eta_D$ is equal to  $ -\left.  \eta^{mm} J^2 \right|_D$, and therefore by Proposition \ref{Saitom} that 
\begin{align*}
\operatorname{det}\eta_D= (-1)^{n+m+1} \left. 2J^2 \partial_{\omega^m} \frac{J_m}{J}  \right|_D=  (-1)^{n+m}\left.  2 J_m \partial_{\omega^m} J\right|_D.
\end{align*}
Since Jacobian $J\sim  \alpha_m \prod_{\beta \in \mathcal{R}_+ \setminus \{\alpha_m\}} \beta$, we get that 
 $$
\left. \partial_{\omega^m} J \right|_D \sim  \prod_{\beta \in \mathcal{R}_+ \setminus \{\alpha_m\}} \left. \beta\right|_D.
$$
 The statement follows by Corollary \ref{factorJ}. 
\end{proof}

This establishes Theorem \ref{thma} for the stratum $D$. 
\begin{theorem}
Theorem \ref{thma1} holds for the stratum $D$. 
\end{theorem}
\begin{proof}
The product in equality (\ref{cod11}) can be written as $\prod_{ H \in \mathcal{A}_D} l_{H}^{m_H}$, where $m_H= |\widetilde{\Sigma}_H|$ with $\widetilde{\Sigma}_H=\{X \in \mathcal{A} | X \supset H, H \neq D\}$. Therefore by Theorem \ref{codim1th} the determinant of $\eta_D$ is  proportional to
\begin{equation*}
\prod_{H \in \mathcal{A}_D} l_H^{k_H},
\end{equation*}
where $k_{H} =m_H+1 =  |\Sigma_H|$, with $\Sigma_H=\{X \in \mathcal{A} |   H \subset X\}$. Let $\beta_H \in {\mathcal R}$ be such that $\beta_H|D$ is a non-zero multiple of $l_H$. Then the root system ${\mathcal R}_{D, \beta} = {\mathcal R}\cap \langle \beta, \alpha_m\rangle$ is the root system of the dihedral group with $k_H$ lines. 
If  it is irreducible then the root system ${\mathcal R}_{D, \beta}^{(0)}$  given by the decomposition \eqref{decomposition} coincides with  ${\mathcal R}_{D, \beta} $ and $k_H = h({\mathcal R}_{D, \beta}^{(0)})$ as required. If ${\mathcal R}_{D, \beta}$ is reducible then  ${\mathcal R}_{D, \beta}^{(0)} = \{\pm \beta\}$ and $k_H =2 = h({\mathcal R}_{D, \beta}^{(0)})$, so the statement follows.
\end{proof}

\subsection{\bf{Codimension} $\pmb{2}$}
Let $\alpha_l, \alpha_m \in \Delta$ be different simple roots. Let us consider the $(n-2)$-dimensional stratum $D=D_{l,m}$. The restriction $\eta_D$  of the Saito metric $\eta$ to the stratum $D$ is well-defined as the components of the metric $\eta$ are polynomials   in the coordinates $x^i$ ($i=1, \dots, n$). However, this is not necessarily true for the individual terms in the expansion \eqref{determinantgeneraltheorem} , \eqref{determinantgeneraltheorem00} of $\operatorname{det}\eta_D$ as these terms can be singular on $D$. Below we will calculate limits of these terms as $x$ tends to $D$ in a prescribed way which will give us the value of $\operatorname{det}\eta_D$. 

More specifically, by Theorem \ref{Saitovo}, the determinant of the metric $\eta_D$ is given by
\begin{align}\label{volume2}
\operatorname{det}\eta_D=-
J^2
\begin{vmatrix}
\eta^{mm} \quad   \eta^{ml} \\
\eta^{ml} \quad  \eta^{ll}
\end{vmatrix},
\end{align}
where the limit of the right-hand side as $x$ tends to $D$ is taken. By Proposition \ref{Saitom} we have
\begin{equation}\label{etaik}
\eta^{ik} = (-1)^{n+k+1} \partial_{\omega^{i}} \frac{J_k}{J} + (-1)^{n+i+1}\partial_{\omega^{k}} \frac{J_i}{J},
\end{equation}
$i, k=1, \dots, n$. And by formulae (\ref{defn}) one has
\begin{align}\label{JkoverJ}
\partial_{\omega^i}\frac{J_k}{J}&= \partial_{\omega^i}\frac{I_{k}}{\alpha_k I}= -\frac{1}{\alpha_{k}^2} \frac{I_k}{I} \delta_{ik} + \frac{1}{\alpha_{k}} \partial_{\omega^i}\frac{I_k}{I}.
\end{align}
Further on, we are interested in the structure of $I$. Let us define 
\beq{defd}
d= | \mathcal{R}_+ \cap \langle \alpha_l, \alpha_m \rangle |-2,
\eeq
which is the degree of vanishing of the polynomial  $I$ on $D$. 
Let us represent $I$ as
\begin{equation}\label{widetildeJ}
 I= f g,
 \end{equation}
where $f \in \mathbb{C}[x]$ is a homogeneous polynomial of degree $d$ in the variables $\alpha_{m}(x), \alpha_{l}(x)$ and $g \in \mathbb{C}[x]$ is not identically zero on $D$. Let $d_0$ be the degree of $f(x)$ as a polynomial in $\alpha_l$, $d_0 \leq d$. We represent $f(x)$ as 
\begin{equation}\label{f}
f(x)=\alpha_m^{d-d_0}(x)\sum_{i=0}^{d_0} a_i \alpha^{i}_l(x) \alpha_m^{d_0-i}(x),
\end{equation}
where $a_i \in \mathbb{C}$, $a_{d_0} \neq 0$.  
%
\begin{comment}
\begin{lemma}\label{functionf}
We have 
$$\left. \alpha_m (x) \partial_{\omega^m} f(x) \right|_{\alpha_{l}=0} =\left. d a_d \alpha_m  (x) ^ d \right|_{\alpha_{l}=0}, $$
\quad \quad and
$$\left. \alpha_l (x) \partial_{\omega^l} f(x) \right|_{\alpha_{m}=0} =\left. d a_0 \alpha_l  (x) ^ d \right|_{\alpha_{m}=0}.$$
\end{lemma}

\begin{proof}
We have by Lemma \ref{deri} that

%
\begin{align*}
\partial_{\omega^{m}} f(x) &=\sum_{i=0}^{d}a_{i} \alpha_{l}(x)^{d-i}\partial_{\omega^{m}} \alpha_{m}(x)^{i}=\sum_{i=0}^{d}a_{i} \alpha_{l}(x)^{d-i}i \alpha_{m}(x)^{i-1}.
\end{align*}
%
Therefore, restricting onto $\alpha_l=0$ we get
%
$$\left. \alpha_m (x) \partial_{\omega^m} f(x) \right|_{\alpha_{l}=0} =\left. d a_d \alpha_m  (x) ^ d \right|_{\alpha_{l}=0}. $$
%
Similarly, one has that
%
\begin{equation}
\left. \alpha_l (x) \partial_{\omega^l} f(x) \right|_{\alpha_{m}=0} =\left. d a_0 \alpha_l  (x) ^ d \right|_{\alpha_{m}=0},
\end{equation}
%
as required. 
\end{proof}
\end{comment}

Let $\alpha, \beta \in \mathcal{R}$. In what follows, we will mean by $\left. F \right|_{\substack{\alpha=0 \\ \beta=0}}$ the restriction of a function $F$ onto $\alpha= \beta=0$ in the order $\alpha=0$ first followed by taking the limit $\beta \rightarrow 0$.
The following lemma takes place.

\begin{lemma}\label{limitf}We have
\begin{equation}
\left. \alpha_l(x) I(x)^{-1} \partial_{\omega^l} I(x) \right|_{\substack{\alpha_m=0 \\ \alpha_l=0}} =d_0.
\end{equation}
\end{lemma}

\begin{proof} 
From  formula (\ref{widetildeJ}) we get 
\begin{align}\label{equalityforffucntion}
 \alpha_l(x) I(x)^{-1} \partial_{\omega^l} I(x)= \alpha_l(x)  g(x)^{-1} \partial_{\omega^l}g(x) + f(x)^{-1} \alpha_{l}(x)  \partial_{\omega^l}f(x).
\end{align}
By formula \eqref{f} one has
\begin{align}\label{equalityforffucntion22}
\left. f(x)^{-1} \alpha_{l}(x)  \partial_{\omega^l}f(x) \right|_{\alpha_m=0}= \left. \frac{\sum_{i=1}^{d_0} i a_i \alpha_l^i \alpha_m^{d_0-i}}{\sum_{i=0}^{d_0} a_i \alpha_l^i \alpha_{m}^{d_0 -i}}\right|_{\alpha_m=0}=d_0. 
\end{align}
The statement follows from formulae \eqref{equalityforffucntion}, \eqref{equalityforffucntion22}.
%Therefore restricting expression \eqref{equalityforffucntion} onto $\alpha_m=0$ first followed by the restriction onto $\alpha_l=0$, the statement follows. 
\end{proof}

Let $\widetilde{\Delta}= \Delta \setminus \{\alpha_l, \alpha_m\}$. Let us consider the diagonal and anti-diagonal terms in the determinant in (\ref{volume2}) separately.  

\begin{lemma}\label{diag} Let  $A=J^2 \eta^{mm} \eta^{ll}$. Then
\begin{equation}\label{Acodime2}
\left. A \right|_{\substack{\alpha_m=0 \\ \alpha_l=0}}=  (-1)^{m+l}4(d_0+1)I_l I_m  \prod_{\alpha \in \widetilde{\Delta}}\left. \alpha^2 \right|_{\substack{\alpha_m=0 \\ \alpha_l=0}}.
\end{equation}
\end{lemma}

\begin{proof} 
By formulae (\ref{etaik}), (\ref{JkoverJ}) we have that
\begin{equation*}
\eta^{kk}=  (-1)^{n+k +1}2 \partial_{\omega^k} \frac{J_k}{J}=  (-1)^{n+k+1}2 \big( - \frac{1}{\alpha^2_k} \frac{I_k}{I} + \frac{1}{\alpha_k} \partial_{\omega^k} \frac{I_k}{I} \big),
\end{equation*}
 for any $k=1, \dots, n$. Therefore
\begin{equation*}
J \eta^{kk}=  (-1)^{n+k+1} 2   \big( -\alpha_{k}^{-1}I_k  + \partial_{\omega^k} I_k - I_k I^{-1} \partial_{\omega^{k} }I\big)  \prod_{\alpha \in \Delta \setminus \{\alpha_k \}}\alpha.
\end{equation*}
Then 
\begin{align*}
%A&=    (-1)^{m+l}4   \big( -\alpha_{m}^{-1}I_m  + \partial_{\omega^m} I_m - I_m I^{-1} \partial_{\omega^{m} }I\big) \big( -\alpha_{l}^{-1}I_l  + \partial_{\omega^l} I_l - I_l I^{-1} %\partial_{\omega^{l} }I\big) \prod_{\alpha \in \Delta} \alpha  \prod_{\alpha \in \widetilde{\Delta}} \alpha  \\
A&=   (-1)^{m+l}4  \Big(-I_m  + \alpha_{m} \big( \partial_{\omega^m} I_m - I_m I^{-1} \partial_{\omega^{m} }I \big)  \Big)  \Big(-I_l  + \alpha_{l} \big( \partial_{\omega^l} I_l - I_l I^{-1} \partial_{\omega^{l} }I \big)  \Big)  \prod_{\alpha \in \widetilde{\Delta}}\alpha^2 .
\end{align*}
We consider the restriction of $A$ on $D_m$ at first. This gives
\begin{equation}
\label{res1tom} 
\left. A \right|_{\alpha_{m}=0} =  (-1)^{m+l}4  I_m \Big(I_l  - \alpha_{l} \big( \partial_{\omega^l} I_l - I_l I^{-1} \partial_{\omega^{l} }I \big)  \Big) \prod_{\alpha \in \widetilde{\Delta}} \left. \alpha^2\right|_{\alpha_m=0}.
\end{equation}
By restricting equality \eqref{res1tom}  further on $\alpha_l=0$ and applying Lemma \ref{limitf} we obtain the statement.
\end{proof}

Let us now consider the anti-diagonal terms. 

\begin{lemma}\label{Blemma}Let $B= \eta^{ml} J$. Then
\begin{equation}
\left. B \right|_{\substack{\alpha_m=0 \\ \alpha_l=0}}  =  (-1)^{n+m}  d_0 I_m\prod_{\alpha \in \widetilde{\Delta}}\left. \alpha \right|_{\substack{\alpha_m=0 \\ \alpha_l=0}}.
\end{equation}
\end{lemma}

\begin{proof}Using formulae  (\ref{etaik}), (\ref{JkoverJ}) we have 
%
%\begin{align*}
%B&=  J\Big( (-1)^{n+l+1} \partial_{\omega^m} \frac{J_l}{J} + (-1)^{n+m+1} \partial_{\omega^l} \frac{J_m}{J} \Big), 
%\end{align*}
%that is
%
\begin{align*}
B=\Big( (-1)^{n+l+1} \alpha_m \big( \partial_{\omega^m} I_l - I_l I^{-1} \partial_{\omega^m} I \big)+ (-1)^{n+m+1} \alpha_l \big( \partial_{\omega^l} I_m - I_m I^{-1} \partial_{\omega^l} I \big) \Big)  \prod_{\alpha \in \widetilde{\Delta}}\alpha. 
\end{align*}
We consider the restriction of $B$ on the hyperplane $D_m$ at first. This gives
\begin{equation}
\label{res01}
\left. B \right|_{\alpha_m=0}=(-1)^{n+m+1} \alpha_l \big( \partial_{\omega^l} I_m - I_m I^{-1} \partial_{\omega^l} I \big) \prod_{\alpha \in \widetilde{\Delta}}\left. \alpha \right|_{\alpha_m=0}.
\end{equation}
Then restricting equality \eqref{res01}  further on $\alpha_l=0$ and applying  Lemma \ref{limitf} we obtain the required result. 
\end{proof}

Now we are ready to obtain a general expression for the determinant $\operatorname{det} \eta_D$. 

\begin{theorem}\label{saitocodimensio2}
The determinant of the metric $\eta_D$ is equal  to 
\begin{align}
(-1)^{l+m+1}(d_0+2)^2  I_m I_l \prod_{\alpha \in \widetilde{\Delta}} \left. \alpha^2\right|_D.
\end{align}
\end{theorem}

\begin{proof}
In the notation of Lemmas \ref{diag}, \ref{Blemma} by Theorem \ref{Saitovo} we have
\begin{align*}
\operatorname{det}\eta_D &=  \left. B^2\right|_{\substack{\alpha_m=0 \\ \alpha_l=0}} - \left. A\right|_{\substack{\alpha_m=0 \\ \alpha_l=0}}. 
\end{align*}
By these lemmas we get
\begin{equation}\label{expresscod2}
\operatorname{det}\eta_D=   \left(d_0^2 I^2_m +   (-1)^{m+l+1}4(d_0+1)I_l I_m  \right) \prod_{\alpha \in \widetilde{\Delta}} \left. \alpha^2\right|_D .
\end{equation}
Let us first consider the case where degree $d$ defined by \eqref{defd} satisfies $d>0$. Then 
%$|\mathcal{R}_+ \cap \langle \alpha_l, \alpha_m \rangle | >2$ and from 
by Proposition \ref{relationJ} we get  $I_m=(-1)^{l-m-1} I_l$ on $D$. Therefore
\begin{equation}\label{expresscod2new}
\operatorname{det}\eta_D=(-1)^{l+m+1}(d_0+2)^2 I_m I_l \prod_{\alpha \in \widetilde{\Delta}} \left. \alpha^2\right|_D,
\end{equation}
as required. Let us now suppose that $d=0$. Then $f$ is constant and $d_0=0$, hence equality \eqref{expresscod2} implies the statement. 
\end{proof}

Theorem \ref{saitocodimensio2} implies Theorem \ref{MMtheorem} for $D$. Let $\mathcal{R}_D = \mathcal{R} \cap \langle \alpha_m, \alpha_l \rangle$. 

\begin{theorem}\label{MT1cod2}
The statement of Theorem \ref{MMtheorem} is true for the stratum $D$.
%, that is the determinant of the metric $\eta_D$ is proportional to 
%
\end{theorem}

\begin{proof}
By Corollary \ref{factorJ} we have that $J_k \sim {\mathcal I}({\mathcal A}_{D_k})$ on $D_k$ for  $k=l, m$.
Therefore $\alpha_l^{-1} J_m \sim {\mathcal I}({\mathcal A}_{D_m}\setminus \{D\})$ and $\alpha_m^{-1} J_l \sim {\mathcal I}({\mathcal A}_{D_l}\setminus \{D\})$ on $D$.
Hence by Theorem \ref{saitocodimensio2} determinant of $\eta_D$ is proportional to
$$
I_m I_l \prod_{\alpha \in \widetilde{\Delta}}  \alpha^2 = J_m J_l \alpha_m^{-1} \alpha_l^{-1} \sim {\mathcal I}({\mathcal A}_{D_m}\setminus \{D\}) {\mathcal I}({\mathcal A}_{D_l}\setminus \{D\})
$$
on $D$. Since  $ \{D\} = \mathcal{A}^D_{D_m}= \mathcal{A}^D_{D_l} $ and the rank of the root system ${\mathcal R}_D$ given by \eqref{RDdecompo} is 2, the statement follows. 
%For any $\widehat{H} \in \mathcal{A}_{D_m}$ let $\alpha_{\widehat{H}} \in \mathcal{R}$ be the corresponding covector such that  $\widehat{H}= \{x \in D_m | \alpha_{\widehat{H}}(x)=0\}$. Similarly for any $\widehat{H} \in \mathcal{A}_{D_l}$ we choose $\alpha_{\widehat{H}} \in \mathcal{R}$ such that $\widehat{H} = \{ x \in D_l | \alpha_{\widehat{H}}(x)=0\}$. We note that from Theorem \ref{saitocodimensio2}  and Corollary \ref{factorJ} $\operatorname{det}\eta_D$ is proportional to
%
%\begin{align}\label{cod2lhformula}
%\left. \alpha_l^{-1} J_m \right|_D  \left. \alpha_m^{-1} J_l \right|_D \sim
 %\prod_{\substack{ \widehat{H} \in \mathcal{A}_{D_m}\\ \widehat{H} \neq D}}\left. \alpha_{\widehat{H}} \right|_D \prod_{\substack{ \widehat{H} \in \mathcal{A}_{D_l}\\ \widehat{H} \neq D}}\left. \alpha_{\widehat{H}} \right|_D  \sim \prod_{H \in \mathcal{A}_D} l_H^{k_H},
%\end{align}
%
%where $l_H =\left.  \alpha_{\widehat{H}}\right|_D$ for $\widehat{H} \in \mathcal{A}_{D_m} \cup \mathcal{A}_{D_l}$ such that $H= \widehat{H} \cap D$, $k_H \in \mathbb{N}$. %Formula \eqref{cod2defininingeq} follows from \eqref{cod2lhformula}.
%\begin{align}\label{cod2defininingeq}
%I(\mathcal{A}_{D_m} \setminus \mathcal{A}^D_{D_m})I(\mathcal{A}_{D_l} \setminus \mathcal{A}^D_{D_l}),
%\end{align}
%on $D$.
\end{proof}

The above analysis shows that the statement of Theorem \ref{MMtheorem} for the strata of codimensions $1$, $2$ and $n-1$ is true. This covers all strata in finite Coxeter groups $I_2(p), H_3, H_4, F_4$. This leaves us to study simply laced cases $E_6, E_7, E_8$ only.

\subsection{\bf{Codimension} $\pmb{3}$}\label{codim3}
We consider now $(n-3)$-dimensional strata $D$ for simply laced  root systems $\mathcal{R} =E_n$, $n=6, 7, 8$, although the analysis works for any irreducible simply laced root system. The analysis  will depend on the types of the Coxeter graphs of the strata in addition to their (co)dimensions.
Thus we consider three possible cases for a stratum $D$, namely, it can be  of type $A_3$, $A_2 \times A_1$ or $A_1^3$.
%Let $\mathcal{R}_+$ be the positive root system of the root systems $E_n$, $n=6, 7, 8$, although the presented analysis below works for any irreducible simply laced root system. 
Let $\lambda, \nu, \theta \in \Delta$  so that $D= D_{\lambda, \nu, \theta}$.

 \textbf{Strata of type} $\bf{A_3}.$  Let us assume that the root system $\mathcal{R}_D=\mathcal{R}\cap \langle \lambda, \nu, \theta \rangle$ is a subsystem of $\mathcal{R}$ of type $A_3$ and consider the corresponding Coxeter subgraph
\begin{align*}
A_3 : \quad
\begin{dynkinDiagram}{A}{3}
\dynkinLabelRoots*{\lambda, \nu, \theta}
\end{dynkinDiagram}
\end{align*}
Note that $\lambda+\nu$, $\nu+\theta$, $\lambda +\nu +\theta$ $\in \mathcal{R}_+$. The Jacobian  $J$ can be represented as
\begin{equation}\label{JA3}
J= \lambda \nu \theta (\lambda+\nu)(\nu+\theta) (\lambda+\nu+\theta) \Pi,
\end{equation}
where $\Pi$ is proportional to ${\mathcal I}(\mathcal{A} \setminus \mathcal{A}^D)$. Note that  $\Pi$ is non-zero on $D$. By Proposition \ref{divisible} we have in the notation  \eqref{Jalpha}
\vspace{-0.3cm}
\begin{align}
J_\lambda&= \nu \theta (\nu+\theta) K_\lambda, \\
J_\nu&= \lambda \theta K_\nu, \label{JnuA3}\\
J_\theta&=\lambda \nu (\lambda+\nu) K_\theta, \label{Jtheta}
\end{align}
for some polynomials $K_\lambda, K_\theta, K_\nu \in \mathbb{C}[x]$.

 We assume without loss of generality that the ordering of simple roots $\sigma\colon \underline{n} \rightarrow \Delta$ is such that $n+ \sigma^{-1}(\lambda)$ is odd, and that $\sigma^{-1}(\nu)= \sigma^{-1}(\lambda)+1$, $\sigma^{-1}(\theta)= \sigma^{-1}(\lambda)+2$. The following statement follows from formula \eqref{notationJ_k} and formulae (\ref{JA3})--(\ref{Jtheta}). 

\begin{proposition}\label{identityA3}
The $\lambda$, $\nu$, and $\theta$ components of the identity field \eqref{identityfieldrestated} are given by
\begin{align}
e^\lambda &= - \frac{K_\lambda}{\lambda (\lambda+\nu) (\lambda+\nu+ \theta) \Pi}, \\
e^\nu&= \frac{K_\nu}{ \nu (\lambda+\nu) (\nu+\theta) (\lambda+\nu + \theta) \Pi}, \\
e^\theta &= -\frac{K_\theta}{\theta(\nu+\theta) (\lambda+\nu+\theta)\Pi}.
\end{align}
\end{proposition}

In what follows, we deal with the restricted metric $\eta_D$ by restricting $\eta$ on $D_\nu$ firstly, then on $D_{\nu, \theta}$ and finally on $D$. Firstly, we derive certain relations between polynomials $K_\lambda, K_\theta, K_\nu$.

\begin{lemma}\label{structureKn}We have
\begin{equation}\label{Kn}
\left. K_\nu\right|_{D_\nu} =\left.  \lambda K_\theta + \theta B \right|_{D_\nu}, 
\end{equation}
for some polynomial $B \in \mathbb{C}[x]$ such that
\begin{equation} \label{B}
\left. B \right|_{D}= \left. K_\theta \right|_D. 
\end{equation}
\end{lemma}

\begin{proof}
By Proposition \ref{relationJ}, $\frac{J_\nu}{\theta}=\frac{J_\theta}{\nu}$ on $D_{\nu, \theta}$, therefore $K_\nu= \lambda K_\theta$ on $D_{\nu, \theta}$ by formulae (\ref{JnuA3}), (\ref{Jtheta}). Consider $K_\nu - \lambda K_\theta$ on the hyperplane $D_\nu$. This polynomial vanishes if $\theta=0$. Therefore we can represent $K_\nu$ on $D_\nu$ as 
\begin{equation}
\left. K_\nu \right|_{D_\nu }= \left. \lambda K_\theta + \theta B\right|_{D_\nu},
\end{equation}
for some $B \in \mathbb{C}[x]$ as required. 

Furthermore, polynomial $K_\nu$ is divisible by $\lambda+\theta$ on $D_\nu$ since by Corollary \ref{factorJ} the restriction   $\left. J_\nu \right|_{D_\nu}$ is divisible by $\left. \lambda+ \nu +\theta \right|_{D_\nu}$. Hence,
\begin{equation}\label{KnP}
\left. K_\nu \right|_{D_\nu} = \left. \lambda K_\theta + \theta B\right|_{D_\nu}=\left.  (\lambda+\theta) P\right|_{D_\nu},
\end{equation}
for some $P \in \mathbb{C}[x]$. Moreover by restricting equality (\ref{KnP}) further on $D_{\nu, \lambda}$, we get that 
%
%\begin{equation}%\label{B=P}
$\left. B \right|_{D_{\nu,\lambda}}=\left. P\right|_{D_{\nu,\lambda}}$.
%\end{equation}
%
 Similarly, restricting equality \eqref{KnP}  on $D_{\nu, \theta}$, we get that $\left. P\right|_{D_{\nu,\theta}}=\left. K_\theta \right|_{D_{\nu,\theta}}$.  
%from equality (\ref{B=P}) 
 %
 %\vspace{-0.3cm}
 %\begin{equation*}
% $  B_D= K_\theta|_D$,
% \end{equation*} 
% as required. 
Relation \eqref{B} follows.
\end{proof}

\begin{lemma} \label{Klambda1}We have
\begin{equation}
 \left. K_\lambda \right|_D=  \left. K_\theta\right|_D. 
\end{equation}
\end{lemma}

\begin{proof}
By Proposition \ref{relationJ}, we have $\frac{J_\nu}{\lambda}= \frac{J_\lambda}{\nu}$ on $D_{\nu, \lambda}$ and hence $K_\nu= \theta K_\lambda$ on $D_{\nu, \lambda}$. It follows from equality (\ref{Kn}) that $\left. K_\nu \right|_{D_{\nu, \lambda}}= \left. \theta B \right|_{D_{\nu, \lambda}}$,  hence $\left. K_\lambda \right|_{D_{\nu, \lambda}}= \left. B \right|_{D_{\nu, \lambda}}$.  The statement now follows from formula \eqref{B}.
\end{proof}

Theorem \ref{Saitovo} gives a general formula for the determinant of the Saito metric $\eta_D$ which we now specialize to the case of codimension $3$ stratum. Let us represent $J$ given by formula (\ref{JA3})  as $J= \lambda \nu \theta \bar{J}$, where
\begin{equation} \label{barJ}
\bar{J}=(\lambda+\nu)(\nu+\theta) (\lambda+\nu+\theta) \Pi.
\end{equation}
We will write components of Saito metric $\eta^{\sigma^{-1}(\alpha)\sigma^{-1}(\beta)}$ as $\eta^{\alpha \beta}$, $\alpha, \beta \in \Delta$. We rearrange $\operatorname{det} \eta_D$ as
\begin{align}\label{detA3}
\operatorname{det}\eta_D&=-
\begin{vmatrix}
\eta^{\lambda \lambda} & \eta^{\lambda \nu} & \eta^{\lambda \theta} \\
\eta^{\lambda \nu} & \eta^{\nu \nu} & \eta^{\nu \theta}\\
\eta^{\lambda \theta} & \eta^{ \nu \theta} & \eta^{\theta \theta}
\end{vmatrix} 
 \left. J^2 \right|_D =-\begin{vmatrix}
\lambda^2 \eta^{\lambda \lambda}  & \lambda \nu \eta^{\lambda \nu} &\lambda \theta  \eta^{\lambda \theta} \\
\lambda \nu \eta^{\lambda \nu}  & \nu^2 \eta^{\nu \nu}  & \nu \theta \eta^{\nu \theta}\\
 \lambda \theta  \eta^{\lambda \theta}& \nu \theta  \eta^{ \nu \theta} &\theta^2  \eta^{\theta \theta} 
\end{vmatrix}  \left. \bar{J}^2 \right|_D. 
\end{align}
Let $A=(a_{ij})_{i,j=1}^3$ be the matrix 
\begin{align}\label{Amatrix}
A=\begin{pmatrix}
\lambda^2 \eta^{\lambda \lambda}  & \lambda \nu \eta^{\lambda \nu} &\lambda \theta  \eta^{\lambda \theta} \\
\lambda \nu \eta^{\lambda \nu}  & \nu^2 \eta^{\nu \nu}  & \nu \theta \eta^{\nu \theta}\\
 \lambda \theta  \eta^{\lambda \theta}& \nu \theta  \eta^{ \nu \theta} &\theta^2  \eta^{\theta \theta} 
\end{pmatrix} .
\end{align}
%
%Let us recall the basis of fundamental coweights $\omega^i$ ($i=1, \dots, n$), 
We will also write $\omega^\alpha$ for the corresponding fundamental coweight  $\omega^{\sigma^{-1}(\alpha)}$ given by formula \eqref{fundamcow}. Then we have  
\begin{equation*}
(\omega^\alpha, \beta)= \begin{dcases} 1, & \alpha=\beta \\  0,& \alpha \neq \beta. \end{dcases}
\end{equation*}

\begin{proposition}\label{A}
The matrix entries $a_{ij}$ ($1 \leq i, j \leq 3$) are well-defined generically on $D_\nu$, and they have the following form on $D_\nu$:
\begin{alignat}{2}
 a_{11}&=\lambda^2 \eta^{\lambda \lambda}  = 2 \lambda^2 \partial_{\omega^\lambda} \Big( \frac{K_\lambda}{\lambda^2 (\lambda+\theta) \Pi}\Big), \label{fo1}\\
a_{22}&= \nu^2 \eta^{\nu \nu}  =  \frac{2 K_\nu}{\lambda \theta (\lambda +\theta) \Pi}, \label{fo2} \\
a_{33}&= \theta^2 \eta^{\theta \theta}  = 2 \theta^2 \partial_{\omega^\theta} \Big( \frac{K_\theta}{\theta^2 (\lambda+\theta) \Pi}\Big),\label{fo3}\\
a_{12}&= \lambda \nu \eta^{\lambda \nu}  = - \frac{\lambda}{\theta} \partial_{\omega^\lambda} \Big( \frac{K_\nu}{\lambda (\lambda +\theta)\Pi}\Big),\label{fo4}\\
a_{13}&=\lambda \theta \eta^{\lambda \theta}   = \frac{\lambda}{\theta} \partial_{\omega^\lambda} \Big( \frac{K_\theta}{(\lambda+\theta)\Pi}\Big) + \frac{\theta}{\lambda} \partial_{\omega^\theta} \Big( \frac{K_\lambda}{(\lambda+\theta)\Pi}\Big),\label{fo5}\\
a_{23}&= \nu \theta \eta^{\nu \theta}=- \frac{\theta}{\lambda} \partial_{\omega^\theta} \Big( \frac{K_\nu}{ \theta (\lambda+\theta)\Pi}\Big)\label{fo6}. 
\end{alignat}
\end{proposition}

%
%The matrix entries $a_{ij}$ are well-defined generically on $D_\nu$ and have the form
%\begin{alignat}{2}
%a_{11}&=\eta^{\lambda \lambda} \lambda^2 =2 (\partial_{\omega^\lambda} \frac{K_\lambda}{\Pi}) (\lambda+\theta)^{-1} - 2 \frac{K_\lambda}{\Pi} \Big(  2 \lambda^{-1} (\lambda+\theta)^{-1} + (\lambda+\theta)^{-2} \Big), \\
%a_{22}&= \eta^{\nu \nu} \nu^2 = 2 \frac{K_\nu}{\Pi} \Big(  \lambda^{-1} \theta^{-1} (\lambda+\theta)^{-1}\Big), \\
%a_{33}&= \eta^{\theta \theta} \theta^2 =2 (\partial_{\omega^\theta} \frac{K_\theta}{\Pi}) (\lambda+\theta)^{-1} -2 \frac{K_\theta}{\Pi}\Big ( 2 \theta^{-1} (\lambda+\theta)^{-1} + (\lambda+\theta)^{-2}\Big),\\
%a_{12}&= \eta^{\lambda \nu} \lambda \nu =- (\partial_{\omega^\lambda} \frac{K_\nu}{\Pi}) \theta^{-1} (\lambda+\theta)^{-1} + \frac{K_\nu}{\Pi} \Big( \lambda^{-1} \theta^{-1} (\lambda+\theta)^{-1} + \theta^{-1} (\lambda+\theta)^{-2} \Big),\\
%a_{13}&=\eta^{\lambda \theta} \lambda \theta  =(\partial_{\omega^\lambda} \frac{K_\theta}{\Pi}) \lambda \theta^{-1} (\lambda+\theta)^{-1} - \frac{K_\theta}{\Pi} \Big (  \lambda \theta^{-1} (\lambda+\theta)^{-2} \Big) \\
%&  \hspace{1.63cm} +(\partial_{\omega^\theta} \frac{K_\lambda}{\Pi}) \theta \lambda^{-1} (\lambda+\theta)^{-1}-\frac{K_\lambda}{\Pi} \Big ( \theta \lambda^{-1} (\lambda+\theta)^{-2} \Big),\\
%a_{23}&= \eta^{\nu \theta}\nu \theta= -(\partial_{\omega^\theta} \frac{K_\nu}{\Pi}) \lambda^{-1} (\lambda+\theta)^{-1} + \frac{K_\nu}{\Pi}  \Big( \lambda^{-1} \theta^{-1} (\lambda+\theta)^{-1} + \lambda^{-1} (\lambda+\theta)^{-2} \Big).
%\end{alignat}

\begin{proof}
By Proposition \ref{Saitom}  we have $\eta^{\alpha\beta} = - \partial_{\omega^{\alpha}}e^{\beta} - \partial_{\omega^{\beta}}e^{\alpha}$ for $\alpha, \beta \in \{ \lambda, \mu, \nu\}$,
where $e^\alpha$ is given by formula  \eqref{notationJ_k}.  Formulae (\ref{fo1}), (\ref{fo3}), (\ref{fo5}) follow from Proposition \ref{identityA3} immediately.  Let us prove formula (\ref{fo2}). We have
 \vspace{-0.2cm}
  \begin{equation*} \nu^2 \eta^{\nu \nu}= -2 \nu^2 \partial_{\omega^\nu} \Big( \frac{K_\nu}{\nu (\lambda+\nu)(\nu+\theta)(\lambda+\nu+\theta)\Pi}\Big).
  \end{equation*}
  By Leibniz rule and taking the limit $\nu(x) \rightarrow 0$ we obtain the formula. Formulae (\ref{fo4}), (\ref{fo6}) follow similarly. 
\end{proof}

By Proposition \ref{A} we see that the entries of $A$ may be singular on $D_{\nu, \theta}$. Therefore, in order to restrict $\bar{J}^2 \operatorname{det}A$ on $D_{\nu, \theta}$ we consider the expansion of $\operatorname{det}A$ and collect the terms with the same order of poles at $\theta=0$. 
Let $\operatorname{det}A$ be 
\begin{equation}\label{detA}
\operatorname{det}A= C + E,
\end{equation} 
where
\begin{equation*}
C=-a_{12}^2 a_{33} + 2 a_{12} a_{23} a_{13} - a_{13}^2 a_{22},
\end{equation*}
and 
\begin{equation*}
E= a_{11} \big( a_{22} a_{33} - a^2_{23}\big).
\end{equation*}
Note that $E$ has a pole at $\theta=0$ of order at most $2$. Let us study the term $C$ near $\theta=0$. 

\begin{lemma}\label{C1,C2}We have 
\begin{equation}\label{C}
C=\frac{1}{\theta^3} C_1+ \frac{1}{\theta^2} C_2,
\end{equation}
where $C_1$, $C_2$ are well-defined generically on $D_{\nu,\theta}$ and have the following form on $D_\nu$:
\begin{align*}
C_1 &=(\lambda+\theta)^{-3} \Bigg( \frac{4 K_\theta}{\Pi}  \Big(  \frac{K_\nu}{\Pi} \Big( \lambda^{-1}  +  (\lambda+\theta)^{-1} \Big) - \partial_{\omega^\lambda} \frac{K_\nu}{\Pi} \Big)^2 \\ & -    \frac{2 K_\nu}{\Pi}  
 \Big( \frac{K_\nu}{\Pi} \Big( \lambda^{-1}   + (\lambda+\theta)^{-1} \Big)- \partial_{\omega^\lambda} \frac{K_\nu}{\Pi}\Big) \times \\
 &\times \Big(  \frac{K_\theta}{\Pi}   ( \lambda+\theta)^{-1}- \partial_{\omega^\lambda} \frac{K_\theta}{\Pi} \Big) - \frac{2 \lambda K_\nu}{\Pi} \Big(  \frac{K_\theta}{\Pi}  ( \lambda+\theta)^{-1}  - \partial_{\omega^\lambda} \frac{K_\theta}{\Pi}  \Big)^2 \Bigg), 
\end{align*}
\begin{align}\label{C2}
C_2 & =(\lambda+\theta)^{-3} \Bigg ( 2 \Big( \partial_{\omega^\lambda} \frac{K_
\nu}{\Pi}   - \frac{K_\nu}{\Pi} \Big( \lambda^{-1}  + (\lambda + \theta)^{-1} \Big) \Big)^2  \Big( \frac{K_\theta}{\Pi} (\lambda+\theta)^{-1} - \partial_{\omega^\theta} \frac{K_
\theta}{\Pi} \Big) \nonumber\\ & + 2 \Big(  \partial_{\omega^\theta} \frac{K_
\nu}{\Pi}   - \frac{K_\nu}{\Pi} (\lambda+\theta)^{-1} \Big)
 \Big( \partial_{\omega^\lambda} \frac{K_
\nu}{\Pi} - \frac{K_\nu}{\Pi} \Big( \lambda^{-1}+ (\lambda + \theta)^{-1} \Big) \Big) \times \nonumber \\
&\times  \Big(  \partial_{\omega^\lambda} \frac{K_
\theta}{\Pi}   - \frac{K_\theta}{\Pi}  (\lambda+\theta)^{-1} \Big) + 2 \theta^2\lambda^{-2} \Big( \partial_{\omega^\theta} \frac{K_\lambda}{\Pi}-\frac{K_\lambda}{\Pi}  (\lambda+\theta)^{-1} \Big) \times \nonumber \\
& \times \Big(  \frac{K_\nu}{\Pi} \Big( \lambda^{-1}   + (\lambda+\theta)^{-1} \Big)- \partial_{\omega^\lambda} \frac{K_\nu}{\Pi} \Big)  \Big(  \frac{K_\nu}{\Pi}  \Big(  \theta^{-1} +  (\lambda+\theta)^{-1} \Big)-\partial_{\omega^\theta} \frac{K_\nu}{\Pi} \Big) - \nonumber \\
& -  2 \theta \lambda^{-1}  \frac{ K_\nu}{\Pi}    \Bigg( (\theta \lambda^{-1})^2 \Big( \partial_{\omega^\theta} \frac{K_\lambda}{\Pi} - \frac{K_\lambda}{\Pi}  (\lambda+\theta)^{-1} \Big)^2 + \nonumber \\
&+  2  \Big( \partial_{\omega^\theta} \frac{K_\lambda}{\Pi} -\frac{K_\lambda}{\Pi}  (\lambda+\theta)^{-1} \Big) \Big( \partial_{\omega^\lambda} \frac{K_\theta}{\Pi} - \frac{K_\theta}{\Pi}  (\lambda+\theta)^{-1} \Big) \Bigg)  \Bigg).
\end{align}
\end{lemma}

\begin{proof} We expand formulae (\ref{fo3}), (\ref{fo6}) as 
\begin{alignat*}{2}
a_{33}& = 2 \partial_{\omega^\theta} \Big( \frac{K_\theta}{ (\lambda+\theta) \Pi}\Big) -  \frac{4 K_\theta}{\theta (\lambda+\theta) \Pi}, \\
a_{23}& = - \frac{1}{\lambda} \partial_{\omega^\theta} \Big( \frac{K_\nu}{(\lambda+\theta)\Pi}\Big)+\frac{K_\nu}{\lambda \theta (\lambda+\theta) \Pi} .
\end{alignat*}
Then expressions for $C_1$, $C_2$ follow by Proposition \ref{A} using Leibniz rule and by collecting terms with the same degree of the pole at $\theta=0$. 
\end{proof}

%By Lemma \ref{structureKn} we have 
%
%\begin{equation}\label{KnuDnuA3}
%\left. K_\nu \right|_{D_{\nu}} = \left. \lambda K_\theta + \theta B \right|_{D_\nu}, 
%\end{equation}
%
%and hence we can represent $K_\nu$ as 
%
%\begin{equation}\label{KnuglobalA3}
%K_\nu = \lambda K_\theta + \theta B + \nu Q,
%\end{equation}
%
%or some polynomial $Q \in \mathbb{C}[x]$. Therefore we get
%

%\begin{equation}\label{partialomegalambdaKn}
%\left.  \partial_{\omega^{\lambda}}\frac{K_\nu}{\Pi} \right|_{D_\nu}= \left.\partial_{\omega^{\lambda}} \frac{\lambda K_\theta + \theta B}{\Pi}\right|_{D_\nu}.
%\end{equation}
 
%since for $\lambda, \mu \in \Delta$ we have
%
%\begin{equation*}
%(\omega^\lambda, \mu)= \begin{dcases} 1, & \lambda=\mu \\  0,& \lambda \neq \mu. \end{dcases}
%\end{equation*}
%
\begin{lemma}\label{C1,C2,next} The function $\left. C_1\right|_{D_\nu}$ is divisible by $\theta$, that is we can represent it as $$C_1|_{D_\nu} =( \widetilde{C}_1 + F \theta) \theta|_{D_\nu},$$ where $\widetilde{C}_{1}$, $F$ are well-defined generically on $D_{\nu, \theta}$ and have the following form on $D_\nu$: 
\begin{align}\label{widetildeC1}
\widetilde{C}_1 = (\lambda+\theta)^{-3} \Bigg(&-4 \lambda \frac{B}{\Pi} \Big( \partial_{\omega^\lambda} \frac{ K_\theta}{\Pi}   - \frac{K_\theta}{\Pi} (\lambda + \theta)^{-1} \Big) ^2  + 6 \lambda  \frac{K_\theta}{\Pi}    \Big( \partial_{\omega^\lambda} \frac{ K_\theta}{\Pi} - \frac{K_\theta}{\Pi}    (\lambda + \theta)^{-1} \Big) \times \nonumber \\
&\times \Big( \widehat{B}- \frac{B}{\Pi} \Big( \lambda^{-1} + (\lambda + \theta)^{-1} \Big)\Big)\Bigg),
\end{align}
and 
\begin{align*}
F &=(\lambda+\theta)^{-3} \Bigg( 4  \Big(\frac{K_\theta}{\Pi}  \Big) \Big( -\widehat{B} + \frac{B}{\Pi} \Big( \lambda^{-1} + (\lambda+\theta)^{-1} \Big)\Big)^2 - \\ 
&-  \frac{2 B}{\Pi}\Big( -\widehat{B}+ \frac{B}{\Pi} \Big( \lambda^{-1} + (\lambda+\theta)^{-1} \Big)\Big) \Big( - \partial_{\omega^\lambda} \frac{K_\theta}{\Pi} + \frac{K_\theta}{\Pi} (\lambda+\theta)^{-1} \Big)\Bigg),
\end{align*}
where $\widehat{B}= \partial_{\omega^\lambda} \frac{B}{\Pi}$.
\end{lemma}

\begin{proof}
Note that 
\begin{equation}\label{somerelationA3}
 \frac{\lambda K_\theta}{\Pi}\big( \lambda^{-1} + (\lambda+\theta)^{-1} \big) - \partial_{\omega^\lambda} \frac{\lambda K_\theta}{\Pi} = \lambda\big( \frac{K_\theta}{\Pi}(\lambda+\theta)^{-1} - \partial_{\omega^\lambda} \frac{K_\theta}{\Pi} \big).
\end{equation}
The statement follows by substituting formulae  % (\ref{KnuDnuA3}), 
\eqref{Kn}  and   its consequence
\begin{equation*}\label{partialomegalambdaKn}
\left.  \partial_{\omega^{\lambda}}\frac{K_\nu}{\Pi} \right|_{D_\nu}= \left.\partial_{\omega^{\lambda}} \frac{\lambda K_\theta + \theta B}{\Pi}\right|_{D_\nu}
\end{equation*}
into $C_1$, collecting equal powers of $\theta$ and making use of (\ref{somerelationA3}). 
\end{proof}

\begin{lemma}\label{lemmaCitilde} We have
\begin{align*}
\left. \widetilde{C}_1  \right|_{D_{\nu, \theta}} &= 2 \lambda^{-4} \Bigg(  4B \frac{K_\theta^2}{\Pi^3} -2 B \lambda \frac{ K_\theta}{\Pi^2}  \partial_{\omega^\lambda} \frac{K_\theta}{\Pi} - 3 \widehat{B} \lambda  \frac{ K_\theta^2}{\Pi^2} -2 \frac{B \lambda^2}{\Pi}  (\partial_{\omega^\lambda} \frac{K_\theta}{\Pi})^2 +3 \widehat{B} \lambda^2 \frac{ K_\theta}{\Pi} \partial_{\omega^\lambda} \frac{K_\theta}{\Pi} \Bigg)\bigg\rvert_{D_{\nu, \theta}}.
\end{align*}
\end{lemma}

\begin{proof}
The statement follows immediately from the restriction of formula (\ref{widetildeC1}) to the stratum $D_{\nu, \theta}$.
\end{proof}

Let us now consider the term $C_2$ in equality (\ref{C}). The restriction of $C_2$ to $D_{\nu, \theta}$ is given in the following lemma. 

\begin{lemma} \label{lemmaC2}We have
\begin{align*}
\left. C_2  \right|_{D_{\nu, \theta}} & = 2\lambda^{-2}  \frac{B}{\Pi} \Big(  \partial_{\omega^\lambda} \frac{K_
\theta}{\Pi}-\frac{K_\theta}{\Pi}\lambda^{-1} \Big)^2 \bigg\rvert_{D_{\nu, \theta}}.
\end{align*}
\end{lemma}

\begin{proof} Restricting formula (\ref{C2}) to $D_{\nu, \theta}$ we get
\begin{align*}
\left. C_2  \right|_{D_{\nu, \theta}} & =2\lambda^{-3} \Bigg (  \Big( \partial_{\omega^\lambda} \frac{K_
\nu}{\Pi}   - \frac{2 K_\nu}{ \lambda \Pi} \Big)^2  \Big( - \partial_{\omega^\theta} \frac{K_
\theta}{\Pi}  +  \frac{K_\theta}{\lambda \Pi} \Big) +  \Big(  \partial_{\omega^\theta} \frac{K_
\nu}{\Pi}   - \frac{K_\nu}{\lambda \Pi} \Big)
 \Big( \partial_{\omega^\lambda} \frac{K_
\nu}{ \Pi} - \frac{2 K_\nu}{\lambda \Pi}  \Big) \times \\
&\times  \Big(  \partial_{\omega^\lambda} \frac{K_
\theta}{\Pi}   - \frac{K_\theta}{\lambda \Pi}  \Big) \Bigg)\bigg\rvert_{D_{\nu, \theta}}.
\end{align*}
It follows from \eqref{Kn} that
%(\ref{KnuDnuA3}) that 
%
\begin{equation}\label{KnudnuthetaA3}
\left. K_\nu \right|_{D_{\nu, \theta}} = \left. \lambda K_\theta \right|_{D_{\nu, \theta}},
\end{equation}
and hence
\begin{equation}\label{partialomegalambdaKnudnutheta}
\left. \partial_{\omega^\lambda} \frac{K_\nu}{\Pi} \right|_{D_{\nu, \theta}} = \left. \partial_{\omega^\lambda} \frac{\lambda K_\theta}{\Pi} \right|_{D_{\nu, \theta}} . 
\end{equation}
We also have from (\ref{Kn}) that 
\begin{equation}\label{partialomegathetaKndnutheta}
\left. \partial_{\omega^\theta} \frac{K_\nu}{\Pi} \right|_{D_{\nu, \theta}}  = \left. \lambda \partial_{\omega^\theta} \frac{K_\theta}{\Pi} + \frac{B}{\Pi} \right|_{D_{\nu, \theta}} .
\end{equation}
By using (\ref{KnudnuthetaA3}), (\ref{partialomegalambdaKnudnutheta}) we get 
\begin{equation*}
\left. \partial_{\omega^\lambda} \frac{K_\nu}{\Pi}  - \frac{2 K_\nu}{\lambda \Pi} \right|_{D_{\nu, \theta}} = \left. \lambda \partial_{\omega^\lambda} \frac{K_\theta}{\Pi} - \frac{K_\theta}{\Pi} \right|_{D_{\nu, \theta}} .
\end{equation*}
Hence, 
\begin{equation}\label{C2Dnutheta}
\left. C_2 \right|_{D_{\nu, \theta}} = \left. 2 \lambda^{-3} \big( \lambda \partial_{\omega^\lambda} \frac{K_\theta}{\Pi} - \frac{K_\theta}{\Pi}\big)^2  \Big( - \partial_{\omega^\theta} \frac{K_\theta}{\Pi} + \frac{K_\theta}{\lambda \Pi} + \lambda^{-1} ( \partial_{\omega^\theta} \frac{K_\nu}{\Pi} - \frac{K_\nu}{\lambda \Pi} ) \Big) \right|_{D_{\nu, \theta}} . 
\end{equation}
The statement follows from formula (\ref{C2Dnutheta}) after substituting expressions (\ref{KnudnuthetaA3}) and (\ref{partialomegathetaKndnutheta}). 
\end{proof}

It follows from Lemmas \ref{C1,C2}, \ref{C1,C2,next} that function $C|_{D_\nu}$ has (at most) the second order pole on $D_{\nu, \theta}$. In the next lemma we find the coefficient at this second order pole in its Laurent expansion.

\begin{lemma}\label{ingr1} Let $z=\theta^2 C.$ Then
\begin{align}\label{zfunction}
\left. z \right|_{D_{\nu, \theta}}&= 2 \lambda^{-4} \Bigg(  4B \frac{K_\theta^2}{\Pi^3} -2 B \lambda \frac{ K_\theta}{\Pi^2}  (\partial_{\omega^\lambda} \frac{K_\theta}{\Pi})  - 3 \widehat{B} \lambda  \frac{ K_\theta^2}{\Pi^2} -2 \frac{B \lambda^2}{\Pi}  (\partial_{\omega^\lambda} \frac{K_\theta}{\Pi})^2 +3 \widehat{B} \lambda^2 \frac{ K_\theta}{\Pi} \partial_{\omega^\lambda} \frac{K_\theta}{\Pi} \Bigg)\bigg\rvert_{D_{\nu, \theta}} \nonumber\\ & + 2\lambda^{-2}  \frac{B}{\Pi} \Big(  \partial_{\omega^\lambda} \frac{K_
\theta}{\Pi} -\frac{K_\theta}{\Pi}\lambda^{-1} \Big)^2 \bigg\rvert_{D_{\nu, \theta}}.
\end{align}
%
%where $\widehat{B}= \partial_{\omega^\lambda}\frac{B}{\Pi}$.  
Further to that, 
\begin{equation}\label{zto4}
\left.  \lambda^4  z \right|_{D} = 10 \left. \big(\frac{K_\theta}{\Pi}\big)^3 \right|_{D}.
\end{equation}
\end{lemma}

\begin{proof}
By Lemmas \ref{lemmaCitilde}, \ref{lemmaC2} we have
\begin{align*}
  z|_{D_{\nu, \theta}}&=  (  {\theta}^{-1} C_1 +C_2)|_{D_{\nu, \theta}}=   (\widetilde{C}_{1}+ C_2)|_{D_{\nu, \theta}} ,
\end{align*}
since $F$ is regular on $D_{\nu, \theta}$ by Lemma \ref{C1,C2,next}. This implies (\ref{zfunction}). Furthermore, 
\begin{equation*}
\left. \lambda^4 z \right|_{D}= \left.10 B \frac{K_\theta^2}{\Pi^3}\right|_{D},
\end{equation*}
which implies the relation \eqref{zto4} by Lemma \ref{structureKn}.
\end{proof}

Finally, we consider the term $E$ in the expression \eqref{detA} for $\operatorname{det}A$. Note that $\theta^2 E$ is well-defined generically at $\theta=0$. We obtain the following result. 

\begin{lemma}\label{ingr2}We have
\begin{equation}\label{theta^2E}
\left. \theta^2 E \right|_{D_{\nu, \theta}} =- \frac{18 K_\theta^2}{\Pi^2}   \partial_{\omega^\lambda} \Big( \frac{K_\lambda}{ \Pi}\lambda^{-3}\Big) \bigg\rvert_{D_{\nu, \theta}}. 
\end{equation}
Furthermore,
\begin{equation}
\left. \lambda^4  \theta^2 E \right|_{D} = 54 \big(\frac{K_\theta}{\Pi}\big)^3\bigg\rvert_{D},
\end{equation}
where we take restrictions on $\nu=0$ at first, then on $\theta=0$ and then on $\lambda=0$. 
\end{lemma}

\begin{proof}By Proposition \ref{A} we have
\begin{align*}
\left. \theta^2 E \right|_{D_{\nu, \theta}} &= 2 \lambda^2 \partial_{\omega^\lambda} \Big( \frac{K_\lambda}{ \Pi}\lambda^{-3}\Big) \Big( -8  \frac{K_\nu K_\theta}{ \Pi^2} \lambda^{-3} - \big(\frac{K_\nu}{\Pi}\big)^2 \lambda^{-4}\Big)\bigg\rvert_{D_{\nu, \theta}},
\end{align*}
which implies (\ref{theta^2E})  since  $\left. K_\nu \right|_{D_{\nu, \theta}}= \left. \lambda K_\theta \right|_{D_{\nu, \theta}}$ by Lemma \ref{structureKn}. Therefore 
\begin{align*}
\left.  \lambda^4  \theta^2 E \right|_{D}& = 54 \frac{K_\theta^2 K_\lambda}{\Pi^3} =  \left. 54 \big(\frac{K_\theta}{\Pi}\big)^3\right|_{D}
\end{align*}
by Lemma \ref{Klambda1}. 
\end{proof}

Using the above we have the following result.
\begin{theorem}\label{theoremA3}
The determinant of the metric $\eta_D$ is equal  to  $ -64 \Pi^{-1} K_\theta^3 $ restricted on $D$. 
\end{theorem}

\begin{proof}
We have
$
\operatorname{det}\eta_D = \left. - \bar{J}^2 \operatorname{det}A \right|_D = - \left.  \bar{J}^2 ( C+ E) \right|_D.
$
Note that $\theta^2 C$ and $\theta^2 E$ are well-defined generically on $D_{\nu, \theta}$, and $\left. \bar{J}\right|_{D_\nu} = \left. \lambda \theta (\lambda+\theta) \Pi \right|_{D_\nu}$. Hence we have
\begin{equation*}
\left. \bar{J}^2E \right|_{D_{\nu, \theta}} = \left. \lambda^2 \theta^2 (\lambda + \theta)^2 E \Pi^2 \right|_{D_{\nu, \theta}} = \left. \lambda^4 (\theta^2 E)\Pi^2 \right|_{D_{\nu, \theta}} .
\end{equation*}
By Lemma \ref{ingr2} we get $$\left. \bar{J}^2 E \right|_D = \left. 54 \frac{K_\theta^3}{\Pi}\right|_D.$$ Similarly we have $$\left. \bar{J}^2 C \right|_{D_{\nu, \theta}} = \left. \lambda^4 (\theta^2 C) \Pi^2 \right|_{D_{\nu, \theta}}.$$ By Lemma \ref{ingr1} we get
\begin{equation*}
\left. \bar{J}^2 C \right|_D = \left. 10 \frac{K_\theta^3}{\Pi} \right|_D, 
\end{equation*}
and the statement follows. 
\end{proof}

%For any $H \in \mathcal{A}$ let $\alpha_{H} \in \mathcal{R}$ be such that $H= \{ x \in V | \alpha_H(x)=0\}$. Similarly for any $H \in \mathcal{A}_{D_\theta}$ we choose $\alpha_H \in \mathcal{R}$ such that $H=\{x \in D_\theta | \alpha_H(x)=0\}$. It follows from Corollary \ref{factorJ} and formula \eqref{Jtheta} that
%

It follows by Corollary \ref{factorJ}  and formula \eqref{Jtheta}   that 

\begin{equation}\label{A3KthetaonD}
\left. K_\theta \right|_D  \sim \left.  {\mathcal I}(\mathcal{A}_{D_\theta} \setminus \mathcal{A}^D_{D_\theta}) \right|_D.
\end{equation}
Also we have that
\begin{equation}\label{A3PionD}
  \Pi   \sim  \mathcal{I}(\mathcal{A} \setminus \mathcal{A}^D).  
\end{equation}

Therefore Theorem \ref{theoremA3} can be reformulated as follows. 

\begin{theorem}\label{D3irr}
The determinant of the metric $\eta_D$ is proportional to 
\begin{equation}\label{detA3defining}
 \mathcal{I}(\mathcal{A}_{D_\theta} \setminus \mathcal{A}^D_{D_\theta})^3   \mathcal{I}(\mathcal{A} \setminus \mathcal{A}^D)^{-1}
\end{equation}
on the stratum $D$. The same is true with $\theta$ replaced with $\lambda$ or $\nu$.
\end{theorem}
Indeed, formula \eqref{detA3defining} is immediate from Theorem \ref{theoremA3} and formulae  \eqref{A3KthetaonD}, \eqref{A3PionD}. One can replace $\theta$ with $\lambda$ by Lemma \ref{Klambda1}. And one can replace $\theta$ with $\nu$ since $K_\nu  = \theta K_\lambda$ on $D_{\nu, \lambda}$ by Proposition \ref{relationJ}, which implies that $K_\lambda$ is proportional to ${\mathcal I}({\mathcal A}_{D_\nu}\setminus {\mathcal A}^D_{D_\nu})$ on $D_{\nu, \lambda}$ with the help of Corollary \ref{factorJ} applied to $J_{\nu}$.

Theorem \ref{D3irr} implies Theorem \ref{MMtheorem} for the stratum $D$ since  the root system ${\mathcal R}_D$ given by   \eqref{RDdecompo} is irreducible of rank 3.  

Let us now consider the cases where the root system $\mathcal{R}_D = \mathcal{R} \cap \langle \lambda, \nu, \theta \rangle$ is reducible, that is $\mathcal{R}_D=A_2 \times A_1$ or $\mathcal{R}_D=A_1^3$. 

\textbf{Strata of type} $\bf{A_2 \times A_1}$. 
Let us assume that the root system $\mathcal{R}_D=\mathcal{R}\cap \langle \lambda, \nu, \theta \rangle$ is a subsystem of $\mathcal{R}$ of type $A_2 \times A_1$ and consider the corresponding Coxeter subgraph
\begin{align*}
A_2 \times A_1 : \quad
\begin{dynkinDiagram}{A}{**}
\dynkinLabelRoots*{\lambda, \nu}
\end{dynkinDiagram}
\quad \quad
\begin{dynkinDiagram}{A}{*}
\dynkinLabelRoots*{\theta}
\end{dynkinDiagram}
\end{align*}
 Note that $\lambda+\nu \in \mathcal{R}_+$. The Jacobian can be represented as
\begin{equation}
J= \lambda \nu \theta (\lambda+\nu) \Pi, \label{JA2A1}
\end{equation}
where $\Pi$ is non-zero on $D$ and it satisfies  proportionality  
\beq{PionDA2A1}
\Pi \sim {\mathcal I}( \mathcal{A} \setminus \mathcal{A}^D). 
\eeq
%Note that $\Pi$ is non-zero on $D$. 
By Proposition \ref{divisible} we have
\begin{align}
J_{\lambda}&= \nu \theta K_{\lambda},\\
J_{\nu}&= \lambda \theta K_{\nu}, \label{KnuA2A1}\\
J_{\theta}&=  \lambda \nu ( \lambda +\nu) K_{\theta}\label{KthetaA2A1} \,
\end{align}
for some polynomials $K_{\lambda}, K_{\nu}, K_{\theta}  \in \mathbb{C}[x]$. We assume without loss of generality that $n+\sigma^{-1}(\lambda)$ is even, $\sigma^{-1}(\nu)=\sigma^{-1}(\lambda)+1$ and $\sigma^{-1}(\theta)-\sigma^{-1}(\lambda)$ is  even. This leads to the following expressions of components \eqref{notationJ_k} of the identity field 
\eqref{identityfieldrestated}.
\begin{proposition}\label{identityA2A1}
The $\lambda$, $\nu$ and $\theta$ components of the identity field $e$ are given by
\begin{align}
e^{\lambda}&= \frac{K_\lambda}{ \lambda(\lambda+\nu)\Pi}, \quad  e^{\nu} = - \frac{K_\nu}{\nu (\lambda+\nu)\Pi}, \quad \text{and} \quad e^{\theta}= \frac{K_\theta}{\theta \Pi}.
\end{align}
\end{proposition}

Let us introduce $\bar{J}= (\lambda+\nu)\Pi$ so that $J= \lambda \nu \theta \bar{J}$. Recall that in these notations $\operatorname{det}\eta_D$ is given by formula (\ref{detA3}). The entries of the matrix $A=(a_{ij})_{i,j=1}^3$ defined in (\ref{Amatrix}) are given as follows. 

\begin{proposition}\label{AmatrixpropA2A1} All the matrix entries $a_{ij}$ are well-defined generically on $D_{\lambda, \theta}$.They have the following form on $D_{\lambda, \theta}$:
\begin{align*}
a_{11}= \frac{2 K_\lambda}{\nu \Pi}, \quad 
 a_{22}=2 \partial_{\omega^{\nu}} \frac{K_\nu}{\Pi}- \frac{4 K_{\nu}}{\nu \Pi} , \quad 
  a_{33}= \frac{2 K_\theta}{\Pi},
  \end{align*}
  \begin{align*}
 a_{12}= \frac{K_\lambda}{\nu\Pi}  -  \partial_{\omega^{\nu}} \frac{K_\lambda}{\Pi} , \quad
 a_{23}= -  \nu  \partial_{\omega^{\nu}} \frac{K_\theta}{\Pi}, \quad a_{13}=0.
\end{align*}
\end{proposition}

\begin{proof}
By Proposition \ref{Saitom} we have $ \eta^{\alpha\beta} = - \partial_{\omega^\alpha}e^{\beta} - \partial_{\omega^\beta}e^{\alpha}$ for $\alpha, \beta \in \{ \lambda, \nu, \theta\}$. The statement follows from Proposition \ref{identityA2A1}.
\end{proof}

%For any $H \in \mathcal{A}$ let $\alpha_H \in \mathcal{R}$ be such that  $H= \{ x \in V| \alpha_{H}(x)=0\}$. Similarly for any $H \in \mathcal{A}_{D_\gamma}$, $\gamma \in \{\theta, \nu\}$ we choose $\alpha_H \in \mathcal{R}$ such that  $H= \{ x \in D_\gamma | \alpha_{H}(x)=0\}$.  
It follows from Corollary \ref{factorJ} and formulae \eqref{KnuA2A1}, \eqref{KthetaA2A1} that
\begin{equation}\label{KnuonDA2A1}
\left. K_\nu \right|_D   \sim \left. {\mathcal I}( \mathcal{A}_{D_\nu} \setminus \mathcal{A}^D_{D_\nu}) \right|_D,
\end{equation}
and 
\begin{equation}\label{KthetaonDA2A1}
\left. K_\theta \right|_D \sim  \left. {\mathcal I}( \mathcal{A}_{D_\theta} \setminus \mathcal{A}^D_{D_\theta}) \right|_D.
\end{equation}

We are ready to establish Theorem \ref{MMtheorem} for $D$ which can be formulated as follows.

\begin{theorem}\label{theoremcodimensionA2A1saito}
The determinant of the metric $\eta_D$ is proportional to
\begin{equation}\label{codA2A1defin}
{\mathcal I}( \mathcal{A}_{D_\nu} \setminus \mathcal{A}^D_{D_\nu})^2  {\mathcal I}( \mathcal{A}_{D_\theta} \setminus \mathcal{A}^D_{D_\theta})  {\mathcal I}(\mathcal{A} \setminus \mathcal{A}^D)^{-1}
\end{equation}
on $D$. The same is true with $\nu$ replaced with $\lambda$ in \eqref{codA2A1defin}.
\end{theorem}

\begin{proof}By formula (\ref{detA3}) we have $\operatorname{det}\eta_D= -\left. \bar{J}^2 \operatorname{det}A \right|_D$, where $A$ is given by \eqref{Amatrix}. Therefore by Proposition \ref{AmatrixpropA2A1}
\begin{align*}
\operatorname{det}\eta_D&= \left. \big( (a_{12}^2-a_{11} a_{22} )a_{33} + a_{11} a_{23}^2 \big)(\lambda+\nu)^2 \Pi^2 \right|_D=   \left. 16 \frac{K_{\lambda} K_{\nu} K_\theta}{\Pi} \right|_D + \left. 2 \frac{K_{\lambda}^2 K_\theta}{\Pi}\right|_D. 
\end{align*}
By Proposition \ref{relationJ},  we have $\frac{J_\lambda}{\nu}= \frac{J_\nu}{\lambda}$ on $D_{\lambda, \nu}$ and hence $\left. K_\nu \right|_{D_{\lambda, \nu}} = \left. K_{\lambda} \right|_{D_{\lambda, \nu}}$.  Therefore $\operatorname{det}\eta_D =18 \Pi^{-1} K_\nu^2 K_\theta$ on $D$. The statement follows by formulae \eqref{KnuonDA2A1}, \eqref{KthetaonDA2A1} and \eqref{PionDA2A1}.
\end{proof}

\textbf{Strata of type} $\bf{A_1^3}$. 
Let us assume that the root system $\mathcal{R}_D=\mathcal{R}\cap \langle \lambda, \nu, \theta \rangle$ is a subsystem of $\mathcal{R}$ of type $A_1^3$ and consider the corresponding Coxeter subgraph
\begin{align*}
A_1^3 : \quad
\begin{dynkinDiagram}{A}{*}
\dynkinLabelRoots*{\lambda}
\end{dynkinDiagram}
\quad\quad
\begin{dynkinDiagram}{A}{*}
\dynkinLabelRoots*{\nu}
\end{dynkinDiagram}
\quad\quad
\begin{dynkinDiagram}{A}{*}
\dynkinLabelRoots*{\theta}
\end{dynkinDiagram}
\end{align*}
The Jacobian $J$ can be represented as 
\begin{equation} \label{JA1^3}
J= \lambda \nu \theta \Pi,
\end{equation} 
where $\Pi$ is non-zero on $D$ and it satisfies the proportionality
\beq{PionDA1^3}
\Pi \sim {\mathcal I}(\mathcal{A} \setminus \mathcal{A}^D).
\eeq
 By Proposition \ref{divisible} we have
\begin{align}
J_{\lambda}&= \nu \theta K_{\lambda}, \label{JlambdaA1^3}\\
J_{\nu}&= \lambda \theta K_{\nu}, \label{JnuA1^3}\\
J_{\theta}&= \lambda \nu K_{\theta}\label{JthetaA1^3}
\end{align}
for some polynomials $K_\lambda, K_\nu, K_\theta \in \mathbb{C}[x]$. We assume without loss of generality that $n + \sigma^{-1}(\gamma)$ is even for any $\gamma \in \{ \lambda, \nu, \theta\}$. This leads to the following expressions of components \eqref{notationJ_k} of the identity field 
\eqref{identityfieldrestated}.
%This leads to the following expressions of components of the identity field $e$ by Proposition \ref{identity}.
\begin{proposition}\label{identityA1^3}The $\lambda$, $\nu$ and $\theta$ components of the identity field $e$ are given by
\begin{equation}
e^{\lambda}=\frac{K_{\lambda}}{\lambda \Pi}, \quad e^{\nu}= \frac{K_{\nu}}{\nu \Pi}, \quad \text{and} \quad e^{\theta}= \frac{K_{\theta}}{\theta\Pi}.
\end{equation}
\end{proposition}

Let us introduce $\bar{J}= \Pi$ so that $J= \lambda \nu \theta \bar{J}$. Recall that in these notations $\operatorname{det}\eta_D$ is given by formula (\ref{detA3}). The entries of the matrix $A=(a_{ij})_{i,j=1}^3$ defined in \eqref{Amatrix} are given as follows. 

\begin{proposition}\label{propoAmatrixA1^3}All the matrix entries $a_{ij}$ ($1 \leq i, j \leq 3$) are well-defined generically on $D$. They have the following form on $D$:
\begin{align*}
a_{11}= \frac{2K_{\lambda}}{\Pi}, \hspace{1cm} a_{22}&= \frac{2K_{\nu}}{\Pi},\hspace{1cm} a_{33}=\frac{2K_{\theta}}{\Pi}, \hspace{1cm} % a_{12}= a_{13}= a_{23}=0
a_{ij}=0 \quad \text{if} \quad i \neq j.
\end{align*}
\end{proposition}
\begin{proof}
By Proposition \ref{Saitom} we have $ \eta^{\alpha \beta} = - \partial_{\omega^\alpha}e^{\beta} - \partial_{\omega^\beta}e^{\alpha}$ for $\alpha, \beta \in \{ \lambda, \nu, \theta\}$. The statement follows by  Proposition \ref{identityA1^3}. 
\end{proof}
 
  It follows from Corollary \ref{factorJ} and formulae \eqref{JlambdaA1^3}-- \eqref{JthetaA1^3} that
\begin{equation}\label{KnuonDA1^3}
  K_\gamma |_D   \sim  {\mathcal I}( \mathcal{A}_{D_\gamma} \setminus \mathcal{A}^D_{D_\gamma})|_D
\end{equation}
for $\gamma =  \lambda,  \nu, \theta$.
Now we can prove Theorem \ref{MMtheorem} for $D$.

\begin{theorem}\label{theoremsaitodetA_1^3}
The determinant of the metric $\eta_D$ is proportional to 
\begin{equation}%\label{theoremA1^3detdefin}
{\mathcal I}(\mathcal{A}\setminus \mathcal{A}^D)^{-1} \prod_{\gamma \in \{ \lambda, \nu, \theta\}} {\mathcal I}( \mathcal{A}_{D_\gamma} \setminus \mathcal{A}^D_{D_\gamma}). 
\end{equation} 
\end{theorem}

\begin{proof}By formula (\ref{detA3}) we have $\operatorname{det}\eta_D= -\left.  \bar{J}^2 \operatorname{det}A \right|_D$, where $A$ is given by formula \eqref{Amatrix}. Therefore by Proposition \ref{propoAmatrixA1^3} we get
\begin{align*}
\operatorname{det}\eta_D&= \left. - \bar{J}^2 \operatorname{det}A \right|_D =\left. -  a_{11}a_{22}a_{33}\Pi^2 \right|_D  =-  \left.   \frac{8 K_{\lambda}K_{\nu} K_{\theta}}{\Pi} \right|_D,
\end{align*}
and the statement follows by formulae \eqref{PionDA1^3}, \eqref{KnuonDA1^3}.
\end{proof}

\section{Exceptional groups: the remaining cases}\label{exceptionalsectionrem}
Determinant of the metric $\eta_D$ for the strata $D$ of codimension 4 in the simply laced groups can be analyzed similarly to the codimension 3 case considered in Subsection \ref{codim3}. Considerations become more technical and we refer to \cite{GA} where this is done in detail.
 
In this section we obtain formulae for the determinant of the restricted Saito metric 
%and analyse the corresponding multiplicities 
for the remaining cases with the help of Mathematica. Thus we consider codimension $5$ strata for $\mathcal{R}=E_7$ and codimension $5$ and $ 6$ strata for $\mathcal{R}=E_8$. We consider Saito metric and use Saito polynomials for these root systems $\mathcal{R}=E_n$, $n=7, 8$. These are explicitly constructed in \cite{flat} and also in \cite{DAB}. Note that in these cases discriminant strata are fully determined by the type of subgraphs of the Coxeter graph of $\mathcal{R}$ defined by $S \subset \Delta$ with the exception of $\mathcal{R}=E_7$ which has two non-equivalent strata of type $A_5$ (see \cite{OT}).

Let us start with the case $n=8$, $\mathcal{R}=E_8$. We use Saito polynomials from \cite{flat} which are written in terms of coordinates $y_i$ ($i=1, \dots, 8$) (denoted as $x_i$ in \cite{flat}) defined by 
\begin{align}\label{coordsyiE8}
y_i= \begin{dcases} \frac{1}{2}(x^i + x^{i+1}), & i \text{ odd}, \\  \frac{1}{2}(x^{i-1} - x^{i}),& i \text{ even} . \end{dcases}
\end{align}
Let us recall the  root system $E_8 \subset V=\mathbb{C}^8$ (see e.g. \cite{cox}):
\begin{align*}
\pm e_i \pm e_j, \quad 1 \leq i < j \leq 8, \quad \frac{1}{2} (\pm e_1 \pm e_2 \pm \dots \pm e_8) \quad \text{(even number of minuses)}.
\end{align*}
Let us fix the following simple system $\Delta \subset E_8$: 
\begin{align}\label{simplesystemE8}
\alpha_1 &= \frac{1}{2}(e_1 - e_2 -e_3 -e_4 -e_5 -e_6-e_7 +e_8), \nonumber\\
\alpha_2&= e_1 +e_2,\\
\alpha_i&= e_{i-1}-e_{i-2}, \quad 3 \leqslant i \leqslant 8, \nonumber
\end{align}
and consider the corresponding Coxeter graph:
$$\dynkin[labels={1,2,3, 4, 5, 6, 7, 8},label macro/.code={\alpha_{#1}}]{E}{8}$$ 
Let us also introduce coordinates $z_i = (\alpha_i, x)$, $1 \leq i \leq 8$. Note that $z_i= \sum_{j=1}^8 a^{(8)}_{ij} y_j$, where $A=A^{(8)}=(a^{(8)}_{ij})_{i,j=1}^8$ is the following matrix:
\begin{align*}
A^{(8)}=\begin{pmatrix}
0 & 1 &-1& 0&-1 &0&0&-1 \\
2 & 0 &0& 0&0 &0&0&0\\
0 & -2 &0& 0&0 &0&0&0\\
-1 & 1 &1& 1&0 &0&0&0\\
0 & 0 &0& -2&0 &0&0&0\\
0 & 0 &-1& 1&1 &1&0&0\\
0 &0 &0& 0&0 &-2&0&0\\
0 & 0 &0& 0&-1 &1&1&1
\end{pmatrix}.
\end{align*}
We have 
\begin{equation*}
\eta=\sum_{i=1}^n dt^i dt^{n+1-i}=\sum_{i=1}^n  \sum_{r=1}^n  \sum_{l=1}^n\frac{\partial t^i}{\partial y_r}  \frac{\partial t^{n+1-i}}{\partial y_l}dy_r dy_{l} =\sum_{r, l=1}^n \eta'_{rl}(y) dy_r dy_l. 
\end{equation*}
In $z$-coordinates we have $\eta= \sum_{i, j=1}^n \eta''_{ij}(z) dz_i dz_j$, where 
\begin{align}\label{saito1forE7,8}
\eta_{ij}''(z)=\sum_{k,l=1}^n (A^{(n)})^{-1}_{ki} (A^{(n)})^{-1}_{lj} \eta'_{kl}(y).
\end{align}

 Let $I=\{i_1, \dots, i_k\}$, $1 \leqslant i_1< \dots < i_k \leqslant n$ and let $J=  \underline{n} \setminus I$. Consider the corresponding stratum $D= D_{i_1, \dots, i_k}$. It is given by equations $z_i=0$, where $i \in I$. The restriction of $\eta$ on $D$ takes the form
\begin{equation}\label{saitoforE7,8}
\eta_D= \sum_{l, k \in J} \left. \eta_{lk}''(z)\right|_D  dz_{l} dz_{k}.
\end{equation}
We use formula \eqref{saitoforE7,8} to find the determinant of the restricted Saito metric with the help of Mathematica. Tables \ref{codimension 5, E_8} and \ref{codimension 6 in E_8}  below give $\operatorname{det}\eta_D$ up to a non-zero proportionality factor for all three-dimensional and two-dimensional strata $D$ in $E_8$ respectively. We list types of strata $\mathcal{R}_D= \mathcal{R} \cap \langle S \rangle$, where $S=\{\alpha_{i_1}, \dots, \alpha_{i_k}\} \subset \mathcal{R}$ in the first column of these tables. We use the notation $\{i_{1}, \dots, i_{k}\}$ to denote the stratum $D$. We get the following statement by inspecting the tables. 

\begin{theorem}
Let $D$ be any  three-dimensional  or two-dimensional stratum for $\mathcal{R}=E_8$. Then the statement of Theorem \ref{thma} is true. 
\end{theorem}

\begin{table}[h]
\caption{ Determinant of the restricted Saito metric, $\operatorname{dim}D=3$, $\mathcal{R}=E_8$}
\label{codimension 5, E_8}
\footnotesize
\renewcommand{\arraystretch}{1.5}
\begin{tabular}{ |m{2cm}|m{13.2cm}|  }
 \hline
\centering $\mathcal{R}_D$, $S$ &  \hspace{5.8cm} $\operatorname{det}\eta_D$ \\
 \hline
\centering $A_5$, \\ $\{4, 5, 6, 7, 8\}$ & $\alpha_1^2 \alpha_2^7 \alpha_3^7 \left(\alpha_1+\alpha_3\right){}^7 \left(\alpha_2+\alpha_3\right){}^{10} \left(\alpha_1+\alpha_2+\alpha_3\right){}^{10} \left(\alpha_1+\alpha_2+2 \alpha_3\right){}^{12} \times $ $\times  \left(\alpha_1+2 \alpha_2+3 \alpha_3\right){}^7  \left(2 \alpha_1+2 \alpha_2+3 \alpha_3\right){}^7 \left(2 \alpha_1+3 \alpha_2+4 \alpha_3\right){}^7 \left(\alpha_1+2 \left(\alpha_2+\alpha_3\right)\right){}^{10}\times $ $\times  \left(\alpha_1+3 \left(\alpha_2+\alpha_3\right)\right){}^2  \left(2 \alpha_1+3 \left(\alpha_2+\alpha_3\right)\right){}^2$   \\
\hline
 \centering $D_5$, \\ $\{1, 2, 3, 4, 5\}$ &  $\alpha_6^{12} \alpha_7^2 \alpha_8^2\left(\alpha_6+\alpha_7\right){}^{12} \left(2 \alpha_6+\alpha_7\right){}^{10}  \left(\alpha_7+\alpha_8\right){}^2 \left(\alpha_6+\alpha_7+\alpha_8\right){}^{12} \left(2 \alpha_6+\alpha_7+\alpha_8\right){}^{10} \times $ $\times \left(2 \alpha_6+2 \alpha_7+\alpha_8\right){}^{10} \left(3 \alpha_6+2 \alpha_7+\alpha_8\right){}^{12} \left(4 \alpha_6+2 \alpha_7+\alpha_8\right){}^2 \left(4 \alpha_6+3 \alpha_7+\alpha_8\right){}^2 \times $ $\times \left(4 \alpha_6+3 \alpha_7+2 \alpha_8\right){}^2$
  \\
  \hline
 \centering  $D_4 \times A_1$,  \\ $\{2, 3, 4, 5, 7\}$ & $\alpha_1^8 \alpha_6^{10} \alpha_8^3 \left(\alpha_1+\alpha_6\right){}^{10} \left(\alpha_1+2 \alpha_6\right){}^8  \left(\alpha_6+\alpha_8\right){}^8 \left(\alpha_1+\alpha_6+\alpha_8\right){}^8 \left(2 \alpha_6+\alpha_8\right){}^3 \times$ $\times \left(\alpha_1+2 \alpha_6+\alpha_8\right){}^{10} \left(2 \alpha_1+2 \alpha_6+\alpha_8\right){}^3   \left(\alpha_1+3 \alpha_6+\alpha_8\right){}^8 \left(2 \alpha_1+3 \alpha_6+\alpha_8\right){}^8\times$ $\times \left(2 \alpha_1+4 \alpha_6+\alpha_8\right){}^3$ \\
  \hline
\centering  $A_4 \times A_1$,  \\ $\{1, 3, 4, 5, 7\}$& $\alpha_2^8 \alpha_6^7 \alpha_8^3\left(\alpha_2+\alpha_6\right){}^{12} \left(2 \alpha_2+\alpha_6\right){}^3 \left(\alpha_2+2 \alpha_6\right){}^6  \left(\alpha_6+\alpha_8\right){}^6 \left(\alpha_2+\alpha_6+\alpha_8\right){}^8\times $ $ \times \left(2 \alpha_2+\alpha_6+\alpha_8\right){}^2  \left(\alpha_2+2 \alpha_6+\alpha_8\right){}^7 \left(2 \alpha_2+2 \alpha_6+\alpha_8\right){}^7 \left(\alpha_2+3 \alpha_6+\alpha_8\right){}^2 \times $ $ \times\left(2 \alpha_2+3 \alpha_6+\alpha_8\right){}^8  \left(3 \alpha_2+3 \alpha_6+\alpha_8\right){}^6  \left(3 \alpha_2+4 \alpha_6+\alpha_8\right){}^3 \left(3 \alpha_2+4 \alpha_6+2 \alpha_8\right){}^2$\\
 \hline
 \centering $A_3 \times A_2$,  \\ $\{2, 3, 4, 6, 7\}$& $\alpha_1^5 \alpha_5^{10}  \alpha_8^4 \left(\alpha_1+\alpha_5\right){}^7 \left(\alpha_1+2 \alpha_5\right){}^7 \left(\alpha_1+3 \alpha_5\right){}^5 \left(2 \alpha_1+3 \alpha_5\right){}^2 \left(\alpha_5+\alpha_8\right){}^6 \times $ $ \times \left(\alpha_1+\alpha_5+\alpha_8\right){}^5  \left(2 \alpha_5+\alpha_8\right){}^4 \left(\alpha_1+2 \alpha_5+\alpha_8\right){}^7 \left(\alpha_1+3 \alpha_5+\alpha_8\right){}^7 \times $ $ \times  \left(2 \alpha_1+3 \alpha_5+\alpha_8\right){}^4\left(\alpha_1+4 \alpha_5+\alpha_8\right){}^5  \left(2 \alpha_1+4 \alpha_5+\alpha_8\right){}^6 \left(2 \alpha_1+5 \alpha_5+\alpha_8\right){}^4  \times $ $ \times  \left(2 \alpha_1+5 \alpha_5+2 \alpha_8\right){}^2$\\
 \hline
\centering  $A_3 \times A_1^2$,  \\ $\{2, 3, 5, 6, 7\}$ & $\alpha_1^3 \alpha_4^{10}  \alpha_8^5\left(\alpha_1+\alpha_4\right){}^6 \left(\alpha_1+2 \alpha_4\right){}^8 \left(\alpha_1+3 \alpha_4\right){}^6 \left(\alpha_1+4 \alpha_4\right){}^3  \left(\alpha_4+\alpha_8\right){}^4 \times $ $ \times  \left(\alpha_1+\alpha_4+\alpha_8\right){}^3  \left(2 \alpha_4+\alpha_8\right){}^5 \left(\alpha_1+2 \alpha_4+\alpha_8\right){}^6 \left(\alpha_1+3 \alpha_4+\alpha_8\right){}^8 \left(\alpha_1+4 \alpha_4+\alpha_8\right){}^6 \times $ $ \times \left(2 \alpha_1+4 \alpha_4+\alpha_8\right){}^5   \left(\alpha_1+5 \alpha_4+\alpha_8\right){}^3 \left(2 \alpha_1+5 \alpha_4+\alpha_8\right){}^4  \left(2 \alpha_1+6 \alpha_4+\alpha_8\right){}^5$\\
 \hline
\centering  $A_2^2 \times A_1$,  \\ $\{1, 2, 3, 5, 6\}$ & $\alpha_4^{12} \alpha_7^4 \alpha_8^2 \left(\alpha_4+\alpha_7\right){}^5 \left(2 \alpha_4+\alpha_7\right){}^6 \left(3 \alpha_4+\alpha_7\right){}^5 \left(4 \alpha_4+\alpha_7\right){}^4  \left(\alpha_7+\alpha_8\right){}^4 \left(\alpha_4+\alpha_7+\alpha_8\right){}^5 \times $ $\times \left(2 \alpha_4+\alpha_7+\alpha_8\right){}^6 \left(3 \alpha_4+\alpha_7+\alpha_8\right){}^5 \left(4 \alpha_4+\alpha_7+\alpha_8\right){}^4 \left(2 \alpha_4+2 \alpha_7+\alpha_8\right){}^4 \times $ $ \times \left(3 \alpha_4+2 \alpha_7+\alpha_8\right){}^5   \left(4 \alpha_4+2 \alpha_7+\alpha_8\right){}^6 \left(5 \alpha_4+2 \alpha_7+\alpha_8\right){}^5 \left(6 \alpha_4+2 \alpha_7+\alpha_8\right){}^4 \times $ $ \times  \left(6 \alpha_4+3 \alpha_7+\alpha_8\right){}^2 \left(6 \alpha_4+3 \alpha_7+2 \alpha_8\right){}^2$\\
 \hline
\centering $A_2 \times A_1^3$, \\ $\{2, 3, 5, 7, 8\}$& $\alpha_1^3 \alpha_4^6 \alpha_6^5 \left(\alpha_1+\alpha_4\right){}^4 \left(\alpha_1+2 \alpha_4\right){}^3  \left(\alpha_4+\alpha_6\right){}^8 \left(\alpha_1+\alpha_4+\alpha_6\right){}^5 \left(2 \alpha_4+\alpha_6\right){}^5  \times $ $ \times \left(\alpha_1+2 \alpha_4+\alpha_6\right){}^8 \left(\alpha_1+3 \alpha_4+\alpha_6\right){}^5 \left(\alpha_1+3 \alpha_4+2 \alpha_6\right){}^8 \left(\alpha_1+4 \alpha_4+2 \alpha_6\right){}^5\times $ $ \times  \left(\alpha_1+4 \alpha_4+3 \alpha_6\right){}^4  \left(2 \alpha_1+4 \alpha_4+3 \alpha_6\right){}^3  \left(\alpha_1+5 \alpha_4+3 \alpha_6\right){}^3 \left(2 \alpha_1+5 \alpha_4+3 \alpha_6\right){}^4 \times $ $ \times\left(2 \alpha_1+6 \alpha_4+3 \alpha_6\right){}^3  \left(\alpha_1+2 \left(\alpha_4+\alpha_6\right)\right){}^5 \left(\alpha_1+3 \left(\alpha_4+\alpha_6\right)\right){}^3$ \\
 \hline
\end{tabular}
\end{table}

\begin{table}[h]
\caption{Determinant of the restricted Saito metric, $\operatorname{dim}D=2$, $\mathcal{R}=E_8$}
\label{codimension 6 in E_8}
\footnotesize
\renewcommand{\arraystretch}{1.5}
\begin{tabular}{ |P{2.2cm}|P{13cm}|  }
 \hline
\centering $\mathcal{R}_D$, $S$ & $\operatorname{det}\eta_D$ \\
 \hline 
\centering  $ A_6$,\\ $\{2, 4, 5, 6, 7, 8\}$& $\alpha_1^2 \alpha_3^{12} \left(\alpha_1+\alpha_3\right){}^{12} \left(\alpha_1+2 \alpha_3\right){}^{18} \left(\alpha_1+3 \alpha_3\right){}^8 \left(2 \alpha_1+3 \alpha_3\right){}^8$  \\
\hline
\centering $D_6$,\\  $\{2, 3, 4, 5, 6, 7\}$ &  $\alpha_1^{18} \alpha_8^{12} \left(\alpha_1+\alpha_8\right){}^{18} \left(2 \alpha_1+\alpha_8\right){}^{12}$
  \\
  \hline
\centering   $E_6$,\\  $\{1, 2,3,4, 5, 6\}$ & $\alpha_7^{18} \alpha_8^2 \left(\alpha_7+\alpha_8\right){}^{18} \left(2 \alpha_7+\alpha_8\right){}^{18} \left(3 \alpha_7+\alpha_8\right){}^2 \left(3 \alpha_7+2 \alpha_8\right){}^2$ \\
  \hline
 \centering  $A_5 \times A_1$, \\ $\{1, 3, 4, 5, 6, 8\}$ & $\alpha_2^{12} \alpha_7^8 \left(\alpha_2+\alpha_7\right){}^{18} \left(2 \alpha_2+\alpha_7\right){}^8 \left(\alpha_2+2 \alpha_7\right){}^7 \left(3 \alpha_2+2 \alpha_7\right){}^7$\\
 \hline
\centering $D_5 \times A_1$, \\ $\{2, 3, 4, 5, 6, 8\}$ & $\alpha_1^{12} \alpha_7^{12} \left(\alpha_1+\alpha_7\right){}^{18} \left(2 \alpha_1+\alpha_7\right){}^3 \left(\alpha_1+2 \alpha_7\right){}^{12} \left(2 \alpha_1+3 \alpha_7\right){}^3$\\
 \hline
 \centering  $A_4 \times A_2$, \\ $\{1, 3, 4, 5, 7, 8\}$ & $\alpha_2^8 \alpha_6^8 \left(\alpha_2+\alpha_6\right){}^{18} \left(2 \alpha_2+\alpha_6\right){}^4 \left(\alpha_2+2 \alpha_6\right){}^8 \left(\alpha_2+3 \alpha_6\right){}^2 \left(2 \alpha_2+3 \alpha_6\right){}^8 \left(3 \alpha_2+4 \alpha_6\right){}^4$\\
  \hline
 \centering  $D_4 \times A_2$,\\  $\{2, 3, 4, 5, 7, 8\}$ & $\alpha_1^8 \alpha_6^{12} \left(\alpha_1+\alpha_6\right){}^{12} \left(\alpha_1+2 \alpha_6\right){}^{12} \left(\alpha_1+3 \alpha_6\right){}^8 \left(2 \alpha_1+3 \alpha_6\right){}^8$\\
  \hline
  \centering  $A_4 \times A_1^2$,\\ $\{2, 3, 5, 6, 7, 8\}$ & $\alpha_1^3 \alpha_4^{12} \left(\alpha_1+\alpha_4\right){}^7 \left(\alpha_1+2 \alpha_4\right){}^{12} \left(\alpha_1+3 \alpha_4\right){}^{12} \left(\alpha_1+4 \alpha_4\right){}^7 \left(\alpha_1+5 \alpha_4\right){}^3 \left(2 \alpha_1+5 \alpha_4\right){}^4$\\
  \hline
 \centering  $A_3^2$,\\ $\{2, 3, 4, 6, 7, 8\}$ & $\alpha_1^5 \alpha_5^{12} \left(\alpha_1+\alpha_5\right){}^8 \left(\alpha_1+2 \alpha_5\right){}^{12} \left(\alpha_1+3 \alpha_5\right){}^8 \left(2 \alpha_1+3 \alpha_5\right){}^5 \left(\alpha_1+4 \alpha_5\right){}^5 \left(2 \alpha_1+5 \alpha_5\right){}^5$\\
  \hline
 \centering   $A_3 \times A_2 \times A_1$,\\ $\{1, 2, 4, 6, 7, 8\}$ & $\alpha_3^5 \alpha_5^7 \left(\alpha_3+\alpha_5\right){}^{18} \left(2 \alpha_3+\alpha_5\right){}^5 \left(\alpha_3+2 \alpha_5\right){}^8 \left(2 \alpha_3+3 \alpha_5\right){}^7 \left(3 \alpha_3+4 \alpha_5\right){}^5 \left(4 \alpha_3+5 \alpha_5\right){}^5$\\
  \hline
 \centering  $A_2^2\times  A_1^2$,\\ $\{1, 2, 3, 5, 6, 8\}$ & $\alpha_4^{12} \alpha_7^5 \left(\alpha_4+\alpha_7\right){}^8 \left(2 \alpha_4+\alpha_7\right){}^{12} \left(3 \alpha_4+\alpha_7\right){}^8 \left(4 \alpha_4+\alpha_7\right){}^5 \left(3 \alpha_4+2 \alpha_7\right){}^5   \left(5 \alpha_4+2 \alpha_7\right){}^5$\\
  \hline
\end{tabular}
\end{table}

Let us now consider the case $n=7$, $\mathcal{R}=E_7 \subset V=\mathbb{C}^8$. The root system $E_7$ contains the following vectors (see e.g. \cite{cox}):
\begin{align*}
\pm e_i \pm e_j, \quad 1 \leq i < j \leq 6,  \quad \pm(e_7 -e_8), \quad \pm\frac{1}{2} (e_7 - e_8  + \sum_{i=1}^6 \epsilon_i e_i),
\end{align*}
where $\epsilon_i =\pm 1$, $\prod_{i=1}^6 \epsilon_i = -1$.
%the number of minus signs in the sum is odd. 
Let us fix simple system $\Delta= \{\alpha_1, \dots, \alpha_7\}$, where $\alpha_i$, $i=1, \dots, 7$ are defined in \eqref{simplesystemE8}.
%and consider the corresponding Coxeter graph:
%
%$$\dynkin[labels={1,2,3, 4, 5, 6, 7},label macro/.code={\alpha_{#1}}]{E}{7}$$ 

We use Saito polynomials from \cite{flat} which are written in terms of coordinates $y_i$ defined by formulae \eqref{coordsyiE8} for any $ 1 \leq i \leq 4$ and defined by the following formulae for $i=5, 6, 7$:
\begin{align*}
y_i= \begin{dcases} \frac{1}{2}(x^i - x^{i+1}), & i=5,7, \\  \frac{1}{2}(x^{i-1} + x^{i}),& i=6.  \end{dcases}
\end{align*}
 Let us also consider coordinates $z_i = (\alpha_i, x)$, $1 \leq i \leq 7$. Note that $z_i= \sum_{j=1}^7 a^{(7)}_{ij} y_j$, where $A=A^{(7)}=(a^{(7)}_{ij})_{i,j=1}^7$ is the following matrix:
\begin{align*}
A^{(7)}=\begin{pmatrix}
0 & 1 &-1& 0&0 &-1&-1 \\
2 & 0 &0& 0&0 &0&0\\
0 & -2 &0& 0&0 &0&0\\
-1 & 1 &1& 1&0 &0&0\\
0 & 0 &0& -2&0 &0&0\\
0 & 0 &-1& 1&1 &1&0\\
0 &0 &0& 0&-2 &0&0\\
\end{pmatrix}.
\end{align*}
We use formulae \eqref{saito1forE7,8}, \eqref{saitoforE7,8} (with $n=7$) to find the determinant of the restricted Saito metric with the help of Mathematica. Table \ref{codimension 5 in E_7} gives $\operatorname{det}\eta_D$ up to a non-zero proportionality factor for any two-dimensional stratum $D$ in $E_7$. We list types of strata $\mathcal{R}_D= \mathcal{R} \cap \langle S \rangle$, where $S=\{\alpha_{i_1}, \dots, \alpha_{i_k}\} \subset \mathcal{R}$ in the first column of this table. We use the notation $\{i_1, \dots, i_k\}$ to denote the stratum $D$. Note that there are two non-equivalent strata of type $A_5$ \cite{OT}. The following statement is a direct consequence of this table. 

\begin{theorem}
Let $D$ be any two-dimensional stratum for $\mathcal{R}=E_7$. Then the statement of Theorem \ref{thma} is true. 
\end{theorem}

\begin{table}[h]
\caption{Determinant of the restricted Saito metric, $\operatorname{dim}D=2$, $\mathcal{R}=E_7$}
\label{codimension 5 in E_7}
\footnotesize
\renewcommand{\arraystretch}{1.5}
\begin{tabular}{ |c|c | }
 \hline
 $\mathcal{R}_D$, $S$&  $\operatorname{det}\eta_D$ \\
 \hline 
$ A_5$, $\{2, 4, 5, 6, 7\}$ & $\alpha_1^2 \alpha_3^{10} \left(\alpha_1+\alpha_3\right){}^{10} \left(\alpha_1+2 \alpha_3\right){}^{10} \left(\alpha_1+3 \alpha_3\right){}^2 \left(2 \alpha_1+3 \alpha_3\right){}^2$  \\
\hline
 $A^{'}_5$,  $\{3, 4, 5, 6, 7\}$ &  $\alpha_1^7 \alpha_2^{10} \left(\alpha_1+\alpha_2\right){}^{12} \left(\alpha_1+2 \alpha_2\right){}^7$
  \\
  \hline
  $D_5$,  $\{1, 2, 3, 4, 5\}$ & $\alpha_6^{12} \alpha_7^2 \left(\alpha_6+\alpha_7\right){}^{12} \left(2 \alpha_6+\alpha_7\right){}^{10}$ \\
  \hline
  $A_4 \times A_1$,  $\{1, 2, 3, 4, 7\}$ & $\alpha_5^8 \alpha_6^3 \left(\alpha_5+\alpha_6\right){}^{12} \left(2 \alpha_5+\alpha_6\right){}^7 \left(3 \alpha_5+2 \alpha_6\right){}^6$\\
 \hline
  $D_4 \times A_1$,  $\{2, 3, 4, 5, 7\}$ & $\alpha_1^8 \alpha_6^{10} \left(\alpha_1+\alpha_6\right){}^{10} \left(\alpha_1+2 \alpha_6\right){}^8$\\
 \hline
  $A_3 \times A_2$,  $\{1, 3, 5, 6, 7\}$ & $\alpha_2^2 \alpha_4^7 \left(\alpha_2+\alpha_4\right){}^7 \left(\alpha_2+2 \alpha_4\right){}^{10} \left(\alpha_2+3 \alpha_4\right){}^5 \left(2 \alpha_2+3 \alpha_4\right){}^5$\\
  \hline
  $A_3 \times A_1^2$,  $\{1, 2, 4, 5, 7\}$ & $\alpha_3^8 \alpha_6^6 \left(\alpha_3+\alpha_6\right){}^{10} \left(2 \alpha_3+\alpha_6\right){}^6 \left(\alpha_3+2 \alpha_6\right){}^3 \left(3 \alpha_3+2 \alpha_6\right){}^3$\\
  \hline
  $A_2^2 \times A_1$ ,  $\{1, 2, 4, 6, 7\}$& $\alpha_3^5 \alpha_5^6 \left(\alpha_3+\alpha_5\right){}^{12} \left(2 \alpha_3+\alpha_5\right){}^4 \left(\alpha_3+2 \alpha_5\right){}^5 \left(2 \alpha_3+3 \alpha_5\right){}^4$\\
  \hline
  $A_2 \times A_1^3$,  $\{1, 2, 3, 5, 7\}$ & $\alpha_4^8 \alpha_6^4 \left(\alpha_4+\alpha_6\right){}^8 \left(2 \alpha_4+\alpha_6\right){}^8 \left(3 \alpha_4+\alpha_6\right){}^4 \left(3 \alpha_4+2 \alpha_6\right){}^4$\\
  \hline
\end{tabular}
\end{table}

Now we are going to establish Theorem \ref{thma1} for these strata in $E_n$. Recall that for any stratum $D$ and $\beta \in \mathcal{R}\setminus \mathcal{R}_D$ we define the root system $\mathcal{R}_{D, \beta}=\langle \mathcal{R}_{D}, \beta \rangle \cap \mathcal{R}$ which has the decomposition \eqref{decomposition} and that we have $\beta \in \mathcal{R}_{D, \beta}^{(0)}$. The approach to finding $\mathcal{R}_{D, \beta}^{(0)}$ is as follows. 

One can assume by applying group action that a generic vector in the stratum $D\cap \Pi_\beta$ is in the fundamental domain of $E_n$ (cf. Proposition \ref{propwaction}), hence the root systems $\mathcal{R}_{D, \beta}$ can be identified with a subgraph inside the Coxeter graph of $E_n$. Then we consider further action of the Coxeter group $W_{D, \beta}$ generated by   orthogonal reflections about the mirrors $\Pi_\alpha$ with $\alpha \in \mathcal{R}_{D, \beta}$. This group fixes the stratum $D\cap \Pi_\beta$ and we can map a generic vector in the space 
%$D\cap\langle {\mathcal R}_{D, \beta}\rangle$ 
$D$ into the fundamental domain for the group $W_{D, \beta}$. Thus we can assume that the stratum $D$ is given by vanishing some of the simple roots of the root system $\mathcal{R}_{D, \beta}$ which is a subset of the set of simple roots $\Delta$.
%and hence stratum $D$ are both in the fundamental domain of $E_n$ (cf. Proposition \ref{propwaction}), hence 
Overall, the root systems $\mathcal{R}_{D} \subset \mathcal{R}_{D, \beta}$ can be identified with embedded one into another subgraphs inside the Coxeter graph for $E_n$.
 We compute the size $|\mathcal{R}_{D, \beta}|$ of the root system $\mathcal{R}_{D, \beta}$ using Mathematica. In most cases considerations of subgraphs of the Coxeter graph of $E_n$ allow to determine $\mathcal{R}_{D, \beta}$ such that 
$\mathcal{R}_{D, \beta}\supset \mathcal{R}_{D}$
from its size uniquely.  
%(see also \cite{oshima} for classification of all subsets of a root system which are irreducible root systems). 

After we find the type of the root system $\mathcal{R}_{D, \beta}$
%\supset \mathcal{R}_{D}$ 
%we consider embedding of the root systems $\mathcal{R}_D \subset \mathcal{R}_{D, \beta}$. 
we  identify its irreducible component $\mathcal{R}_{D, \beta}^{(0)}$ with the help of  Lemma \ref{irredulemma1} and relations \eqref{irred1}, \eqref{irred2}. 
Finally, we check that its Coxeter number is equal to the degree of the corresponding factor $\beta$ in the determinant of $\eta_D$.
We give these results in Tables \ref{tabled=3E8}, \ref{tabled=2E8} for the root system $\mathcal{R}=E_8$ and in Table \ref{tabled=2E7} for the root system $\mathcal{R}=E_7$.

  The cases when the knowledge of $|\mathcal{R}_{D, \beta}|$ does not immediately lead to the type of $\mathcal{R}_{D, \beta}$ are as follows:
\begin{enumerate}[(i)]
\item $\mathcal{R}=E_8$, $\operatorname{dim}D=3$, $|\mathcal{R}_{D, \beta}|=42$ in which case $\mathcal{R}_{D, \beta}= A_6$ or $\mathcal{R}_{D, \beta}= D_5 \times A_1$,
\item $\mathcal{R}=E_8$, $\operatorname{dim}D=3$, $|\mathcal{R}_{D, \beta}|=24$ in which case $\mathcal{R}_{D, \beta}=A_4 \times A_1^2$ or $\mathcal{R}_{D, \beta}=A_3^2$,
%\item $\mathcal{R}=E_8$, $\operatorname{dim}D=2$, $|\mathcal{R}_{D, \beta}|=44$ in which case $\mathcal{R}_{D, \beta}=A_6 \times A_1$ or $\mathcal{R}_{D, \beta}= D_5 \times A_1^2$,
\item $\mathcal{R}=E_7$, $\operatorname{dim}D=2$, $|\mathcal{R}_{D, \beta}|=42$ in which case $\mathcal{R}_{D, \beta}= A_6$ or $\mathcal{R}_{D, \beta}= D_5 \times A_1$.
\end{enumerate}

Let us consider these remaining cases in detail. 

\begin{enumerate}[(i)]
\item 
Considerations of possible embeddings of Coxeter graphs for ${\mathcal R}_D$ inside the Coxeter graphs for $D_5\times A_1$ and $A_6$ inside $E_8$ 
% Considerations of Coxeter graphs and their subgraphs for $D_5$ and $A_6$ 
allow to determine $\mathcal{R}_{D, \beta}$ in all the cases except for when $\mathcal{R}_D= A_4 \times A_1$. Let us consider this case.
 
  Let us consider firstly $\beta \in \mathcal{R}$ such that $\left. \beta \right|_D= \left.  \alpha_6 \right|_D$. Then it is immediate from the Coxeter graph of $E_8$ that $\mathcal{R}_{D, \beta}=A_6$. 

 Let us now consider $\left. \beta \right|_D= \left. \alpha_2 + 2 \alpha_6 + \alpha_8 \right|_D$. Suppose that $\mathcal{R}_{D, \beta}= D_5 \times A_1$. Note that $A_4 \times A_1$ is not a subsystem of $D_5$. Therefore it has to be that $\beta \in D_5$ and ${\mathcal R}_{D, \beta}^{(0)} = \langle A_4, \beta \rangle \cap \mathcal{R}= D_5$. One can choose $$\beta = \alpha_1+ \alpha_2+\alpha_8  + 2( \alpha_3 +\alpha_6 + \alpha_7)  + 3 (\alpha_4 + \alpha_5) \in \mathcal{R}$$ so that $\left. \beta \right|_D$ has the required form. Then one can check by Mathematica that 
$$
|\langle A_4, \beta \rangle \cap \mathcal{R}|= 30 \neq 40 = |D_5|. 
$$
This contradiction implies that $\mathcal{R}_{D, \beta}=A_6$.
 
  The case $\left. \beta \right|_D = \left. 2 \alpha_2 + 2 \alpha_6 + \alpha_8\right|_D$ is similar. One can choose $$\beta =  \alpha_1+ \alpha_7 + \alpha_8 +2( \alpha_2 +  \alpha_3 + \alpha_5 + \alpha_6)+ 3 \alpha_4 \in \mathcal{R}$$ so that $\left. \beta \right|_D$ has the required form. 
  
  Now let us consider the case when $\left. \beta \right|_D = \left. \alpha_2 \right|_D$. It is immediate from the Coxeter graph of $E_8$ that $\mathcal{R}_{D, \beta}= D_5 \times A_1$.
  
 Consider the case when $\left. \beta \right|_D=\left. 2 \alpha_2 + 3 \alpha_6 + \alpha_8\right|_D$. One can choose $$\beta= \alpha_1+\alpha_8+ 2 (\alpha_2 + \alpha_3 +\alpha_7) + 3(\alpha_4  + \alpha_5 + \alpha_6) \in \mathcal{R}$$ so that $\left. \beta \right|_D$ has the required form. One can check by Mathematica that 
$$
|\langle A_4, \beta \rangle \cap \mathcal{R}|= 40=|D_5|.
$$
 Note that $\pm \alpha_7 \in \mathcal{R}_{D, \beta}$ is orthogonal to the vector space $\langle A_4, \beta\rangle$. Since $|\mathcal{R}_{D, \beta}|=42$ it follows that the root system $\mathcal{R}_{D, \beta}$ is reducible which implies that $\mathcal{R}_{D, \beta}=D_5 \times A_1$. 
   
   The case $\left. \beta \right|_D =\left.   \alpha_2 +  \alpha_6 + \alpha_8 \right|_D$ is similar. One can choose $\beta = \sum_{i=1}^8 \alpha_i \in \mathcal{R}$ so that $\left. \beta \right|_D$ has the required form. 
 
\item 
Considerations of possible embeddings of Coxeter graphs for ${\mathcal R}_D$ inside the Coxeter graphs for $A_4\times A_1^2$ and $A_3^2$ inside $E_8$
%Considerations of Coxeter graphs and their subgraphs for $A_3$ and $A_4$ 
 allow to determine $\mathcal{R}_{D, \beta}$ in all the cases except for when $\mathcal{R}_D= A_3 \times A_1^2$. Let us consider this case.

 Consider firstly $\beta \in \mathcal{R}$ such that $\left. \beta \right|_D=\left.  \alpha_4 + \alpha_8\right|_D$. Suppose that $\mathcal{R}_{D, \beta}= A_4 \times A_1^2$. 
%Note that $A_3 \times A_1$ is not a subsystem of $A_4$. 
Then it has to be that $\beta \in A_4$ and $\langle A_3, \beta \rangle \cap \mathcal{R}= A_4$. One can choose $\beta= \sum_{\substack{i=2 \\ i \neq 3}}^8\alpha_i $ so that $\left. \beta \right|_D$ has the required form. Then one can check by Mathematica that $|\langle A_3, \beta \rangle \cap \mathcal{R}|=12 \neq 20 =|A_4|$. This contradiction implies that $\mathcal{R}_{D, \beta}=A_3^2$.  

The case $\left. \beta \right|_D= \left. 2 \alpha_1 + 5 \alpha_4 + \alpha_8\right|_D$ is similar. In this case one can choose $$\beta= \alpha_8 +  2( \alpha_1 +\alpha_7)+ 3 (\alpha_2 + \alpha_3 + \alpha_6) + 4 \alpha_5 +5\alpha_4\in \mathcal{R}$$ so that $\left. \beta \right|_D$ has the required form.

 Now let us consider the case when $\left. \beta \right|_D= \left.  \alpha_8 \right|_D$. Then it is immediate from the Coxeter graph of $E_8$ that $\mathcal{R}_{D, \beta}=A_4 \times A_1^2$.
 
  Consider now the case when $\left. \beta \right|_D=\left. 2 \alpha_1 +6 \alpha_4+ \alpha_8 \right|_D$. One can choose $$\beta = \alpha_8 + 2 (\alpha_1 + \alpha_7) + 3 \alpha_2 + 4 (\alpha_3 + \alpha_6) +5 \alpha_5 + 6 \alpha_4 \in \mathcal{R}$$ so that $\left. \beta \right|_D$ has the required form. Suppose that $\mathcal{R}_{D, \beta}=A_3^2$. Then it has to be that $\langle A_1^2, \beta \rangle \cap \mathcal{R}= A_3$. One can check by Mathematica that $|\langle A_1^2,\beta \rangle \cap \mathcal{R}|= 6\neq 12=|A_3|$. This contradiction implies that $\mathcal{R}_{D, \beta}=A_4 \times A_1^2$. 
  
  The cases 
$\left. \beta \right|_D=\left.  2 \alpha_1 + 4 \alpha_4 + \alpha_8 \right|_D$  and $\left. \beta \right|_D=2 \alpha_4 +\alpha_8$ are similar. One can choose $$\beta= \alpha_7 + \alpha_8 + 2(\alpha_1 + \alpha_2 + \alpha_6) + 3(\alpha_3 + \alpha_5) + 4 \alpha_4 \in \mathcal{R}$$ and $$ \beta= \alpha_7 + \alpha_8+\alpha_2 + \alpha_3 + 2(\alpha_4 + \alpha_5 + \alpha_6) \in \mathcal{R}$$ respectively, so that $\left. \beta \right|_D$ have the required forms. 

%\item Considerations of Coxeter graphs and their subgraphs for $D_5$ and $A_6$ allow to determine allows to determine $\mathcal{R}_{D, \beta}$ in all the cases except for when $\mathcal{R}_D= A_4 \times A_1^2$. Let us consider firstly the case when $\left. \beta \right|_D = \left. \alpha_1 + 4 \alpha_4 \right|_D$. Suppose that $\mathcal{R}_{D, \beta}=D_5 \times A_1^2$. Then it has to be that $\beta  \in D_5$ and $\langle A_4, \beta \rangle \cap \mathcal{R}=D_5$. One can choose $$\beta = \alpha_1 +2 (\alpha_2 + \alpha_3 + \alpha_6) + 4\alpha_4 + 3\alpha_5 + \alpha_7 \in \mathcal{R}$$ so that $\left. \beta \right|_D$ has the required form. Then one can check by Mathematica that $|\langle A_4, \beta \rangle \cap \mathcal{R}|=30 \neq 40= |D_5|$. This contradiction implies that $\mathcal{R}_{D, \beta}= A_6 \times A_1$.The case $\left. \beta \right|_D=\alpha_1 + \alpha_4$ is similar. One can choose $\beta = \sum_{ \substack{i=1 \\ i \neq 2, 8}}^8\alpha_i \in \mathcal{R}$ so that $\left. \beta \right|_D$ has the required form. 
 
\item
Considerations of possible embeddings of Coxeter graphs for ${\mathcal R}_D$ inside the Coxeter graphs for $D_5\times A_1$ and $A_6$  inside $E_7$
% Considerations of Coxeter graphs and their subgraphs for $D_5$ and $A_6$ 
allow to determine $\mathcal{R}_{D, \beta}$ in all the cases except for when $\mathcal{R}_D= A_4 \times A_1$. Let us consider this case.

 Consider firstly $\beta \in \mathcal{R}$ such that $\left. \beta \right|_D=\left. \alpha_5 \right|_D$. Then it is immediate from the Coxeter graph of $E_7$ that $\mathcal{R}_{D, \beta}=D_5 \times A_1$. Let us now consider the case when $\left. \beta \right|_D= \left. 2 \alpha_5 +\alpha_6\right|_D$. Suppose that $\mathcal{R}_{D, \beta}=D_5 \times A_1$. Then it has to be that $\beta \in D_5$ and $\langle A_4, \beta \rangle \cap \mathcal{R}=D_5$. One can choose $$\beta=\alpha_2+\alpha_3+\alpha_6 +2 (\alpha_4+ \alpha_5)\in \mathcal{R}$$ so that $\left. \beta \right|_D$ has the required form. One can check by Mathematica that $|\langle A_4, \beta \rangle \cap \mathcal{R}|=30 \neq 40=|D_5|$. This contradiction implies that $\mathcal{R}_{D, \beta}=A_6$.
\end{enumerate}

%and deduce the type of the root system $\mathcal{R}_{D, \beta}^{(0)}$ with the help of Lemma \ref{irredulemma1} and the classification of subsystems of root systems \cite{oshima}. We thus show that the statement of Theorem \ref{thma1} is true. We give these results in Tables \ref{tabled=3E8}, \ref{tabled=2E8} for the root system $\mathcal{R}=E_8$ and in Table \ref{tabled=2E7} for the root system $\mathcal{R}=E_7$.

%\vspace{0.2cm}

As a direct corollary of Tables \ref{tabled=3E8}, \ref{tabled=2E8} we get the following statement.
 
 \begin{theorem}
Let $D$ be any three-dimensional or two-dimensional stratum for $\mathcal{R}=E_8$. Then the statement of Theorem \ref{thma1} is true. 
\end{theorem}

We get the following statement as a direct corollary of Table \ref{tabled=2E7}.

\begin{theorem}
Let $D$ be any two-dimensional stratum for $\mathcal{R}=E_7$. Then the statement of Theorem \ref{thma1} is true. 
\end{theorem}

%\vspace{-5mm}

%\vspace{0.2cm}

\begin{table}[htp]
\centering
\caption{$\mathcal{R}_{D, \beta}$, $\operatorname{dim}D=3$, $\mathcal{R}=E_8$}
\label{tabled=3E8}
\hspace*{-0.7cm}
\footnotesize
\begin{tabular}{|c |c | c| c| c|c|}
 \hline
\rule{0pt}{3ex} $\mathcal{R}_D$, $S$ & $\left. \beta\right|_D$ &$|\mathcal{R}_{D, \beta}|$ & $\mathcal{R}_{D,\beta}$ &$\mathcal{R}_{D, \beta}^{(0)}$ &$h(\mathcal{R}_{D, \beta}^{(0)})$  \\
 \hline
 \multirow{4}*{\makecell{$A_5$, \\ $\{4, 5, 6, 7, 8\}$}} 
 & $\alpha_1+\alpha_2+2 \alpha_3$ & $72$& $E_6$ &$E_6$ &$12$\\ \cline{2-6} 
  &$\alpha_2+\alpha_3$, $\alpha_1+\alpha_2+\alpha_3$, $\alpha_1+2 \alpha_2+2 \alpha_3$ & $60$ & $D_6$ &$D_6$&$10$ \\ \cline{2-6}
   &\makecell{ $\alpha_2$, $\alpha_3$, $2 \alpha_1+3 \alpha_2+4 \alpha_3$,\\ $2 \alpha_1+2 \alpha_2+3 \alpha_3$, $\alpha_1+2 \alpha_2+3 \alpha_3$, $\alpha_1+\alpha_3$} &$42$ & $A_6$ & $A_6$&$7$ \\ \cline{2-6}
  & $\alpha_1$, $2 \alpha_1+3 \alpha_2+3 \alpha_3$, $\alpha_1+3 \alpha_2+3 \alpha_3$ & $32$ &$A_5 \times A_1$ &$A_1$&$2$\\ 
  \midrule
  \multirow{3}*{\makecell{$D_5$, \\$\{1, 2, 3, 4, 5\}$}} &$\alpha_6$, $\alpha_6+\alpha_7+\alpha_8$, $3 \alpha_6+2 \alpha_7+\alpha_8$, $\alpha_6+\alpha_7$  & $72$ & $E_6$ & $E_6$&$12$ \\ \cline{2-6}
  & $2 \alpha_6+\alpha_7+\alpha_8$, $2 \alpha_6+2 \alpha_7+\alpha_8$, $2 \alpha_6+\alpha_7$ & $60$ & $D_6$& $D_6$&$10$ \\  \cline{2-6}
  &\makecell{$\alpha_7$, $\alpha_8$, $4 \alpha_6+3 \alpha_7+2 \alpha_8$, $\alpha_7+\alpha_8$, \\$4 \alpha_6+3 \alpha_7+\alpha_8$, $4 \alpha_6+2 \alpha_7+\alpha_8$ } & $42$ & $D_5 \times A_1$ &$A_1$ &$2$\\
  \midrule
\multirow{3}*{\makecell{$D_4 \times A_1$,\\ $\{2, 3, 4, 5, 7\}$}}
 & $\alpha_6$, $\alpha_1+2 \alpha_6+\alpha_8$, $\alpha_1+\alpha_6$ & $60$ & $D_6$ & $D_6$&$10$ \\ \cline{2-6}
  &\makecell{$\alpha_1$, $\alpha_1+3 \alpha_6+\alpha_8$, $\alpha_1+\alpha_6+\alpha_8$,  $\alpha_6+\alpha_8$, \\ $2 \alpha_1+3 \alpha_6+\alpha_8$,  $\alpha_1+2 \alpha_6$ } & $42$ & $D_5 \times A_1$ &$D_5$ &$8$ \\ \cline{2-6}
   &\makecell{$\alpha_8$, $2 \alpha_1+4 \alpha_6+\alpha_8$,\\ $2 \alpha_1+2 \alpha_6+\alpha_8$, $2 \alpha_6+\alpha_8$} & $30$ & $D_4 \times A_2$ &$A_2$ &$3$ \\
 \midrule 
 \multirow{5}*{\makecell{$A_4 \times A_1$,\\$\{1, 3, 4, 5, 7\}$}}
  & $\alpha_2+\alpha_6$ &$72$ & $E_6$ & $E_6$ &$12$\\  \cline{2-6}
  & $\alpha_6$, $\alpha_2+2 \alpha_6+\alpha_8$, $2 \alpha_2+2 \alpha_6+\alpha_8$ & $42$ & $A_6$ & $A_6$ &$7$ \\ \cline{2-6}
 & $\alpha_2$, $2 \alpha_2+3 \alpha_6+\alpha_8$, $\alpha_2+\alpha_6+\alpha_8$ & $42$ & $D_5 \times A_1$ & $D_5$ &$8$ \\ \cline{2-6}
 & $\alpha_2+2 \alpha_6$, $\alpha_6+\alpha_8$, $3 \alpha_2+3 \alpha_6+\alpha_8$ & $32$ & $A_5 \times A_1$ & $A_5$  &$6$ \\ \cline{2-6}
 & $\alpha_8$, $2 \alpha_2+\alpha_6$, $3 \alpha_2+4 \alpha_6+\alpha_8$ & $26$ & $A_4 \times A_2$ & $A_2$  &$3$\\ \cline{2-6}
 &\makecell{ $3 \alpha_2+4 \alpha_6+2 \alpha_8$, $2 \alpha_2+\alpha_6+\alpha_8$, \\$\alpha_2+3 \alpha_6+\alpha_8$} & $24$ & $A_4 \times A_1^2$ & $A_1$ &$2$ \\
 \midrule
  \multirow{6}*{\makecell{$A_3 \times A_2$, \\$\{2, 3, 4, 6, 7\}$}}
 & $\alpha_5$ & $60$ & $D_6$ & $D_6$&$10$ \\ \cline{2-6}
 &\makecell{ $\alpha_1+2 \alpha_5+\alpha_8$, $\alpha_1+3 \alpha_5+\alpha_8$, \\$\alpha_1+\alpha_5$, $\alpha_1+2 \alpha_5$} & $42$ & $A_6$ & $A_6$&$7$  \\ \cline{2-6}
 & $\alpha_5+\alpha_8$, $2 \alpha_1+4 \alpha_5+\alpha_8$  & $30$ & $D_4 \times A_2$ & $D_4$ &$6$ \\ \cline{2-6}
 &$\alpha_1$, $\alpha_1+4 \alpha_5+\alpha_8$, $\alpha_1+\alpha_5+\alpha_8$, $\alpha_1+3 \alpha_5$ & $26$ & $A_4 \times A_2$ & $A_4$&$5$ \\ \cline{2-6}
 & $\alpha_8$, $2 \alpha_1+5 \alpha_5+\alpha_8$, $2 \alpha_1+3 \alpha_5+\alpha_8$, $2 \alpha_5+\alpha_8$ & $24$ & $A_3^2$ &$A_3$ &$4$  \\ \cline{2-6}
 & $2 \alpha_1+5 \alpha_5+2 \alpha_8$, $2 \alpha_1+3 \alpha_5$ & $20$ & $A_3 \times A_2 \times A_1$ & $A_1$ &$2$ \\
 \midrule
  \multirow{6}*{\makecell{$A_3 \times A_1^2$, \\$\{2, 3, 5, 6, 7\}$}}
  & $\alpha_4$ & $60$ & $D_6$ & $D_6$ &$10$   \\ \cline{2-6}
  & $\alpha_1+3 \alpha_4+\alpha_8$, $\alpha_1+2 \alpha_4$ & $42$ & $D_5 \times A_1$ & $D_5$ &$8$  \\ \cline{2-6}
  & \makecell{$\alpha_1+2 \alpha_4+\alpha_8$, $\alpha_1+4 \alpha_4+\alpha_8$,\\ $\alpha_1+3 \alpha_4$, $\alpha_1+\alpha_4$} & $32$ & $A_5 \times A_1$ & $A_5$   &$6$ \\ \cline{2-6}
  &$ \alpha_4+\alpha_8$, $2 \alpha_1+5 \alpha_4+\alpha_8$ & $24$ &$A_3^2$ & $A_3$ &$4$ \\ \cline{2-6}
   &$\alpha_8$, $2 \alpha_1+6 \alpha_4+\alpha_8$, $2 \alpha_1+4 \alpha_4+\alpha_8$, $2 \alpha_4+\alpha_8$  & $24$ & $A_4 \times A_1^2$ &$A_4$ &$5$ \\ \cline{2-6}
  &$\alpha_1$, $\alpha_1+5 \alpha_4+\alpha_8$, $\alpha_1+\alpha_4+\alpha_8$, $\alpha_1+4 \alpha_4$ & $20$ &$A_3 \times A_2 \times A_1$ & $A_2$ &$3$\\
 \midrule
  \multirow{5}*{\makecell{$A_2^2 \times A_1$, \\ $\{1, 2, 3, 5, 6\}$}}
  & $\alpha_4$ & $72$ & $E_6$ & $E_6$ &$12$  \\ \cline{2-6}
  &$2 \alpha_4+\alpha_7$, $2 \alpha_4+\alpha_7+\alpha_8$, $4 \alpha_4+2 \alpha_7+\alpha_8$  & $32$ & $A_5 \times A_1$ &$A_5$&$6$ \\ \cline{2-6} 
  &\makecell{$\alpha_4+\alpha_7+\alpha_8$,  $5 \alpha_4+2 \alpha_7+\alpha_8$, $3 \alpha_4+\alpha_7$,\\ $3 \alpha_4+\alpha_7+\alpha_8$, $3 \alpha_4+2 \alpha_7+\alpha_8$, $\alpha_4+\alpha_7$} & $26$ & $A_4 \times A_2$ & $A_4$ &$5$\\ \cline{2-6}
  & \makecell{$4 \alpha_4+\alpha_7+\alpha_8$, $6 \alpha_4+2 \alpha_7+\alpha_8$,\\$\alpha_7$, $4 \alpha_4+\alpha_7$, $\alpha_7+\alpha_8$, $2 \alpha_4+2 \alpha_7+\alpha_8$}& $20$ & $A_3 \times A_2 \times A_1$ & $A_3$ &$4$\\ \cline{2-6}
  & $\alpha_8$, $6 \alpha_4+3 \alpha_7+2 \alpha_8$, $6 \alpha_4+3 \alpha_7+\alpha_8$ &$16$ & $A_2^2 \times A_1^2$ & $A_1$ &$2$ \\ 
  \midrule
  \multirow{5}*{\makecell{$A_2 \times A_1^3$, \\ $\{2, 3, 5, 7, 8\}$}} 
  & $\alpha_4+\alpha_6$, $\alpha_1+2 \alpha_4+\alpha_6$, $\alpha_1+3 \alpha_4+2 \alpha_6$ & $42$ & $D_5 \times A_1$ &$D_5$ &$8$  \\ \cline{2-6} 
  & $\alpha_4$ & $30$ &$D_4 \times A_2$ &$D_4$  &$6$ \\ \cline{2-6} 
  & \makecell{$ \alpha_6 $, $2 \alpha_4+\alpha_6$, $\alpha_1+\alpha_4+\alpha_6 $, $\alpha_1+3 \alpha_4+\alpha_6$,\\ $\alpha_1+4 \alpha_4+2 \alpha_6$, $\alpha_1+2 \alpha_4+2 \alpha_6$} & $24$ & $A_4 \times A_1^2$ &$A_4$ &$5$  \\ \cline{2-6} 
  & $2 \alpha_1+5 \alpha_4+3 \alpha_6$, $\alpha_1+4 \alpha_4+3 \alpha_6$, $\alpha_1+\alpha_4$ & $20$ &$A_3 \times A_2 \times A_1$ &$A_3$ &$4$ \\ \cline{2-6} 
  &\makecell{ $\alpha_1$, $2 \alpha_1+4 \alpha_4+3 \alpha_6$, $2 \alpha_1+6 \alpha_4+3 \alpha_6$,\\ $\alpha_1+3 \alpha_4+3 \alpha_6$, $\alpha_1+5 \alpha_4+3 \alpha_6$, $\alpha_1+2 \alpha_4$} &$16$ &$A_2^2 \times A_1^2$ &$A_2$ &$3$\\
  \hline
\end{tabular}
\end{table}

\begin{table}[h]
\centering
\caption{$\mathcal{R}_{D, \beta}$, $\operatorname{dim}D=2$, $\mathcal{R}=E_8$}
\label{tabled=2E8}
\footnotesize
\begin{tabular}{|c |c | c| c| c|c|}
 \hline
\rule{0pt}{3ex} $\mathcal{R}_D$, $S$ & $\left. \beta\right|_D$ &$|\mathcal{R}_{D, \beta}|$ & $\mathcal{R}_{D,\beta}$ &$\mathcal{R}_{D, \beta}^{(0)}$ &$h(\mathcal{R}_{D, \beta}^{(0)})$  \\
 \hline
 \multirow{4}*{\makecell{$A_6$, \\  $\{2, 4, 5, 6, 7, 8\}$}} 
 & $\alpha_1+2 \alpha_3$ &$126$& $E_7$ &$E_7$ &$18$ \\ \cline{2-6} 
 & $\alpha_3$, $\alpha_1+\alpha_3$ & $84$& $D_7$ & $D_7$ &$12$\\ \cline{2-6}
 & $2 \alpha_1+3 \alpha_3$, $\alpha_1+3 \alpha_3$ & $56$&$A_7$ & $A_7$ &$8$ \\ \cline{2-6}
 &$\alpha_1$ & $44$&$A_6 \times A_1$ &$A_1$ &$2$ \\
 \midrule
  \multirow{2}*{\makecell{$D_6$, \\ $\{2, 3, 4, 5, 6, 7\}$}} 
  & $\alpha_1$, $\alpha_1+\alpha_8$ & $126$& $E_7$ & $E_7$ &$18$  \\ \cline{2-6}
  & $\alpha_8$, $2 \alpha_1+\alpha_8$   &$84$&$D_7$ & $D_7$ &$12$\\
  \midrule
  \multirow{2}*{\makecell{$E_6$, \\ $\{1, 2, 3, 4, 5, 6\}$}} 
  & $\alpha_7$, $\alpha_7+\alpha_8$, $2 \alpha_7+\alpha_8$ & $126$&$E_7$ &$E_7$ &$18$ \\ \cline{2-6}
  &$\alpha_8$, $3 \alpha_7+\alpha_8$, $3 \alpha_7+2 \alpha_8$ &$74$& $E_6 \times A_1$ & $A_1$ &$2$\\
  \midrule
   \multirow{4}*{\makecell{$A_5 \times A_1$, \\ $\{1, 3, 4, 5, 6, 8\}$}} 
   & $\alpha_2+\alpha_7$ &$126$&$E_7$ & $E_7$ &$18$ \\ \cline{2-6}
   & $\alpha_7$, $2 \alpha_2+\alpha_7$ &$56$& $A_7$ &$A_7$ &$8$ \\ \cline{2-6}
   & $\alpha_2$ &$74$& $E_6 \times A_1$& $E_6$ &$12$ \\ \cline{2-6}
   & $3 \alpha_2+2 \alpha_7$, $\alpha_2+2 \alpha_7$ &$44$& $A_6 \times A_1$ &$A_6$ &$7$ \\
   \midrule
     \multirow{4}*{\makecell{$D_5 \times A_1$, \\ $\{2, 3, 4, 5, 6, 8\}$}} 
     & $\alpha_1+\alpha_7$ & $126$& $E_7$ & $E_7$ &$18$  \\ \cline{2-6}
     & $\alpha_7$ & $84$& $D_7$ & $D_7$ &$12$ \\ \cline{2-6}
     &$\alpha_1$, $\alpha_1+2 \alpha_7$ &$74$& $E_6 \times A_1$ & $E_6$ &$12$ \\ \cline{2-6}
     & $2 \alpha_1+3 \alpha_7$, $2 \alpha_1+\alpha_7$ & $46$& $D_5 \times A_2$ & $A_2$ &$3$ \\
     \midrule
   \multirow{4}*{\makecell{$A_4 \times A_2$, \\ $\{1, 3, 4, 5, 7, 8\}$}} 
   & $\alpha_2+\alpha_6$ & $126$& $E_7$ & $E_7$ &$18$\\ \cline{2-6}
   &$\alpha_6$, $\alpha_2+2 \alpha_6$ & $56$& $A_7$ & $A_7$ &$8$\\ \cline{2-6}
   & $\alpha_2 $, $2 \alpha_2+3 \alpha_6$  &$46$& $D_5 \times A_2$ & $D_5$ &$8$\\ \cline{2-6}
   & $2 \alpha_2+\alpha_6$, $3 \alpha_2+4 \alpha_6$ &$32$&$ A_4 \times A_3$ & $A_3$ &$4$\\ \cline{2-6}
   & $\alpha_2+3 \alpha_6$ &$28$& $A_4 \times A_2 \times A_1$ & $A_1$&$2$ \\
   \midrule
   \multirow{2}*{\makecell{$D_4 \times A_2$, \\ $\{2, 3, 4, 5, 7, 8\}$}} 
 &$\alpha_6$, $\alpha_1+2 \alpha_6$, $\alpha_1+\alpha_6$ &$84$& $D_7$ & $D_7$ &$12$ \\ \cline{2-6}
 &$\alpha_1$, $2 \alpha_1+3 \alpha_6$, $\alpha_1+3 \alpha_6$ &$46$& $D_5 \times A_2$ &$D_5$ &$8$ \\
 \midrule
  \multirow{5}*{\makecell{$A_4 \times A_1^2$, \\ $\{2, 3, 5, 6, 7, 8\}$}} 
 & $\alpha_4$ &$84$& $D_7$ & $D_7$ &$12$ \\ \cline{2-6}
 & $\alpha_1+3 \alpha_4$, $\alpha_1+2 \alpha_4$ & $74$& $E_6 \times A_1$ & $E_6$ &$12$\\ \cline{2-6}
 & $ \alpha_1+4 \alpha_4$, $\alpha_1+\alpha_4$ & $44$&$A_6 \times A_1$ & $A_6$&$7$ \\ \cline{2-6}
 & $2 \alpha_1+5 \alpha_4$ &$32$& $A_4 \times A_3$ & $A_3$ &$4$\\ \cline{2-6}
 &$\alpha_1$, $\alpha_1+5 \alpha_4$ & $28$&$A_4 \times A_2 \times A_1$ &$ A_2$ &$3$ \\
 \midrule
  \multirow{3}*{\makecell{$A_3^2$, \\ $\{2, 3, 4, 6, 7, 8\}$}} 
 & $\alpha_5$, $\alpha_1+2 \alpha_5$ & $84$&$D_7$ & $D_7$ &$12$ \\ \cline{2-6}
 & $\alpha_1+3 \alpha_5$, $\alpha_1+\alpha_5$ &$56$& $A_7$ & $A_7$ &$8$ \\ \cline{2-6}
 &$\alpha_1$, $2 \alpha_1+5 \alpha_5$, $\alpha_1+4 \alpha_5 $, $2 \alpha_1+3 \alpha_5$ &$32$&$A_4 \times A_3$ &$A_4$ &$5$ \\
 \midrule
 \multirow{5}*{\makecell{$A_3 \times A_2 \times A_1$, \\ $\{1, 2, 4, 6, 7, 8\}$}} 
 & $\alpha_3+\alpha_5$ & $126$& $E_7$ & $E_7$ &$18$ \\ \cline{2-6}
 &$\alpha_5$, $2 \alpha_3+3 \alpha_5$ &$44$& $A_6 \times A_1$ &$ A_6$ &$7$  \\ \cline{2-6}
 & $\alpha_3+2 \alpha_5$ &$46$& $D_5 \times A_2$ & $D_5$ &$8$  \\ \cline{2-6}
 &$\alpha_3$, $3 \alpha_3+4 \alpha_5$ &$32$& $A_4 \times A_3$ &$A_4$ &$5$\\ \cline{2-6}
  & $2 \alpha_3+\alpha_5$, $4 \alpha_3+5 \alpha_5$ &$28$& $A_4 \times A_2 \times A_1$ &$A_4$ &$5$ \\
 \midrule
 \multirow{3}*{\makecell{$A_2^2 \times A_1^2 $, \\ $\{1, 2, 3, 5, 6, 8\}$}} 
 &$\alpha_4$, $2 \alpha_4+\alpha_7$ &$74$& $E_6 \times A_1$ & $E_6$ &$12$ \\ \cline{2-6}
 & $3 \alpha_4+\alpha_7$, $\alpha_4+\alpha_7$ &$46$& $D_5 \times A_2$ & $D_5$ &$8$ \\ \cline{2-6}
 &$\alpha_7$, $4 \alpha_4+\alpha_7$, $5 \alpha_4+2 \alpha_7$, $3 \alpha_4+2 \alpha_7$ & $28$& $A_4 \times A_2 \times A_1$ & $A_4$ &$5$ \\
 \hline
\end{tabular}
\end{table}

\begin{table}[h]
\centering
\caption{$\mathcal{R}_{D, \beta}$, $\operatorname{dim}D=2$, $\mathcal{R}=E_7$}
\label{tabled=2E7}
\footnotesize
\begin{tabular}{|c |c | c| c| c|c|}
 \hline
\rule{0pt}{3ex} $\mathcal{R}_D$, $S$ & $\left. \beta\right|_D$ &$|\mathcal{R}_{D, \beta}|$& $\mathcal{R}_{D, \beta}$ &$\mathcal{R}_{D, \beta}^{(0)}$ &$h(\mathcal{R}_{D, \beta}^{(0)})$ \\
 \hline
   \multirow{2}*{\makecell{$A_5$, \\ $\{2, 4, 5, 6, 7\}$}} 
   & $\alpha_3$, $\alpha_1+\alpha_3$, $\alpha_1+2 \alpha_3$ &$60$& $ D_6$ & $D_6$ &$10$ \\ \cline{2-6}
   &$\alpha_1$, $2 \alpha_1+3 \alpha_3$, $\alpha_1+3 \alpha_3$ &$32$& $A_5 \times A_1$ & $A_1$ &$2$\\
   \midrule
   \multirow{3}*{\makecell{$A^{'}_5$, \\ $\{3, 4, 5, 6, 7\}$}} 
   & $\alpha_1+\alpha_2$ &$72$& $E_6$ &$E_6$ &$12$ \\ \cline{2-6}
   & $\alpha_2 $ & $60$& $D_6$ & $D_6$ &$10$ \\ \cline{2-6}
   & $\alpha_1$, $\alpha_1+2 \alpha_2$ &$42$& $A_6$ & $A_6$ &$7$ \\
   \midrule
  \multirow{3}*{\makecell{$D_5$, \\ $\{1, 2, 3, 4, 5\}$}} 
 & $\alpha_6$, $\alpha_6+\alpha_7$ &$72$& $E_6$ & $E_6$ &$12$\\ \cline{2-6}
 & $2 \alpha_6+\alpha_7$ &$60$& $D_6$ & $D_6$ &$10$ \\ \cline{2-6}
 & $\alpha_7$ &$42$& $D_5 \times A_1$ & $A_1$ &$2$ \\
 \midrule
  \multirow{5}*{\makecell{$A_4 \times A_1$, \\ $\{1, 2, 3, 4, 7\}$}} 
  & $\alpha_5+\alpha_6$ &$72$& $E_6$ & $E_6$ &$12$ \\ \cline{2-6}
  & $2 \alpha_5+\alpha_6$ &$42$& $A_6$ & $A_6$ &$7$ \\ \cline{2-6}
  & $\alpha_5$ &$42$& $D_5 \times A_1$ & $D_5$ &$8$ \\ \cline{2-6}
  & $3 \alpha_5+2 \alpha_6$ &$32$& $A_5 \times A_1$ & $A_5$ &$6$ \\ \cline{2-6}
  & $\alpha_6$ & $26$&$A_4 \times A_2$ & $A_2$ &$3$ \\
  \midrule
  \multirow{2}*{\makecell{$D_4 \times A_1$, \\ $\{2, 3, 4, 5, 7\}$}} 
  & $\alpha_6$, $\alpha_1+\alpha_6$ &$60$& $D_6$ & $D_6$ &$10$ \\ \cline{2-6}
  & $\alpha_1$, $\alpha_1+2 \alpha_6$ & $42$&$D_5 \times A_1$ & $D_5$ &$8$ \\
  \midrule
  \multirow{4}*{\makecell{$A_3 \times A_2$, \\ $\{1, 3, 5, 6, 7\}$}} 
  & $\alpha_2+2 \alpha_4$ &$60$& $D_6$ & $D_6$ &$10$ \\ \cline{2-6}
  & $\alpha_4$, $\alpha_2+\alpha_4$ & $42$&$A_6$ &$A_6$ &$7$ \\ \cline{2-6}
  & $2 \alpha_2+3 \alpha_4$, $\alpha_2+3 \alpha_4$ &$26$& $A_4 \times A_2$ & $A_4$ &$5$ \\ \cline{2-6}
  & $\alpha_2$ & $20$&$A_3 \times A_2 \times A_1$ & $A_1$ &$2$ \\
  \midrule
  \multirow{4}*{\makecell{$A_3 \times A_1^2$, \\ $\{1, 2, 4, 5, 7\}$}} 
  & $\alpha_3+\alpha_6$ &$60$& $D_6$ & $D_6$ &$10$ \\ \cline{2-6}
  & $\alpha_3$ &$42$& $D_5 \times A_1$ & $D_5$ &$8$  \\ \cline{2-6}
  & $\alpha_6$, $2 \alpha_3+\alpha_6$  &$32$& $A_5 \times A_1$ & $A_5$ &$6$ \\ \cline{2-6}
  & $\alpha_3+2 \alpha_6$, $3 \alpha_3+2 \alpha_6$ &$20$& $A_3 \times A_2 \times A_1$ & $A_2$ &$3$ \\
  \midrule
   \multirow{4}*{\makecell{$A_2^2 \times A_1$, \\ $\{1, 2, 4, 6, 7\}$}}
   & $\alpha_3+\alpha_5$ &$72$& $E_6$ & $E_6$ &$12$ \\ \cline{2-6}
   & $\alpha_5$ &$32$& $A_5 \times A_1$ & $A_5$ &$6$ \\ \cline{2-6}
   & $\alpha_3$, $\alpha_3+2 \alpha_5$ &$26$& $A_4 \times A_2$ & $A_4$ &$5$ \\ \cline{2-6}
   & $2 \alpha_3+\alpha_5$, $2 \alpha_3+3 \alpha_5$ &$20$& $A_3 \times A_2 \times A_1$ & $A_3$ &$4$ \\
   \midrule
   \multirow{2}*{\makecell{$A_2 \times A_1^3$, \\ $\{1, 2, 3, 5, 7\}$}}
   & $\alpha_4$, $\alpha_4+\alpha_6$, $2 \alpha_4+\alpha_6$ &$42$& $D_5 \times A_1$ & $D_5$ &$8$  \\ \cline{2-6}
   &  $\alpha_6$, $3 \alpha_4+\alpha_6$, $3 \alpha_4+2 \alpha_6$ &$20$& $A_3 \times A_2 \times A_1$ &$ A_3$ &$4$ \\
   \hline
\end{tabular}
\end{table}

\clearpage

\section{Concluding remarks}

The Saito metric is a remarkable flat metric on the space of orbits of a finite Coxeter group. As we showed its restriction to the discriminant strata also has the interesting property of the factorisation of the determinant. The proof we gave depends on the Coxeter group: while Landau--Ginzburg superpotential approach was applied to the classical series different arguments had to be made for the exceptional groups including computer calculations in a few cases. It would be desirable to find a uniform proof of the main result, Theorem \ref{MMtheorem}. 
Studying further differential-geometric properties of the restricted Saito metric $\eta_D$ may be of interest too.

It would also be interesting to investigate generalisations to the discriminant strata of the extended affine Weyl groups orbit spaces, the corresponding Frobenius manifolds were considered in \cite{DubZh, DZZ, DZZS}. Furthermore, Frobenius structures for the orbit spaces of complex reflection groups were discovered in \cite{KaMaSe} (see also \cite{KaMaSe2}, \cite{DubA}). In this case one does not have the full structure of Frobenius manifold in general but rather the weaker Saito structure without metric \cite{Sab} (see also \cite{KMS}, \cite{AL} for further studies of the complex reflection groups case). Note that by Corollary \ref{cordet} our result on the determinant of the metric $\eta_D$ is reformulated in terms of the determinant of the operator of multiplication by the Euler vector field $E_D$ on the stratum $D$, which is a well defined operator as we show that $D$ is a natural submanifold. Therefore one may try to weaken the Coxeter groups settings where metrics $g$ and $\eta$ both exist and consider Saito structures without metric on the orbit spaces of complex reflection groups. 

Finally, T. Douvropoulos  informed us about his recent conjecture  on the freeness of a class of multiarrangements related to considerations in this paper \cite{dou}. These multiarrangements are restricted Coxeter arrangements ${\mathcal A}_D$ with the multiplicities coming from the multiplicities of factors in Theorem \ref{MMtheorem} (see \cite{OT} and \cite{ziegler} for the theory of free  (multi)arrangements). In the special cases conjecture reduces to the freeness of Coxeter multiarrangements \cite{Ter}, \cite{yoshinaga} and the freeness of restricted Coxeter arrangements with multiplicities one \cite{OTart}. It would be interesting to study this conjecture by making use of analysis of this paper of the discriminant strata. We hope to do this elsewhere.

%ness, and \cite{Ter}, \cite{Yosh}, \cite{AW} for the treatment of the Coxeter case with multiplicities).

\end{document}